\documentclass[11pt]{amsart}
\usepackage{amsfonts, amsbsy}
\usepackage{ dsfont}
\usepackage{amssymb, mathabx, mathtools}
\usepackage{xcolor}
\usepackage[colorlinks,citecolor=blue,urlcolor=blue]{hyperref}
\usepackage[english]{babel}
\usepackage{amsmath}
\theoremstyle{remark}
\newcommand{\re}[1]{(\ref{#1})}

\newcommand{\MA} {\mathfrak{A}}

\newcommand{\Sup} {\displaystyle \sup}

\newcommand{\Max} {\displaystyle \max}
\newcommand{\Min} {\displaystyle \min}

\newcommand{\Int }    {\displaystyle \int}
\newcommand{\Sum }    {\displaystyle \sum}

\newcommand{\norm}[1]{\left\Vert#1\right\Vert}
\newcommand{\abs}[1]{\left\vert#1\right\vert}

\selectfont

\theoremstyle{plain}
\newtheorem{assumption}{\textbf{Assumption}}[section]
\newtheorem{theorem}{\textbf{Theorem}}[section]
\newtheorem{lemma}[theorem]{\textbf{Lemma}}

\newtheorem{proposition}[theorem]{\textbf{Proposition}}
\newtheorem{cor}[theorem]{\textbf{Corollary}}
\theoremstyle{definition}

\newtheorem{remark}[theorem]{\textbf{Remark}}

\newtheorem{definition}[theorem]{\textbf{Definition}}

\numberwithin{equation}{section}
\numberwithin{figure}{section}
\newcommand{\bh}{\mathbf{H}}
\newcommand{\bz}{\mathcal{Z}}

\newcommand{\be}{\mathbb{E}}
\newcommand{\bu}{\mathbf{u}}

\newcommand{\bv}{\mathbf{v}}
\newcommand{\bw}{\mathbf{w}}
\newcommand{\bo}{\mathcal{O}}
\newcommand{\bk}{\mathcal{K}}
\newcommand{\bp}{\mathcal{P}}
\newcommand{\buc}{\mathcal{U}}
\newcommand{\bl}{\mathcal{L}}
\newcommand{\bec}{\mathcal{E}}

\newcommand{\bhc}{\mathcal{H}}

\newcommand{\eps}{\varepsilon}

\title{ On the existence and uniqueness of solution to a stochastic  chemotaxis-Navier-stokes model}

\begin{document}
\selectlanguage{english}
\title[On the Stochastic chemotaxis-Navier-Stokes model]{On the existence and uniqueness of  solution to a stochastic  chemotaxis-Navier-stokes model}
\author{E. Hausenblas$^*$, B. Jidjou Moghomye$^*$ and P. A. Razafimandimby$^{**}$}
\dedicatory{\vspace{-10pt}\normalsize{$^{*}$  Department of Mathematics and Information Technology, Montanuniversitaet Leoben, Leoben Franz Josef Strasse 18, 8700 Leoben, Austria \\$^{**}$ School of Mathematical Science, Dublin City University, Collins Avenue Dublin $9$, Ireland  }}


\begin{abstract} In this article, we study a mathematical system which models the dynamic of the collective behaviour of oxygen-driven swimming bacteria in an aquatic fluid flowing in a two dimensional bounded domain under stochastic perturbation. This model can be seen as a  stochastic version of Chemotaxis-Navier-Stokes model. We prove the existence of a unique (probabilistic) strong  solution. In addition, we establish some properties of the strong solution. More precisely, we prove that the unique solution is non-negative and satisfies the mass conservation property and an energy inequality.
\end{abstract}

\date{\today}
\maketitle

\section{Introduction}

The migration of bacteria cells to a higher concentration of a chemical has been observed in biological applications concerning aerobic bacteria. This phenomenon, called chemotaxis, is presumed to have a deep impact on the time evolution of a bacteria population. There are different concepts of chemotaxis depending on the kind of bacteria and the chemical. In the present article, we focus on the mathematical model describing an oxygen-driven bacteria suspension swimming in an incompressible fluid like water which was firstly proposed in \cite{Tuv}. Mainly, the system consists of three coupled partial differential equations. The first equation describes the fluid flow with field velocity $\bu$.  The second equation describes the dynamic of  the oxygen concentration $c$, and the last equation describes the dynamic of the population density $n$ of the bacteria. Now, the coupled model can be written as
\begin{equation}\label{1.2}
\begin{cases}
 d \bu +\left[(\bu\cdot \nabla) \bu +\nabla P-\eta\Delta \bu
 \right]dt=n\nabla\Phi dt\hspace{0.2cm}\text{in}\hspace{0.2cm}[0,T]\times\bo,\vspace{0.2cm}\\
dc + \bu\cdot \nabla cdt=\left[\mu\Delta c -nf(c)\right]dt\hspace{0.2cm}\text{in}\hspace{0.2cm}[0,T]\times\bo,\vspace{0.2cm}\\
d n +\bu\cdot \nabla n dt= \left[\delta\Delta n-\nabla \cdot (n \chi(c)\nabla c)\right]dt\hspace{0.2cm}\text{in}\hspace{0.2cm}[0,T]\times\bo, \vspace{0.2cm}\\
\nabla \cdot \bu=0\hspace{0.2cm}\text{in}\hspace{0.2cm}[0,T]\times\bo, \vspace{0.2cm}\\
n(0)=n_0,\quad c(0)=c_0,\quad \bu(0)=\bu_0\qquad\text{in}\qquad\mathcal{O}.
\end{cases}
\end{equation}
In addition to the unknows  $\bu$, $c$, $n$, we have the scalar pressure $P$. The positive number $T$ is the final observation time, and $\bo\subset \mathbb{R}^2$ is a  domain where the cells and the fluid move and interact. The positive constants $\eta$, $\mu$
and $\delta$ are the corresponding diffusion coefficients for the fluid, the oxygen, and the bacteria, respectively. The given functions $\chi$ and $f$ denote the chemotactic sensitivity and the oxygen
consumption rate, respectively. The symbol $\varPhi$ denotes a given time-independent potential function representing, e.g., the gravitational force or centrifugal force.

The mathematical analysis of system \re{1.2} has been investigated by several authors.   The existence of weak solutions and the existence of a unique classical solution have been  proven, see for instance \cite{Cao,Chae,Duan1,Duan2,Duan3,Liu,Tao,Tao1,Win,Win1} and references therein. In the case $d=2$, the existence of a global weak solutions for \re{1.2} without the nonlinear convective term $(\bu\cdot\nabla)\bu$ is obtained in \cite{Duan2,Tao,Tao1} and in \cite{Duan3} with nonlinear diffusion. The existence of weak global solutions under various assumptions on the data can be found in \cite{Duan1,Liu}; the global existence of smooth solutions has been proven in \cite{Chae,Win}. Results on the existence of classical solution are found in  \cite{Cao,Duan2,Win1}.

Fix $T>0$. In this paper, we are interested in the mathematical analysis of a stochastic version of problem \re{1.2} in the two-dimensional bounded domain. More precisely, for a given family of independent, identically distributed standard real-valued  Brownian motions $\{\beta^k\}_{k=1,2}$,  and  a cylindrical Wiener processes  $W$ evolving on  a fixed separable Hilbert space $\mathcal{U}$  defined on a filtered probability space, $(\Omega, \mathbb{F}, (\mathcal{F}_t)_{t\in[0,T]}, \mathbb{P})$,
we consider the following system
\begin{equation}\label{1.1}
\begin{cases}
 d \bu +\left[(\bu\cdot \nabla) \bu +\nabla P-\eta\Delta \bu
 \right]dt=n\nabla\Phi dt+ g(\bu,c)dW_t\hspace{0.2cm}\text{in}\hspace{0.2cm}[0,T]\times\bo,\\
dc + \bu\cdot \nabla cdt=\left[\mu\Delta c -nf(c)\right]dt+ \gamma\Sum_{k=1}^2\sigma_k\cdot\nabla c \circ d \beta^k_t\hspace{0.2cm}\text{in}\hspace{0.2cm}[0,T]\times\bo,\\
d n +\bu\cdot \nabla n dt= \left[\delta\Delta n-\nabla \cdot (n \chi(c)\nabla c)\right]dt\hspace{0.2cm}\text{in}\hspace{0.2cm}[0,T]\times\bo, \vspace{0.2cm}\\
 \nabla \cdot \bu=0\hspace{0.2cm}\text{in}\hspace{0.2cm}[0,T]\times\bo,\vspace{0.2cm}\\
\frac{\partial n}{\partial \nu}=\frac{\partial c}{\partial \nu}=0\qquad\text{on}\qquad [0,T]\times\partial \mathcal{O},\vspace{0.2cm}\\
\bu=0\qquad\text{on}\qquad [0,T]\times\partial \mathcal{O},\vspace{0.2cm}\\
n(0)=n_0,\quad c(0)=c_0,\quad \bu(0)=\bu_0\qquad\text{in}\qquad\mathcal{O},
\end{cases}
\end{equation}
where $\mathcal{O}\subset\mathbb{R}^2$ is  a bounded domain with smooth boundary $\partial\mathcal{O}$ and the positive constant $\gamma$ is the intensity of the noise. The symbol $\circ$ means that the stochastic differential is understood in the Stratonovich sense. The main difference between the deterministic model \re{1.2} and the stochastic model \re{1.1} is the presence of the terms $g(\bu,c)dW_t$ and $ \gamma\sum_{k=1}^2\sigma_k\cdot\nabla c \circ d \beta^k_t$ called noise terms. The presence of these noise terms weakened the regularity in time of the velocity field and  the concentration of oxygen and so, make the mathematical analysis more involved.

Our investigation is motivated by the need for a sound mathematical analysis for the understanding of the effect of small scale perturbations such as random pollution of water or air which are inherently present in nature (see \cite{Chri,Nis}). The presence of these stochastic perturbations can lead to new and important phenomena. In fact, in two-dimensional case, many models such as  the Navier-Stokes equation, the Oldroy-B type model, the Landau-Lifshitz-Bloch equation, and magnetohydrodynamics model  with sufficiently degenerate noise for example have a unique invariant measure and hence exhibit ergodic behavior in the sense that the time average of a solution is equal to the average over all possible initial data. Despite continuous efforts in the last 30  years, such property has so far not been found for the deterministic counterpart of these equations. This property could lead
to profound understanding of the nature of turbulence. To the best of our knowledge, the only papers that consider the mathematical analysis of a stochastic version of chemotaxis-fluid interaction model are \cite{Zhai,Zhang} where the authors have proved the existence of both mild and weak solutions for the model \re{1.1} with $\gamma=0$ and $g(\bu,c)=g(\bu)$ in a two and three dimensional bounded domain under some strong assumptions on the data.

The aim of this article is to study the global resolvability of problem \re{1.1} with  positive parameters $\eta$, $\mu$ $\gamma$ and $\delta$ different from zero. We prove the existence and uniqueness of a probabilistic strong solution in a two dimensional bounded domain. The proof is based on a Galerkin scheme and  the Yamada-Watanabe Theorem. Let us recall that the presence of the noise on the $c$-equation makes the mathematical analysis of the model more involved. In fact, the noise term in $c$-equation makes impossible the application of the deterministic maximum principle method for the proof of the non-negativity of solution as is done in the literature. Moreover, the stochastic version of maximum principle method where we learn from \cite{Denis1} need to be adapted in order to conserve the positivity of solutions. The main difference between our work and that of \cite{Zhai} is that the  model considered in \cite{Zhai}  does not contain any noise on the $c$-equation and the noise term in the $\bu$ equation depend only on the velocity field $\bu$. Therefore, the present paper can be seen as a generalization of \cite{Zhai}.

The organisation of this article is as follows. In Section 2, we define various functional spaces,  and introduce  assumptions which are used throughout in our paper. In Section 3, we state and prove  the main result which is the existence of a unique probabilistic strong solution. In Section 4, we give a detailed proof of important ingredients which have been useful for the proof of the main result. In Section 5, we prove the mass conservation property and the non-negativity of the strong solution. Besides that, we prove an energy inequality which may be useful for the study of the invariant measure in future.

\section{Functional setting of the model and assumptions}

Throughout the paper, we assume that $\bo \subset \mathbb{R}^2$  is a bounded domain with boundary $\partial \bo$ of class $C^\infty$. The symbol $ L^p(\mathcal{O})$ denotes the $L^p$ space with respect to the Lebesgue measure while $W ^{m,p}(\mathcal{O})$ denotes the Sobolev space of functions whose distributional derivatives of order up to $m$ belong to $L^p(\bo)$. The spaces of functions $\phi : \bo\to \mathbb{R}^2$ such that each component of $\phi$  belongs to $L^p(\bo)$ or to $W^{m,p}(\bo)$ are denoted by $\mathbb{L}^p(\bo)$ or by  $\mathbb{W}^{m,p}(\bo)$.  We denote by $\abs{.}_{L^p}$ the norm on $ L^p(\mathcal{O})$ or  $\mathbb{L}^p(\bo)$  and by  $\norm{.}_{W^{m,q}}$  the norm on $W^{m,p}(\mathcal{O})$ or $\mathbb{W}^{m,p}(\bo)$.  For $p = 2$ the function space $W^{m,2}(\bo)$ (resp. $\mathbb{W}^{m,2}(\bo)$) is denoted by $H^{m}(\bo)$ (resp. $\mathbb{H}^{m}(\bo)$)  and its norm will be denoted by $\abs{\cdot}_{H^m}$. By $\mathbb{H}_0^{1}(\bo)$ we mean the space of functions in $\mathbb{H}^1$ that vanish on the boundary  $\partial\bo$. The inner product on $L^2(\bo)$ will be denoted by $(\cdot,\cdot)$.  Following the notations using in \cite{Tem} for the Navier-Stokes model, we introduce the following space
 $\mathcal{V}=\{ \bv\in C^\infty_c(\bo;\mathbb{R}^2): \text{such that } \nabla\cdot\bv=0\},$
and  define the spaces  $H$ and $V$ as the closure of $\mathcal{V}$ in  $\mathbb{L}^2(\bo)$ and $\mathbb{H}_0^1(\bo)$, respectively. We endow $H$ with the scalar product and norm of  $\mathbb{L}^2(\bo)$. As usual, we equip the space $V$ with the gradient-scalar product and the gradient-norm $\abs{\nabla\cdot}_{L^2}$, which is equivalent to the $\mathbb{H}_0^1(\bo)$-norm. As usual, $\mathcal{P}$ denotes the Helmholtz projection from $\mathbb{L}^2(\bo)$ onto $H$. 
It is also known that $V$ is dense in $H$ and that the embedding is continuous and compact. Identifying $H$ with its dual, we have the Gelfand triple
$V\hookrightarrow H\hookrightarrow V^*$.

We define the Newmann Laplacian  operator on $L^2(\bo)$ by $A_1\phi=-\Delta\phi $ for all $\phi\in D(A_1)$ where
\begin{equation*}
 D(A_1)=\{\phi\in H^2(\bo): \frac{\partial\phi}{\partial\nu}=0, \text{ on } \partial\bo \}.
\end{equation*}
It is known that $A_1$ is a non-negative self-adjoint operator in $L^2(\bo)$. As we are working on a bounded
domain, $A_1$ has compact resolvent, see e.g. \cite{Brz1}.
Hence, there exists an orthonormal basis $\{\varphi_i\}_{i=1}^\infty\subset C^\infty(\bo)$ of $L^2(\bo)$ consisting of the eigenfunctions of the Neumann Laplacian $A_1$. Also we have the dense and compact embeddings $H^2(\bo)\hookrightarrow H^1(\bo)\hookrightarrow L^2(\bo)$.

Now we define the Hilbert space $\bhc$ by
\begin{equation*}
\bhc =H\times H^1(\bo),
\end{equation*}
endowed with the scalar product whose associated norm is given by
\begin{equation*}
\abs{(\bu,c)}^2_\bhc=\abs{\bu}^2_{L^2}+\abs{c}^2_{H^1}, \ (\bu,c)\in \bhc.
\end{equation*}

We introduce the bilinear operators $B_0$, $B_1$ and $R_2$   and their associated trilinear forms $b_0$, $b_1$ and $r_2$ respectively as follows:
 \begin{equation*} 
 	\begin{split}
&(B_0(\bu ,\bv),\bw)=\int_\bo[(\bu(x)\cdot\nabla)\bv(x)]\cdot \bw (x)dx=b_0(\bu,\bv,\bw), \ \forall \bu\in V,\  \bv\in V,\  \bw\in V,\\
&(B_1(\bu, c),\psi)=\int_\bo\bu(x)\cdot\nabla c(x)\psi(x) dx=b_1(\bu,c,\psi),	\ \forall \bu\in V, \  c\in H^1(\bo),\  \psi \in  H^1(\bo),
 	\end{split}
 \end{equation*}
 \begin{align}
(R_2(n, c),\psi)&=\int_\bo \nabla\cdot(n(x)\nabla c(x))\psi(x) dx\notag\\
&=-\int_\bo n(x)\nabla c(x)\cdot\nabla\psi(x) dx=r_2(n,c,\psi),	\ \forall n\in L^2(\bo), \  c\in H^1(\bo),\  \psi \in H^3(\bo).\notag
 \end{align}
It is well known in \cite[Chapter II, Section 1.2]{Tem} that the operator $B_0$ is well-defined.  The operator $B_1$ is well-defined for  $\bu\in V$,   $c\in H^1(\bo)$ and  $ \psi \in  H^1(\bo)$ since by the H\"older inequality and the Sobolev embedding of $H^1(\bo)$ into $L^4(\bo)$, we have
\begin{align*}
(B_1(\bu, c),\psi)&\leq \abs{\bu}_{L^4}\abs{\nabla c}_{L^2}\abs{\psi}_{L^4}\\
&\leq\bk  \abs{\nabla\bu}_{L^2}\abs{ c}_{H^1}\abs{\psi}_{H^1}.
\end{align*}
In a similar way, we can also check that the operator $R_2$ is well-defined for  $n\in L^2(\bo)$,   $ c\in H^1(\bo)$ and  $ \psi \in  H^1(\bo)$. In fact, in addition to  the H\"older inequality, by using the Sobolev embedding of $H^2(\bo)$ into $L^\infty(\bo)$, we see that
\begin{align*}
	(R_2(n, c),\psi)&\leq \abs{n}_{L^2}\abs{\nabla c}_{L^2}\abs{\nabla\psi}_{L^\infty}\\
	&\leq \abs{n}_{L^2}\abs{ c}_{H^1}\abs{\psi}_{H^3}.
\end{align*}

We also introduce the following coupling  mappings $R_0$ and $R_1$
\begin{equation*}
\begin{split}
	&(R_0(n,\varPhi), \bv)=\int_\bo n(x)\nabla\varPhi(x)\cdot\bv(x) dx,\ \forall n  \in L^2(\bo), \ \bv\in H,\ \varPhi\in W^{1,\infty}(\bo),\\
&(R_1(n, c),\psi)=\int_\bo n(x) f(c(x))\psi(x) dx,	\  \forall n\in L^2(\bo), \ c\in L^\infty(\bo),\  \psi \in L^2(\bo), \ f\in L^\infty(\mathbb{R}).
\end{split}
\end{equation*}
We note that the operators $R_0$ and $R_1$ are well-defined. Indeed, for  $n  \in L^2(\bo)$,  $\bv\in H$ and $ \varPhi\in W^{1,\infty}(\bo)$ we see that
\begin{equation*}
(R_0(n,\varPhi), \bv)\leq \abs{\varPhi}_{W^{1,\infty}}\abs{n}_{L^2}\abs{\bv}_{L^2}.
\end{equation*}
Further, for $n\in L^2(\bo)$,  $c\in L^\infty(\bo)$,  $\psi \in L^2(\bo)$ and  $f\in L^\infty(\mathbb{R})$, we also see that
\begin{equation*}
(R_1(n, c),\psi)\leq \abs{f(c)}_{L^{\infty}}\abs{n}_{L^2}\abs{\psi}_{L^2}.
\end{equation*}
Hereafter, $\mathfrak{A}:=(\Omega, \mathbb{F}, (\mathcal{F}_t)_{t\in[0,T]}, \mathbb{P})$ will be a complete probability space equipped with a filtration $(\mathcal{F}_t)_{t\in[0,T]}$ satisfying the usual conditions, i.e. the filtration is right-continuous and all null
sets of $\mathcal{F}$ are elements of $\mathcal{F}_0$.  Let $\buc$ be a separable Hilbert space with basis $\{e_k\}_{k=1}^\infty$  and  $W$ be a cylindrical  Wiener process over $\buc$. In particular, according to \cite[Proposition 4.3]{Da} the Wiener process $t\mapsto W_t$ can be expressed as
\begin{equation*}
	W_t=\sum_{k=1}^\infty W_t^ke_k,\qquad\text{for all } t\in [0,T],
\end{equation*}
where $\{W^k:k\in\mathbb{N}\}$ is a family of mutually independent  standard $\mathbb{R}$-valued Brownian motion over $\mathfrak{A}$. 

For any Hilbert space $X$, we will denote by $\mathcal{L}^2(\buc;X)$ the separable Hilbert space of Hilbert-Schmidt operators from $\buc$ into $X$.
For a separable Banach space $X$,  $p\in[1,\infty)$ and $T>0$ we denote by  $\mathcal{M}^p_{\MA}(0,T;X)$ the space of all processes $\psi\in L^p(\Omega\times(0,T),d\mathbb{P}\otimes dt;X)$ over $\MA$, being  $\{\mathcal{F}_t\}_{t\in [0,T]}$-progressively measurable. We denote by  $L^p(\Omega;C([0,T];X))$, $1\leq p<\infty$, the space of all continuous and $\{\mathcal{F}_t\}_{t\in [0,T]}$-progressively measurable $X$-valued processes $\{\psi_t;\ 0\leq t\leq T \}$ over $\MA$  satisfying
\begin{equation*}
	\mathbb{E}\left[\sup_{t\in[0,T]}\norm{\psi_t}_X^p \right]<+\infty.
\end{equation*}
If $Y$ is a Banach space, we  will denote by $\mathcal{L} (X, Y )$ the space of bounded linear operators.

From the theory of stochastic integration on infinite dimensional Hilbert space (see \cite[Chapter 4]{Da}), for any process $\rho\in\mathcal{M}^2_{\MA}(0,T;\mathcal{L}^2(U;H))$, the stochastic integral of $\rho$ with respect to the Wiener process $t\mapsto W_t$ is denoted by $$\int_0^t\rho(s)dW_s,\quad0\leq t\leq T,$$ and is defined as the unique continuous $H$-valued martingale over $\mathfrak{A}$, such that for all $h\in H$, we have
\begin{equation*}
	\left(\int_0^t\rho(s)dW_s,h\right)_H=\sum_{k=1}^\infty\int_0^t(\rho(s)e_k,h)_HdW^k_s,\ \ 0\leq t\leq T,
\end{equation*}
where the integral with respect to $dW^k_s$ is understood in the sense of It\^o.

We introduce now the following conditions on the parameters and functions involved in the system (\ref{1.1}).\newline
\begin{assumption}\label{ass_1}
For the parameter functions $\chi$, $f$ and $\varPhi$  in (\ref{1.1}), we assume that $\chi(c)$ is a non-negative constant,  i.e. $\chi(c)=\chi>0$ and  require that  $f$ and $\varPhi$  satisfy
\begin{equation}\label{2.1}
	\begin{split}
		&f\in C^1([0,\infty)), \qquad f(0)=0,\qquad\text{and}\qquad f>0,\qquad f'>0\qquad\text{in } (0,\infty),\\
		&\varPhi\text{ is time-independent and }\varPhi\in W^{1,\infty}(\bo).
	\end{split}
\end{equation}
\end{assumption}
Throughout this paper, we set
\begin{equation}\label{Eq:K-f}
\bk_f:=\frac{\chi^2}{2\delta\Min_{0\leq c\leq \abs{c_0}_{L^\infty}}f'}+\frac{1}{\Min_{0\leq c\leq \abs{c_0}_{L^\infty}}f'}.
\end{equation}
Furthermore,  we consider a family of vector fields $\{\sigma_1,\sigma_2\}$ satisfying the following assumptions.
\begin{assumption}\label{ass_2}
	\item[($\mathbf{A}_1$)] For $k\in \{1,2\}$, $\sigma_k:=(\sigma_k^1,\sigma_k^2)\in  W^{1,\infty}(\mathcal{O})\times W^{1,\infty}(\mathcal{O})$ and $\sigma_k=0$ on $\partial \bo$.
	\item[($\mathbf{A}_2$)] $\sigma_k$ is a divergence free vector fields, that is $\nabla\cdot \sigma_k=0$, for $k=1,2$.
	\item[($\mathbf{A}_3$)]  The matrix-valued function $q : \bo \times  \bo \to \mathbb{R}^2 \otimes  \mathbb{R}^2$ defined by
	\begin{equation}
		q^{i,j}(x,y)=\sum_{k=1}^{2}\sigma_k^i(x)\sigma_k^j(y), \qquad \forall i, j=1,2\ \text{and} \ \forall x, y \in\bo,
	\end{equation}
	satisfies  $q(x, x) = Id_{\mathbb{R}^2}$ for any $x\in  \bo$.
\end{assumption}
Before introducing the other standing assumptions used in this paper, we shall make few important remarks and observations on Assumption \ref{ass_2} and the noise $\Sum_{k=1}^2\sigma_k\cdot\nabla c \circ d \beta^k_t$.
\begin{remark}
	Setting for $k=1,2$, 
	\begin{equation*}
 \sigma_k(x)=\begin{cases}
 	g_k\quad \text{  if } x\in \bar{\bo}\setminus\partial\bo,\vspace{0.2cm}\\
 	0 \quad  \text{  if } x\in \partial\bo,
 \end{cases}
	\end{equation*}
	where $\{ g_1, g_2\}$ is the canonical basis of $\mathbb{R}^2$,  the family of vector fields $\{\sigma_1,\sigma_2\}$ satisfies (\textbf{A}$_1$), (\textbf{A}$_2$) and (\textbf{A}$_3$).
\end{remark}
Hereafter we will use the following notation
\begin{equation}
\abs{\sigma}_{L^\infty}=\left(\sum_{k=1}^2\abs{\sigma_k}^2_{L^\infty}\right)^{1/2}\ \text{ and } \ \abs{\sigma}_{W^{1,\infty}}=\left(\sum_{k=1}^2\abs{\sigma_k}^2_{W^{1,\infty}}\right)^{1/2}.
\end{equation}
Owing to \cite[p. 65, Section 4.5.1]{Duan}, the Stratonovich integral $\gamma\int_0^t\sigma_k\cdot\nabla c(s) \circ d \beta^k_s$ can be expressed  as the It\^o integral with a  correction term as follows:
\begin{equation}
\gamma\int_0^t\sigma_k\cdot\nabla c(s) \circ d \beta^k_s=\frac{1}{2}\int_0^tD_c(\gamma\sigma_k\cdot\nabla c(s))(\gamma\sigma_k\cdot\nabla c(s))ds+\gamma\int_0^t\sigma_k\cdot\nabla c(s)d \beta^k_s, \label{2.6*}
\end{equation}
where, $D_c(\gamma\sigma_k\cdot\nabla c)$ denotes the Fr\'echet derivative of $\gamma\sigma_k\cdot\nabla c$ with respect to $c$.
\begin{lemma}\label{lemma2.1}
	If Assumption \ref{ass_2} holds, then for all $t\in [0,T]$,
	\begin{equation}
	\frac{1}{2}\int_0^t\sum_{k=1}^2  D_c(\gamma\sigma_k\cdot\nabla c(s))(\gamma\sigma_k\cdot\nabla c(s))ds=\frac{\gamma^2}{2}\int_0^t\Delta c(s)ds, \ c\in H^2(\bo).\label{2.7}
	\end{equation}
\end{lemma}
\begin{proof}
Let $c\in H^2(\bo)$ and  $t\in [0,T]$ be arbitrary but fixed. Then for all $s\in [0,t]$ and  $k= 1, 2$,
	\begin{equation*}
	\sum_{k=1}^2 D_c(\gamma\sigma_k\cdot\nabla c)(\gamma\sigma_k\cdot\nabla c)=\gamma\sum_{k=1}^2\sigma_k\cdot\nabla (\gamma \sigma_k\cdot\nabla c)=\gamma^2\sum_{k=1}^2\sigma_k\cdot\nabla (\sigma_k\cdot\nabla c).
	\end{equation*}
	Since $\nabla\cdot \sigma_k=0$, we remark that $\sigma_k\cdot\nabla c=\nabla\cdot (c\sigma_k)$ and therefore,
	\begin{align}
	\gamma^2\sum_{k=1}^2\sigma_k\cdot\nabla (\sigma_k\cdot\nabla c)=\gamma^2\sum_{k=1}^2\sigma_k\cdot\nabla (\nabla\cdot (c\sigma_k))=\gamma^2\sum_{k=1}^2\nabla\cdot\left(\sigma_k\nabla\cdot (c\sigma_k)\right).\label{2.8}
	\end{align}
	For the second equality we have used once more the fact that $\nabla\cdot \sigma_k=0$ for all $k=1,2$.
	
Since $\sigma_k=(\sigma_k^1,\sigma_k^2)\in W^{1,\infty}(\mathcal{O})\times W^{1,\infty}(\mathcal{O})$ and $c\in H^2(\bo)\hookrightarrow L^\infty(\bo)$,	 we can apply the  differentiation of  product formula given in \cite[Proposition 9.4, P. 269]{Bre} to obtain,
	\begin{equation}
	\sum_{k=1}^2\nabla\cdot\left(\sigma_k(\nabla\cdot (c\sigma_k)\right)=\sum_{i,j=1}^2\frac{\partial^2}{\partial x_i\partial x_j}(q^{ij}(x,x)c)-\nabla\cdot\left(\left(\sum_{k=1}^2\sigma_k\cdot\nabla\sigma_k \right)c\right),\label{2.9}
	\end{equation}
	where  $\sigma_k\cdot\nabla\sigma_k$ is the vector field with components
	\begin{equation*}
	(\sigma_k\cdot\nabla\sigma_k)^i=\sum_{j=1}^2\sigma^j_k\frac{\partial}{\partial x_j}\sigma^i_k.
	\end{equation*}
	Applying the  differentiation of  product formula once more, for $j = 1,2$, we see that
	\begin{equation}
	\sum_{k=1}^2\sum_{j=1}^2(\nabla\sigma_k\cdot\sigma_k)^i=\sum_{j=1}^2\frac{\partial}{\partial x_j}q^{ij}(x,x)-\sum_{k=1}^2\sigma_k^i\nabla\cdot\sigma_k=\sum_{j=1}^2\frac{\partial}{\partial x_j}\delta_{ij}=0.\label{2.10}
	\end{equation}
	In (\ref{2.10}), we have used the fact that $\nabla\cdot\sigma_k=0$ and also the fact that $q^{ij}=\delta_{ij}$ (see  (\textbf{A}$_3$) of assumption $2$). \\
	From \re{2.9} and \re{2.10}, we infer that
	\begin{equation}
	\sum_{k=1}^2\nabla\cdot\left(\sigma_k\nabla\cdot (c\sigma_k)\right)=\sum_{i,j=1}^2\frac{\partial^2}{\partial x_i\partial x_j}(q^{ij}(x,x)c)=\sum_{i,j=1}^2\frac{\partial^2}{\partial x_i\partial x_j}(\delta_{ij}c)=\Delta c.\label{2.11}
	\end{equation}
	Combining \re{2.11} and \re{2.8}, we derive \re{2.7} which completes the proof of Lemma \ref{lemma2.1}.
\end{proof}
Define for $k\in \{1,2\}$, a map $\phi_k:H^1(\bo)\to L^2(\bo)$ by $\phi_k(c)=\sigma_k\cdot\nabla c$. Then, the map  $\phi: H^1(\bo)\to  \mathcal{L}^2(\mathbb{R}^2; L^2(\bo))$ given by
\begin{equation*}
\phi(c)(h)=\sum_{k=1}^2\phi_k(c)h_k,\qquad c\in H^1(\bo),\  h=(h_1,h_2)\in\mathbb{R}^2,
\end{equation*}
is well defined under the condition (\textbf{A}$_1$).
Let $\{g_1,g_2\}$ be the orthonormal basis of $\mathbb{R}^2$ then $\phi(c)(g_k)=\phi_k(c)$, for all $c\in H^1(\bo)$. Let $\beta=(\beta^1,\beta^2)$ be a standard two dimensional Brownian motion over $\MA$, independent of $W$. We will repeatedly use the following notation
\begin{equation}
\phi(c)d\beta_s=\sum_{k=1}^2\phi_k(c)d \beta^k_s.	\label{2.3}
\end{equation}
We recall that
throughout this paper, the symbols $\bk$, $\bk_{GN}$ and $\bk_i$, $i\in\mathbb{N}$ will denote positive constants which may change from one line to another.
\begin{assumption}\label{ass_3}
 Let   $g:\bhc\to \mathcal{L}^2(\buc,H)$ be a continuous mapping.
In particular,
 there exists a positive constant $L_g$ such that for any $(\bu,c)\in  \bhc$,
\begin{equation}\label{2.5}
	\abs{g(\bu,c)}_{\mathcal{L}^2(\buc,H)}\leq L_g(1+\abs{(\bu,c)}_{\bhc}).
\end{equation}
\end{assumption}
\begin{assumption}\label{ass_4}
Let $g:\bhc\to \mathcal{L}^2(\buc,H)$ be a Lipschitz-continuous mapping.
In particular,
 there exists a positive constant $L_{Lip}$ such that for all $(\bu_i,c_i)\in \bhc$, $i=1,2$,
\begin{equation}
\abs{g(\bu_1,c_1)-g(\bu_2,c_2)}_{\mathcal{L}^2(\buc;H)}\leq L_{Lip}\abs{(\bu_1-\bu_2,c_1-c_2)}_\bhc.\label{5.2}
\end{equation}
\end{assumption}
Using  the previous notations,   setting $\xi=\eta+\frac{\gamma^2}{2}$, and taking into account Lemma \ref{lemma2.1},
 the model \re{1.1} can  formally be written in the following abstract form
\begin{equation}\label{3.1}
	\begin{split}
	&	\bu(t) +\Int_0^t[\eta A_0\bu(s)+B_0(\bu(s),\bu(s)]ds=\bu_0+\Int_0^tR_0(n(s),\varPhi)ds + \Int_0^tg(\bu(s),c(s)) dW_s,\vspace{0.2cm}\\
	&	c(t) + \Int_0^t[\xi A_1c(s)+B_1(\bu(s), c(s))]ds= c_0-\Int_0^t R_1(n(s),c(s))ds+\gamma\int_0^t\phi(c(s))d\beta_s,\vspace{0.2cm}\\
	&	n(t) +\Int_0^t[\delta A_1n(s)+B_1(\bu(s),n(s))]ds =n_0-\Int_0^t R_2(n(s),c(s))ds.
	\end{split}
\end{equation}
These equations are understood being valid in $V^*$,  $H^{-2}(\bo)$ and  $H^{-3}(\bo)$, respectively. 

We end this section by introduce  some notations. Let $Y$ be a Banach space.
By $\mathcal{C}([0,T]:Y)$ we denote the space of continuous functions  $\bv:[0,T]\to Y$ with the topology induced by the norm  defined by
	$$\abs{\bv}_{\mathcal{C}([0,T];Y)}:=\Sup_{0\leq s\leq T}\norm{\bv(s)}_{Y}.$$
With $L^{2}(0,T;Y)$ we denote the space of measurable functions $\bv:[0,T]\to Y$ with the topology generated by the norm $$\abs{\bv}_{L^{2}(0,T;Y)}:=\left(\int_0^T\norm{\bv(s)}_{Y}^2ds\right)^{1/2},$$
while by   $L_w^{2}(0,T;Y)$ we denote the space of measurable functions $\bv:[0,T]\to Y$ with weak topology.

For a Hilbert space $X$, we denote by $X_w$ the space $X$ endowed with the weak topology and by 
$C([0,T];X_{w})$ we denote  the space of  functions $\bv:[0,T]\to X_{w}$  that are weakly continuous.

\section{The main result: Existence of probabilistic strong solutions}

This section is devoted to the statement of the main result of this paper.  Before proceeding further,  let us  state the following definition.
\begin{definition}\label{defi4.0}
A probabilistic strong solution of the problem \eqref{1.1} is a $H\times H^1(\mathcal{O})\times L^2(\mathcal{O})$-valued stochastic process $(\bu,c,n)$ such that
\begin{description}
\item[i)] We have 
 $\mathbb{P}$-a.e.
\begin{equation*}
\begin{split}
& \bu \in \mathcal{C}([0,T];H) \cap L^{2}(0,T;V), \\
&c\in \mathcal{C}([0,T];H^1(\bo))\cap L^{2}(0,T;H^2(\bo)),\\
&n\in \mathcal{C}([0,T];L^2_{w}(\bo))\cap L^{2}(0,T;H^1(\bo))\cap\mathcal{C}([0,T];H^{-3}(\bo)).
\end{split}
\end{equation*}
\item[ii)] $(\bu,c,n):[0,T]\times\Omega\to H\times H^1(\mathcal{O})\times L^2(\mathcal{O})$ is  progessively measurable  and  for all $p\geq 1$
\begin{align}
	&\be\sup_{0\leq s\leq T}	\abs{\bu(s)}^p_{L^2}+\be\left(\int_0^T\abs{\nabla \bu(s)}_{L^2}^2ds\right)^p<\infty,\notag\\
	& 
	\qquad\be\left(\int_0^T\abs{n(s)}^2_{L^2}ds\right)^p<\infty,\label{4.41}\\
	&\mbox{and}\quad \be\sup_{0\leq s\leq T}	\abs{c(s)}^p_{H^1}+\be\left(\int_0^T\abs{c(s)}_{H^2}^2ds\right)^p<\infty.\notag
\end{align}
\item[ iii)] for all $t\in [0,T]$ the following identity holds $\mathbb{P}$-a.s.
\begin{equation}\label{4.1}
	\begin{split}
		&\bu(t) +\Int_0^t[\eta A_0\bu(s)+B_0(\bu(s),\bu(s))]ds=\bu_0+\Int_0^tR_0(n(s),\varPhi)ds + \Int_0^tg(\bu(s),c(s)) dW_s,\vspace{0.2cm}\\
	&	c(t) + \Int_0^t[\xi A_1c(s)+B_1(\bu(s), c(s))]ds= c_0-\Int_0^t R_1(n(s),c(s))ds+\gamma\int_0^t\phi(c(s))d\beta_s,\vspace{0.2cm}\\
	&	n(t) +\Int_0^t[\delta A_1n(s)+B_1(\bu(s),n(s))]ds =n_0-\Int_0^t R_2(n(s),c(s))ds,
	\end{split}
\end{equation}
in $V^*$, $H^{-2}(\bo)$ and  $H^{-3}(\bo)$, respectively.
\end{description}
\end{definition}

Let us now present the main result of this section.
\begin{theorem}\label{theo4.2}
	Let Assumption \ref{ass_1}, Assumption  \ref{ass_2}, Assumption  \ref{ass_3}, and Assumption  \ref{ass_4}  be valid.
Let us assume that the initial data $(\bu_0,c_0,n_0)$ belong to
$$ 
H\times L^\infty(\bo)\cap H^1(\bo)\times L^2(\bo).
$$
In addition, let us assume that   $c_0(x)>0$, $n_0(x)>0$ for all $x\in \bo$ and $$\int_\bo n_0(x)\ln n_0(x)dx<\infty,$$ as well as 
	\begin{align}
	\frac{4\bk_f\Max_{0\leq c\leq \abs{c_0}_{L^\infty}}f^2}{\Min_{0\leq c\leq \abs{c_0}_{L^\infty}}f'} \leq \delta,\
	\gamma^2\leq \frac{\min\left(\xi, \frac{\xi}{2\bk_0} \right)}{6\abs{\sigma}_{L^\infty}^2},  \text{ and }  \gamma^{2p}\leq\frac{3^p\xi^p}{2^{2p+1}\abs{\sigma}_{L^\infty}^{2p}8^p},\label{4.46*}
	\end{align}
	for all $p\geq 2$, where $\bk_0$ is positive constant such that $\abs{\psi}^2_{H^2}\leq\bk_0(\abs{\Delta \psi}^2_{L^2}+\abs{\psi}_{H^1}^2)$, for all $\psi\in H^2(\bo)$ (see \cite[Proposition 7.2, P. 404]{Tay} for the existence of  such constant). 
Then, there exists a unique probabilistic strong solution to the problem \re{1.1} in the sense of Definition  \ref{defi4.0}.
\end{theorem}
\begin{remark}
We note that in the case where $f(c)=c$,  then we have  $\bk_f=\frac{\chi^2+2\delta}{2\delta}$,  and the first inequality of  the condition \re{4.46*} is satisfied if
 $$\abs{c_0}_{L^\infty}\leq\frac{\delta\sqrt{2}}{2\sqrt{\chi^2+2\delta}}.$$ Furthermore, the  condition \re{4.46*}  have been introduced in order to control the cell term in the inequality \re{3.34*} and  the higher regularity of the noise term on the $c$-equation in the inequalities \re{3.51} and \re{3.66}. However, it is known in  \cite[Remark 1.1]{Li} that,  for the two-dimensional deterministic chemotaxis system, there exists a critical mass phenomenon.  When the total initial mass of cells $\int_\bo n_0(x)dx$ above a critical mass $m_{\text{crit}}$ (i.e.  $\int_\bo n_0(x)dx>m_{\text{crit}}$), solutions blow-up in finite time, otherwise, all solutions remain bounded. While, for the two-dimensional stochastic chemotaxis system, it is shown in \cite{May} that, if the chemotaxis sensibility $\chi$ is sufficiently large, then blow-up occurs with probability $1$. For the coupled system \re{1.1}, despite the rapid flow of fluid, we also expect some phenomenons to appear. Then,   it is important to ask oneself what will happen if the  condition  \re{4.46*}  is violated? The answer to this question will be given by the  study of the blow-up criterion of the system \re{1.1} in future.
\end{remark}
In order to prove Theorem \ref{theo4.2}, we will first show that problem \eqref{1.2} has a probabilistic weak solution, see Definition \ref{4.1},  then prove the non-negativity property and the $L^\infty$-stability property  of weak solution,  which give us the possibility to  prove the pathwise uniqueness, and finally apply the Yamada-Watanabe Theorem.
But before proceeding further, we now introduce the concept of a probabilistic weak solution.
%
\begin{definition}\label{defi4.1}
	A weak probabilistic  solution of the problem \re{1.1} is a system $$(\bar{\Omega}, \bar{\mathcal{F}},\bar{\mathbb{F}}, \bar{\mathbb{P}},(\bu,c,n),(\bar{W},\bar{\beta})),$$
	 where
	\begin{description}
		\item[i)] $(\bar{\Omega}, \bar{\mathcal{F}},\bar{\mathbb{F}}, \bar{\mathbb{P}})$ is a  filtered probability space,
		\item[ii)] $(\bar{W},\bar{\beta})$ is a cylindrical Wiener processes on $\buc\times\mathbb{R}^2$ over $(\bar{\Omega}, \bar{\mathcal{F}},\bar{\mathbb{F}}, \bar{\mathbb{P}})$,
		\item[iii)]and  $(\bu,c,n):[0,T]\times\bar{\Omega}\to \bhc\times L^2(\bo)$ is a strong solution to \eqref{1.2} with driving noise $(\bar{W},\bar{\beta})$ on the filtered probability space  $(\bar{\Omega}, \bar{\mathcal{F}},\bar{\mathbb{F}}, \bar{\mathbb{P}})$.
	\end{description}
\end{definition}
The existence of weak solution to our problem is given in the following proposition.
\begin{proposition}\label{theo4.1}
	Let us assume that  Assumption \ref{ass_1}, Assumption  \ref{ass_2} and  Assumption  \ref{ass_3} are satisfied. Let $$(\bu_0,c_0,n_0)\in H\times L^\infty(\bo)\cap H^1(\bo)\times L^2(\bo),$$
	such that $c_0(x)>0$, $n_0(x)>0$ for all $x\in \bo$ and $$\int_\bo n_0(x)\ln n_0(x)dx<\infty.$$ We also assume that \re{4.46*} holds.
	Then, there exists at least one probabilistic weak solution to the problem \re{1.1} in the sense of Definition  \ref{defi4.1}.
\end{proposition}
The proof of Proposition \ref{theo4.1}, which  is very  technical is postponed  to Section 5.

Next, we prove some properties of  probabilistic weak solutions to the problem \re{1.1} such as the non-negativity and the $L^\infty$-stability which will be useful for the proof of the pathwise uniqueness result. In fact, the main ingredient for the pathwise uniqueness   is the  $L^\infty$-stability property but to obtain this property we will need the  non-negativity property.
\begin{lemma} \label{lem3.6}
Let Assumption \ref{ass_1} and  Assumption  \ref{ass_2} are satisfied.	Let  $(\bar{\Omega}, \bar{\mathcal{F}},\bar{\mathbb{F}}, \bar{\mathbb{P}},(\bu,c,n),(\bar{W},\bar{\beta}))$ be a probabilistic weak solution to the problem \re{1.1}. If  $c_0>0$ and $n_0> 0$, then the following inequality hold $\bar{\mathbb{P}}$-a.s
	\begin{equation}
		n(t)>0,\text{ and } c(t)> 0, \text{   for all  } t\in [0,T].\label{3.4}
	\end{equation}
\end{lemma}
\begin{proof} 
	We  will follow the idea developed in \cite[Section 3.1]{Francesco} combined with the idea of \cite[Lemma 14]{Denis1} and \cite[Theorem 3.7]{Brze}.  Let $t\in[0,T]$ arbitrary but fixed. We then  define
	$n_-(t) := \max(-n(t), 0)$ and  remark that  $n_-(t)\in W^{2,2}(\bo)$.
	Hence,  we multiply equation $(\ref{3.1})_3$ by $n_-(t)$, integrate over $\mathcal{O}$, and use an integration-by-parts   to obtain $\bar{\mathbb{P}}$-a.s.
	\begin{align}
		\frac{1}{2}\frac{d}{dt}\abs{n_-(t)}^2_{L^2}&=-\int_\mathcal{O}\bu(t,x)\cdot\nabla n_-(t,x)n_-(t,x)dx -\delta\abs{\nabla n_-(t)}^2_{L^2}\notag\\
		&\qquad-\chi\int_\mathcal{O}n(t,x)\nabla c(t,x)\nabla n_-(t,x)dx\notag\\
		&=\frac{1}{2}\int_\mathcal{O}n_-^2(t,x)\nabla\cdot\bu(t,x) dx -\delta\abs{\nabla n_-(t)}^2_{L^2}+\chi\int_\mathcal{\bo}n_-(t,x)\nabla c(t,x)\nabla n_-(t,x)dx\label{2.15*}\\
		&\leq -\delta\abs{\nabla n_-(t)}^2_{L^2}+\chi\abs{n_-(t)}_{L^4}\abs{\nabla c(t)}_{L^4}\abs{\nabla n_-(t)}_{L^2}.\notag
	\end{align}
	By  the Gagliardo-Nirenberg-Sobolev  inequality \re{4.4.} and the  Young inequality, we note that
	\begin{align}
		\chi\abs{n_-}_{L^4}\abs{\nabla c}_{L^4}\abs{\nabla n_-}_{L^2}
		&\leq \bk(\abs{n_-}^{1/2}_{L^2}\abs{\nabla n_-}^{1/2}_{L^2}+\abs{n_-}_{L^2})\abs{\nabla c}_{L^4}\abs{\nabla n_-}_{L^2}\notag\\
		&\leq \bk\abs{n_-}^{1/2}_{L^2}\abs{\nabla c}_{L^4}\abs{\nabla n_-}^{3/2}_{L^2}+\bk\abs{n_-}_{L^2}\abs{\nabla c}_{L^4}\abs{\nabla n_-}_{L^2}\notag\\
		&\leq \frac{\delta}{2}\abs{\nabla n_-}^2_{L^2}+\bk\abs{n_-}^{2}_{L^2}(\abs{\nabla c}^4_{L^4}+\abs{\nabla c}^2_{L^4})\label{2.16}\\
		&\leq \frac{\delta}{2}\abs{\nabla n_-}^2_{L^2}+\bk\abs{n_-}^{2}_{L^2}(\abs{\nabla c}^4_{L^4}+1).\notag
	\end{align}
	Owing to the fact that $\bar{\mathbb{P}}$-a.s. $c\in \mathcal{C}([0,T];H^1(\bo))\cap L^{2}(0,T;H^2(\bo))$, by  the following Gagliardo-Niremberg inequality \begin{equation}
		\abs{f}_{L^4}\leq \bk_{GN} (\abs{f}^{1/2}_{L^2}\abs{\nabla f}^{1/2}_{L^2}+\abs{f}_{L^2}), \ \ f\in W^{1,2}(\bo),\label{4.4.}
	\end{equation}  we note that for all $t\in [0,T]$ and $\bar{\mathbb{P}}$-a.s.
	\begin{align*}
		\int_0^t(\abs{\nabla c(s)}^4_{L^4}+1)ds&\leq \int_0^t\abs{\nabla c(s)}^4_{L^4}ds+t\\
		&\leq \bk \int_0^T\abs{\nabla c(s)}^{2}_{L^2}\abs{c(s)}^{2}_{H^2}ds+\int_0^T\abs{\nabla c(s)}^4_{L^2}ds+T\\
		& \leq \bk \sup_{0\leq s\leq T}\abs{\nabla c(s)}_{L^2}^2\int_0^T\abs{c(s)}^2_{H^2}ds+\sup_{0\leq s\leq T}\abs{\nabla c(s)}^2_{L^2}\int_0^T\abs{\nabla c(s)}^2_{L^2}ds+T\\
		&\leq \bk \sup_{0\leq s\leq T}\abs{c(s)}_{H^1}^2\int_0^T\abs{c(s)}^2_{H^2}ds+T<\infty.
	\end{align*}
	Hence, integrating \re{2.15*} over $[0,T]$, and using the inequality \re{2.16},  we infer that $\bar{\mathbb{P}}$-a.s.
	\begin{equation*}
		\abs{n_-(t)}^2_{L^2}\leq\abs{n_-(0)}^2_{L^2}+\bk \int_0^t (\abs{\nabla c(s)}^4_{L^4}+\abs{\nabla c(s)}^2_{L^4})\abs{n_-(s)}^2_{L^2}ds.
	\end{equation*}
	Thanks to Gronwall's inequality, we derive that
	\begin{equation*}
		\abs{n_-(t)}^2_{L^2}\leq\abs{(n_0)_-}^2_{L^2}\exp\left(\bk\int_0^t(\abs{\nabla c(s)}^4_{L^4}+\abs{\nabla c(s)}^2_{L^4})ds\right),
	\end{equation*}
	which  implies that $\bar{\mathbb{P}}$-a.s,  $n_-(t) = 0$ and the non-negativity of $n(t)$ follows. 
	
	For the proof of the  non-negativity property of $c(t)$,  the main idea is to apply the It\^o formula to the function $\Psi:H^2(\bo)\to \mathbb{R}$ defined by $\Psi(z)=\int_\mathcal{O}z_-^2(x)dx$ where $z_-=\max(-z; 0)$. Since  the function $\Psi$ is not twice Fr\'echet differentiable,  we will follow the idea of \cite[Lemma 14]{Denis1} (see also \cite[Theorem 3.7]{Brze}) by introducing the following  approximation of $\Psi$.  Let $\varphi:\mathbb{R}\to [-1;0]$ be a $C^\infty$ class increasing function such that
	\begin{equation}\label{4.14**}
		\varphi(s)=\begin{cases}
			-1 \text{ if } s\in (-\infty, -2]\\
			0 \text{ if } s\in [-1, +\infty).
		\end{cases}
	\end{equation}
	Let $\{\psi_h\}_{h\in\mathbb{N}}$ be a sequence of smooth functions defined by $\psi_h (y)=y^2\varphi(hy)$, for all $y\in \mathbb{R}$ and $h\in\mathbb{N}$.
	For any $h\in \mathbb{N}$, we consider the following sequence of function $\Psi_h:H^2(\bo)\to \mathbb{R}$ defined by $$\Psi_h(c)=\int_\mathcal{O}\psi_h(c(x))dx, \text{  for  } c\in H^{2}(\mathcal{O}).$$
We note that the  mapping $\Psi_h$ is twice Fr\'echet-differentiable and 
	\begin{equation*}
		\Psi'_h(c)(k)=2\int_\mathcal{O}c(x)\varphi(hc(x))k(x)dx+h\int_\mathcal{O}c^2(x)\varphi'(hc(x))k(x)dx, \quad \forall c, k\in H^{2}(\mathcal{O}),
	\end{equation*}
	as well as
	\begin{align*}
		\Psi_h^{''}(c)(z,k)&=m^2\int_\mathcal{O}c^2(x)\varphi^{''}(hc(x))z(x)k(x)dx\\
		&\qquad{}+4h\int_\bo c(x)\varphi'(hc(x))z(x)k(x)dx+2\int_\mathcal{O}\varphi(hc(x))z(x)k(x)dx,\quad \forall c,z,  k\in H^{2}(\mathcal{O}).
	\end{align*}
	By applying  the It\^o formula to  $t\mapsto\Psi_h(c(t))$, we obtain $\bar{\mathbb{P}}$-a.s.
	\begin{align}
		\Psi_h(c(t))-\Psi_h(c(0))&=\int_0^t\Psi'_h(c(s))\left(\bu(s)\cdot \nabla c(s)+\xi\Delta c(s) -n(s)f(c(s)) \right)ds\notag\\
		&\qquad{}+\frac{1}{2}\int_0^t\sum_{k=1}^2\Psi^{''}_h(c(s))\left(\gamma\phi_k(c(s)),\gamma\phi_k(c(s))\right)ds\label{2.21**}\\
		&\qquad{}+\gamma \sum_{k=1}^2\int_0^t\Psi'_h(c(s))(\phi_k(c(s)))d \bar{\beta}^k_s.\notag
	\end{align}
Now, we will find a simpler representation of the formula \re{2.21**}.

For a fixed $k=1,2$, we remark that for all $h\geq 1$,
\begin{equation}
	h\varphi'(hc)\sigma_k\cdot\nabla c=\sigma_k\cdot(h\varphi'(hc)\nabla c)=\sigma_k\cdot\nabla (\varphi(hc)),\label{4.16***}
\end{equation}
and also that	$2c\sigma_k\cdot\nabla c=\sigma_k\cdot\nabla c^2.$
Hence,  for any $h\geq 1$ thanks to an integration-by-parts and the fact that $\sigma_k=0$ on $\partial \bo$,  we have that for any $h\in \mathbb{N}$,
\begin{align}
	\Psi'_h(c)(\phi_k(c))&=2\int_\bo c(x)\varphi(hc(x))\sigma_k(x)\cdot\nabla c(x)dx+h\int_\bo c^2(x)\varphi'(hc(x))\sigma_k(x)\cdot\nabla c(x)dx\notag\\
	&=\int_\bo \varphi(hc(x))\sigma_k(x)\cdot\nabla c^2(x)dx+\int_\bo c^2(x)\sigma_k(x)\cdot\nabla (\varphi(hc(x)))dx\notag\\
	&=-\int_\bo c^2(x)\nabla\cdot(\varphi(hc(x))\sigma_k(x))dx+\int_{\partial\bo}  c^2(\sigma)\varphi(hc(\sigma))\sigma_k(\sigma)\cdot\nu d\sigma\label{4.18}\\
	&\qquad+\int_\bo c^2(x)\sigma_k(x)\cdot\nabla (\varphi(hc(x)))dx\notag \\
	&=-\int_\bo c^2(x)\nabla\cdot(\varphi(hc(x))\sigma_k(x))dx+\int_\bo c^2(x)\sigma_k(x)\cdot\nabla (\varphi(hc(x)))dx.\notag
\end{align}
Owing to the fact that $\nabla\cdot\sigma_k=0$, we derive that
\begin{align}
	\Psi'_h(c)(\phi_k(c))&=
	-\int_\bo c^2(x)\varphi(hc(x))\nabla\cdot\sigma_k(x)dx\notag\\
	&\qquad{}-\int_\bo c^2(x)\sigma_k(x)\cdot\nabla(\varphi(hc(x)))dx+\int_\bo c^2(x)\sigma_k(x)\cdot\nabla (\varphi(hc(x)))dx\label{4.11*}\\
	&=0\notag.
\end{align}
We note that
\begin{align}
	\sum_{k=1}^2\sigma_k\cdot\nabla c\sigma_k\cdot\nabla c&=\sum_{k=1}^2\sum_{i,j=1}^2\sigma^i_k\sigma_k^j\frac{\partial c}{\partial x_i}\frac{\partial c}{\partial x_j}\notag\\
	&=\sum_{i,j=1}^2q^{ij}(x,x)\frac{\partial c}{\partial x_i}\frac{\partial c}{\partial x_j}\label{2.24*}\\
	&=\sum_{i,j=1}^2\delta_{ij}\frac{\partial c}{\partial x_i}\frac{\partial c}{\partial x_j}=\sum_{i=1}^2\frac{\partial c}{\partial x_i}\frac{\partial c}{\partial x_i}=\abs{\nabla c}^2.\notag
\end{align}	
Therefore,
\begin{align}
	\sum_{k=1}^2\Psi^{''}_h(c)\left(\gamma\phi_k(c),\gamma\phi_k(c)\right)&=\gamma^2h^2\int_\mathcal{O}c^2(x)\varphi^{''}(hc(x))\abs{\nabla c(x)}^2dx\notag\\
	&\qquad{}+4h\gamma^2\int_\bo c(x)\varphi'(hc(x))\abs{\nabla c(x)}^2dx+2\gamma^2\int_\mathcal{O}\varphi(hc(x))\abs{\nabla c(x)}^2dx.\notag
\end{align}
On the other hand, by  integration-by-parts, we get
\begin{align}
	&\gamma^2\Psi'_h(c)(\Delta c)\notag\\
	&=2\gamma^2\int_\mathcal{O}c(x)\varphi(hc(x))\Delta c(x)dx+h\gamma^2\int_\mathcal{O}c^2(x)\varphi'(hc(x))\Delta c(x)dx\notag\\
	&=-2\gamma^2\int_\mathcal{O}\nabla c(x)\cdot\nabla(c(x)\varphi(hc(x)))dx-h\gamma^2\int_\mathcal{O}\nabla c(x)\cdot\nabla(c^2(x)\varphi'(hc(x)))dx\notag\\
	&\quad+2\gamma^2\int_{\partial\mathcal{O}}\frac{\partial c(\sigma)}{\partial\nu}\varphi(hc(\sigma))\Delta c(\sigma)d\sigma+2h\gamma^2\int_{\partial\mathcal{O}}\frac{\partial c(\sigma)}{\partial\nu}c(\sigma)\varphi'(hc(\sigma))\Delta c(\sigma)d\sigma\notag\\
	&=-2\gamma^2\int_\mathcal{O}\varphi(hc(x))\abs{\nabla c(x)}^2dx-2h\gamma^2\int_\bo c(x)\varphi'(hc(x))\abs{\nabla c(x)}^2dx\label{4.13*}\\
	&\qquad{}-2h\gamma^2\int_\bo c(x)\varphi'(hc(x))\abs{\nabla c(x)}^2dx-\gamma^2h^2\int_\mathcal{O}c^2(x)\varphi^{''}(hc(x))\abs{\nabla c(x)}^2dx\notag\\
	&=-\sum_{k=1}^2\Psi^{''}_h(c)\left(\gamma\phi_k(c),\gamma\phi_k(c)\right).\notag
\end{align}
In the equality \re{4.13*}, we have used the fact that $\frac{\partial c}{\partial\nu}$ vanishes on $\partial\bo$.

Therefore, recalling that $\xi=\eta+\frac{\gamma^2}{2}$ and using \re{4.11*} and \re{4.13*}, the equality \re{2.21**} is equivalent to
	\begin{equation}
		\int_\bo\psi_h(c(t,x))dx-\int_\bo\psi_h(c_0(x))dx
		=\int_0^t\Psi'_h(c(s))\left(\bu(s)\cdot \nabla c(s)+\eta\Delta c(s) -n(s)f(c(s)) \right)ds,\notag
	\end{equation}
	from which along with the passage  to the limit as $h\to \infty$ we infer that
	\begin{align}
		&-\int_\bo c^2_-(t,x)dx+\int_\bo (c_0(x))_-^2dx\notag\\
		&=-2\int_0^t\int_\mathcal{O}\left((\bu(s,x)\cdot \nabla c(s,x)+\eta\Delta c(s,x) -n(s,x)f(c(s,x)) )\right)c(s,x)1_{\{c(s,x)<0\}}dxds\notag\\
		&=2\int_0^t\int_\mathcal{O}\left(\eta\abs{\nabla c(s,x)}^2+n(s,x)f(c(s,x))c(s,x) \right)1_{\{c(s,x)<0\}}dxds.\notag
	\end{align}
	We note that, in the last line,  we used an integration-by-parts and the fact that $\nabla\cdot\bu=0$.
	By the mean value theorem,  the fact that $f(0)=0$, and  $f'>0$ as well as  $1_{\{c<0\}}>0$, $c^2>0$, and $n> 0$, we deduce   that $\abs{c_-(t)}^2_{L^2}\leq\abs{(c_0)_-}^2_{L^2}$. This implies that $c_-(t)=0$ $\bar{\mathbb{P}}$-a.s. and end the proof of Lemma \ref{lem3.6}.
\end{proof}
With the non-negativity of probabilistic weak solutions in hand, we are able now to state and prove the $L^\infty$-stability.
\begin{cor}\label{lem3.7}
Under the same assumptions as in Lemma \ref{lem3.6},	if  $(\bar{\Omega}, \bar{\mathcal{F}},\bar{\mathbb{F}}, \bar{\mathbb{P}},(\bu,c,n),(\bar{W},\bar{\beta}))$ is a probabilistic weak solution to the problem \re{1.1}, then   for all $t\in [0,T]$
	\begin{equation}
		\abs{c(t)}_{L^\infty}\leq\abs{c_0}_{L^\infty},\quad\bar{\mathbb{P}}\text{-a.s.}\label{3.14}
	\end{equation}
\end{cor}
\begin{proof}
The idea of  the proof comes from \cite[Section 3.2]{Francesco}. We apply the It\^o formula to the process $t\mapsto\Psi(c(t)):=\int_\mathcal{O}c^p(t,x)dx$, for any $p\geq2$ and evaluate the limit as $p$ tends to $\infty$. Let $\Psi: H^2(\bo)\to \mathbb{R}$ be the functional defined by $\Psi(c)=\int_\mathcal{O}c^p(x)dx$. Note that this mapping  is twice Fr\'echet-differentiable and 
	\begin{align*}
		&\Psi'(c)(h)=p\int_\mathcal{M}c^{p-1}(x)h(x)dx, \quad \forall c, h\in H^2(\bo),\\
		&\Psi^{''}(c)(h,k)=p(p-1)\int_\mathcal{M}c^{p-2}(x)h(x)k(x)dx, \quad \forall c,h,  k\in H^2(\bo).
	\end{align*}
	Applying the  It\^o formula to the process $t\mapsto\Psi(c(t))$, yields
	\begin{align}
		\Psi(c(t))-\Psi(c(0))&=\int_0^t\Psi'(c(s))\left(\bu(s)\cdot \nabla c(s)+\xi\Delta c(s) -n(s)f(c(s)) \right)ds\notag\\
		&\quad{}+\frac{1}{2}\int_0^t\sum_{k=1}^2\Psi^{''}(c(s))\left(\gamma\phi_k(c(s)),\gamma\phi_k(c(s))\right)ds+\gamma \sum_{k=1}^2\int_0^t\Psi'(c(s))(\phi_k(c(s)))d\bar{ \beta}^k_s.\label{2.28}
	\end{align}
By integration-by-parts, the divergence free property of $\sigma_k$ and the fact that $\sigma_k=0$ on $\partial\bo$, we remark that for all $k\geq1$,
\begin{align}
	\Psi'(c)(\phi_k(c))&=p\int_\mathcal{O}c^{p-1}(x)\sigma_k(x)\cdot\nabla c(x)dx\notag\\
	&=\int_\mathcal{O}\sigma_k(x)\cdot\nabla c^p(x)dx\label{2.29*}\\
	&=-\int_\mathcal{O}c^p(x)\nabla\cdot\sigma_k(x)dx+\int_{\partial\mathcal{O}} c^p(\sigma)\sigma_k(\sigma)\cdot\nu d\sigma=0.\notag
\end{align}
This implies that the stochastic term in \re{2.28} vanishes.

Note that
\begin{equation}
	\int_\mathcal{O}\Delta c(x)c^{p-1}(x)dx=-(p-1)\int_\mathcal{O}\abs{\nabla c(x)}^2c(x)^{p-2}dx.\label{2.30*}
\end{equation}
Since $\nabla\cdot\bu=0$, by integration by part, we infer that
\begin{align}
	\int_\mathcal{O}\bu(x)\cdot\nabla c(x) c^{p-1}(x)dx=\frac{1}{p}\int_\mathcal{O}\bu(x)\cdot\nabla c^p(x)dx=0.\label{2.31*}
\end{align}
Using the equalities (\ref{2.29*}), (\ref{2.30*}) and  (\ref{2.31*}), we deduce from \re{2.29*} that
\begin{align}
	\Psi(c(t))-\Psi(c_0)&=\int_0^{t}\int_\mathcal{O}\left(-p(p-2)\xi\abs{\nabla c(s,x)}^2c^{p-2}(s,x)-pn(s,x)f(c(s,x))c^{p-1}(s,x) \right)dxds\notag
	\\
	&\qquad{}+\frac{p(p-1)}{2}\int_0^{t}\int_\mathcal{O}c^{p-2}(s)\sum_{k=1}^2\sigma_k(x)\cdot\nabla c(s,x)\sigma_k(x)\cdot\nabla c(s,x)dxds.\label{2.32*}
\end{align}
From the equality \re{2.24*}, we get $\sum_{k=1}^2\sigma_k\cdot\nabla c\sigma_k\cdot\nabla c=\abs{\nabla c}^2$. Hence, the equality \re{2.32*} becomes
	\begin{equation*}
		\Psi(c(t))-\Psi(c_0)=\int_0^t\int_\mathcal{O}\left(-p(p-2)\abs{\nabla c(s,x)}^2c^{p-2}(s,x)-pn(s,x)f(c(s,x))c^{p-1}(s,x) \right)dxds.
	\end{equation*}
	Using the non-negative property of  $n$ and $c$ proved in Lemma \ref{lem3.6} combined with the non-negativity of the function $f$,  we infer from the last equality  that for all  $p\geq 2$ and $t \in [0,T]$, $\abs{c(t)}_{L^p}\leq\abs{c_0}_{L^p}$, which along with the passage to the limit  $p\to +\infty$  completes the  proof of  Theorem \ref{lem2.2} (see \cite[Theorem 2.14]{Adam} for a detailed proof). 
\end{proof}
We now proceed with the statement and proof of the pathwise uniqueness of the weak solution.
\begin{proposition}\label{lem5.1}
	We assume that the assumptions of Theorem \ref{theo4.2}  hold. If $$(\Omega,\mathcal{F},\{\mathcal{F}_t\}_{t\in[0,T]},\mathbb{P},(\bu_1,c_1,n_1),(\bar{W},\bar{\beta})) \text{  and  } (\Omega,\mathcal{F},\{\mathcal{F}_t\}_{t\in[0,T]},\mathbb{P},(\bu_2,c_2,n_2),(\bar{W},\bar{\beta}))$$ are  two weak probabilistic solutions of system \re{3.1} with the same initial data $(\bu_0,c_0,n_0)$, then
	\begin{equation}
	(\bu_1(t),c_1(t),n_1(t))=(\bu_2(t),c_2(t),n_2(t))\qquad \mathbb{P}\text{-a.s.}\qquad \text{for all } t\in[0,T].
	\end{equation}
\end{proposition}
\begin{proof}
For $t\in [0,T]$,	let $$(\bw(t),\psi(t),\varphi(t))=(\bu_1(t)-\bu_2(t),c_1(t)-c_2(t),n_1(t)-n_2(t)).$$ Then this process satisfies  $(\bw(0),\psi(0),\varphi(0))=0$ and for all $t\in[0,T]$, we have
	\begin{align}
	\bw(t) &+\Int_0^t[\eta A_0\bw(s)+B_0(\bw(s),\bu_1(s))+B_0(\bu_2(s),\bw(s))]ds \notag\\
		&=\Int_0^tR_0(\varphi(s),\varPhi)ds+ \Int_0^t[g(\bu_1(s),c_1(s))-g(\bu_2(s),c_2(s))] dW_s,\label{5.4}
	\end{align}
	\begin{align}
	\psi(t) &+ \Int_0^t[\xi A_1\psi(s)+B_1(\bw(s), c_1(s))+B_1(\bu_2(s), \psi(s))]ds\notag\\
	&= -\Int_0^t [R_1(n_1(s),c_1(s))-R_1(n_2(s),c_2(s))]ds+\gamma\int_0^t\phi(\psi(s))d\beta_s,\label{5.5}
	\end{align}
	\begin{align}\label{5.6}
	\varphi(t) &+\Int_0^t[\delta A_1\varphi(s)+B_1(\bw(s), n_1(s))+B_1(\bu_2(s), \phi(s))]ds\notag\\
	& =-\Int_0^t [R_2(n_1(s),c_1(s))-R_2(n_2(s),c_2(s))]ds.
	\end{align}
	Using the fact that $(B_0(\bu_2,\bw),\bw)=0$,  we get by applying the It\^o formula to $t\mapsto\abs{\bw(t)}_{L^2}^2$ that
	\begin{align}
	\abs{\bw(t)}_{L^2}^2+2\eta\int_0^t\abs{\nabla\bw(s)}^2_{L^2}ds&= -2\int_0^t(B_0(\bw(s),\bu_1(s)),\bw(s))ds +2\int_0^t(R_0(\varphi(s),\varPhi),\bw(s))ds\notag\\
	&+\int_0^t\abs{g(\bu_1(s),c_1(s))-g(\bu_2(s),c_2(s))}^2_{\mathcal{L}^2(\buc,H)}ds\label{5.7}\\
	&+ 2\int_0^t(g(\bu_1(s),c_1(s))-g(\bu_2(s),c_2(s)),\bw(s)) dW_s.\notag
	\end{align}
Using the continuous embeddings $V\hookrightarrow H$ and $H^1(\bo)\hookrightarrow L^4(\bo)$ as well as the  H\"older inequality and the Young inequality, we derive that
	\begin{align}
	2\abs{(B_0(\bw,\bu_1),\bw)}&\leq 2\abs{\bw}_{L^4}\abs{\bu_1}_{L^4}\abs{\bw}_{L^2}\notag\\
	&\leq \frac{\eta}{5}\abs{\nabla \bw}^2_{L^2}+\bk\abs{\nabla \bu_1}_{L^2}^2\abs{\bw}^2_{L^2},\label{5.8}
	\end{align}
and
	\begin{align}
	2\abs{(R_0(\varphi,\varPhi),\bw)}&\leq 2\abs{\nabla\varPhi}_{L^\infty}\abs{\varphi}_{L^2}\abs{\bw}_{L^2}\notag\\
	&\leq \bk\abs{\nabla\varPhi}_{L^\infty}\abs{\varphi}_{L^2}\abs{\nabla \bw}_{L^2}\label{5.9}\\
	&\leq  \frac{\eta}{5}\abs{\nabla \bw}^2_{L^2}+\bk\abs{\varPhi}^2_{W^{1,\infty}}\abs{\varphi}^2_{L^2}.\notag
	\end{align}
	Thanks to \re{5.2}, we have
	\begin{equation}
	\abs{g(\bu_1,c_1)-g(\bu_2,c_2)}^2_{\mathcal{L}^2(\buc,H)}\leq L^2_{Lip}(\abs{\bw}^2_{L^2}+\abs{\psi}^2_{H^1}).\label{5.10}
	\end{equation}
	Since $\nabla\cdot\sigma_1=\nabla\cdot\sigma_2=0$, we obtain $(\phi(\psi),\psi)=0$. Futhermore,  by the fact that $\nabla\cdot \bu_2=0$, we  derive that  $(B_1(\bu_2, \psi),\psi)=0$.  Next, we recall that ($\mathbf{A}_3$) implies
	\begin{equation*}
	\abs{\phi(\psi)}^2_{\mathcal{L}^2(\mathbb{R}^2;L^2)}=\sum_{k=1}^2\int_\bo\abs{\sigma_k(x)\cdot\nabla \psi(x)}^2dx=\abs{\nabla \psi}^2_{L^2}.
	\end{equation*}
Hence, by applying the It\^o formula to  $t\mapsto\abs{\psi(t)}^2_{H^1}$, we see that
	\begin{align}
	&\abs{\psi(t)}_{H^1}^2+2\int_0^t\left(\mu\abs{\nabla\psi(s)}^2_{L^2}+\xi\abs{A_1\psi(s)}^2_{L^2}\right)ds\notag\\
	&=-2\int_0^t(B_1(\bw(s), c_1(s)),\psi(s))ds-2\int_0^t (R_1(n_1(s),c_1(s))-R_1(n_2(s),c_2(s)),\psi(s))ds\notag\\
	&\qquad{}+2\int_0^t(B_1(\bw(s), c_1(s))+B_1(\bu_2(s), \psi(s)),A_1\psi(s))ds\label{5.11}\\
	&\qquad{}-2\int_0^t (R_1(n_1(s),c_1(s))-R_1(n_2(s),c_2(s)),A_1\psi(s))ds\notag\\
	&\qquad{}+\gamma^2\int_0^t\abs{\nabla\phi(\psi(s))}^2_{\mathcal{L}^2(\mathbb{R}^2;L^2)}ds+2\gamma\int_0^t(\nabla\phi(\psi(s)),\nabla\psi(s))d\beta_s.\notag
	\end{align}
	Taking  the $L^2$-inner product of the  equation \re{5.6}  with $\varphi$ and adding the result to \re{5.11},  yield
	\begin{align}
	&\abs{\varphi(t)}_{L^2}^2+\abs{\psi(t)}_{H^1}^2+2\int_0^t(\mu\abs{\nabla\psi(s)}^2_{L^2}+\xi\abs{A_1\psi(s)}^2_{L^2}+\delta\abs{\nabla\varphi(s)}^2_{L^2})ds\notag\\
	&=-2\int_0^t(B_1(\bw(s), c_1(s)),\psi(s))ds-2\int_0^t (R_1(n_1(s),c_1(s))-R_1(n_2(s),c_2(s)),\psi(s))ds\notag\\
	&\quad{}+2\int_0^t(B_1(\bw(s), c_1(s))+B_1(\bu_2(s), \psi(s)),A_1\psi(s))ds\notag\\
	&\quad{}-2\int_0^t (R_1(n_1(s),c_1(s))-R_1(n_2(s),c_2(s)),A_1\psi(s))ds\label{5.12}\\
	&\quad{} -2\int_0^t [r_2(\varphi(s),c_1(s),\varphi(s))+r_2(n_2(s),\psi(s),\varphi(s))]ds+\gamma^2\int_0^t\abs{\nabla\phi(\psi(s))}^2_{\mathcal{L}^2(\mathbb{R}^2;L^2)}ds\notag\\
	&\quad{}-2\int_0^t(B_1(\bw(s), n_1(s)),\varphi(s))ds+2\gamma\int_0^t(\nabla\phi(\psi(s)),\nabla\psi(s))d\beta_s.\notag
	\end{align}
Now, we give an estimate for the right-hand side of \eqref{5.12}.
Similarly to \eqref{5.8}, we have
	\begin{align}
	2\abs{(B_1(\bw, c_1),\psi)}&\leq 2\abs{\bw}_{L^4}\abs{\nabla c_1}_{L^2}\abs{\psi(s)}_{L^4}\notag\\
	&\leq\bk\abs{\nabla\bw}_{L^2}\abs{\nabla c_1}_{L^2}\abs{\psi}_{H^1}\label{5.13}\\
	&\leq \frac{\eta}{5}\abs{\nabla\bw}^2_{L^2}+\bk\abs{\nabla c_1}^2_{L^2}\abs{\psi}_{H^1}.\notag
	\end{align}
	Thanks to  the continuous embedding $H^1(\bo)\hookrightarrow L^4(\bo)$ and the $L^\infty$-stability property proved in Corollary \ref{lem3.7},  we have
	\begin{align}
	2(R_1(n_1,c_1)-R_1(n_2,c_2),\psi)
&\leq 2\abs{R_1(n_1,c_1)-R_1(n_2,c_2)}_{L^2}\abs{\psi}_{L^2}\notag\\
	&\leq 4\abs{(f(c_1)-f(c_2))n_1}^2_{L^2}+4\abs{f(c_2)\psi}^2_{L^2}+2\abs{\psi}^2_{L^2}\notag\\
	&\leq 4\sup_{0\leq r\leq \abs{c_0}_{L^\infty}}(f'(r))^2\abs{n_1\psi}^2_{L^2}+4\sup_{0\leq r\leq \abs{c_0}_{L^\infty}}f(r)\abs{\psi}^2_{L^2}+2\abs{\psi}^2_{L^2}\notag\\
	&\leq \bk\abs{\psi}_{L^4}^2\abs{n_1}^2_{L^4}+\bk_f\abs{\psi}^2_{L^2}\notag.
	\end{align}
Applying  the Galiardo-Nirenberg-Sobolev  inequality, we arrive at
\begin{align}
2(R_1(n_1,c_1)-R_1(n_2,c_2),\psi)&\leq\bk\abs{\psi}^2_{H^1}\left(\abs{\nabla n_1}_{L^2}\abs{n_1}_{L^2}+\abs{n_1}_{L^2}^2\right)+\bk_f\abs{\psi}^2_{L^2}\notag\\
	&\leq \bk\left(\abs{\nabla n_1}_{L^2}\abs{n_1}_{L^2}+\abs{n_1}_{L^2}^2\right)\abs{\psi}^2_{H^1}+\bk_f\abs{\psi}^2_{H^1}.\label{5.14}
\end{align}
	Thanks to  the Ladyzhenskaya, Galiardo-Nirenberg-Sobolev, and  Young inequalities, we find that
	\begin{align}
	2\abs{(B_1(\bw, c_1),A_1\psi)}
&\leq 2\abs{\bw}_{L^4}\abs{\nabla c_1}_{L^4}\abs{A_1\psi}_{L^2}\notag\\
	&\leq \frac{\xi}{6}\abs{A_1\psi}_{L^2}^2+\bk\abs{\bw}_{L^2}\abs{\nabla\bw}_{L^2}\left(\abs{c_1}_{H^2}\abs{\nabla c_1}_{L^2}+\abs{\nabla c_1}_{L^2}^2\right)\label{5.15}\\
	&\leq \frac{\xi}{6}\abs{A_1\psi}_{L^2}^2+ \frac{\eta}{5}\abs{\nabla\bw}_{L^2}^2+\bk\left(\abs{c_1}^2_{H^2}\abs{\nabla c_1}_{L^2}^2+\abs{\nabla c_1}_{L^2}^4\right)\abs{\bw}_{L^2}^2.\notag
	\end{align}
	We recall that there exist a positive constant $\bk_0$, such that $\abs{\psi}_{H^2}^2\leq \bk_0(\abs{A_1\psi}^2+\abs{\psi}_{H^1}^2)$. Hence, using also the continuous embedding $V\hookrightarrow H$, we obtain
	\begin{align}
	2\abs{(B_1(\bu_2, \psi),A_1\psi)}
 	&\leq 2\abs{\bu_2}_{L^4}\abs{\nabla \psi}_{L^4}\abs{A_1\psi}_{L^2}\notag\\
	&\leq \frac{\xi}{6}\abs{A_1\psi}_{L^2}^2+\bk\abs{\bu_2}_{L^2}\abs{\nabla\bu_2}_{L^2}\left(\abs{\psi}_{H^2}\abs{\nabla\psi}_{L^2}+\abs{\nabla\psi}_{L^2}^2\right)\notag\\
	&\leq \frac{\xi}{6}\abs{A_1\psi}_{L^2}^2+\frac{\bk_0^{-1}\xi}{6}\abs{\psi}_{H^2}^2+\bk\abs{\bu_2}^2_{L^2}\abs{\nabla\bu_2}^2_{L^2}\abs{\nabla\psi}^2_{L^2}\label{5.16}\\
	&\quad {}  +\bk\abs{\bu_2}_{L^2}\abs{\nabla\bu_2}_{L^2}\abs{\nabla\psi}_{L^2}^2\notag\\
	&\leq \frac{\xi}{3}\abs{A_1\psi}_{L^2}^2+\frac{\xi}{6}\abs{\psi}_{H^1}^2+\bk\left(\abs{\bu_2}^2_{L^2}\abs{\nabla\bu_2}^2_{L^2}+\abs{\nabla\bu_2}^2_{L^2}\right)\abs{\psi}_{H^1}^2.\notag
	\end{align}
	Using a similarly argument as in \re{5.14}, we arrive at
	\begin{align}
2\abs{(R_1(n_1,c_1)-R_1(n_2,c_2),A_1\psi}
	&\leq\frac{\xi}{6}\abs{A_1\psi}_{L^2}^2+\bk\abs{R_1(n_1,c_1)-R_1(n_2,c_2)}_{L^2}^2\notag\\
	&\leq \frac{\xi}{6}\abs{A_1\psi}_{L^2}^2+\bk\abs{\psi}_{L^4}^2\abs{n_1}^2_{L^4}+\bk_f\abs{\psi}^2_{L^2}\label{5.17}\\
	&\leq \frac{\xi}{6}\abs{A_1\psi}_{L^2}^2+\bk_f\abs{\psi}^2_{H^1}+\bk\left(\abs{\nabla n_1}_{L^2}\abs{n_1}_{L^2}+\abs{n_1}_{L^2}^2\right)\abs{\psi}^2_{H^1}.\notag
	\end{align}
By 	using an integration-by-parts and H\"older, and the Galiardo-Nirenberg-Sobolev inequalities, we see that
	\begin{align}
	2\abs{(B_1(\bw, n_1),\varphi)}
&\leq2\abs{\int_\bo n_1(x)\bw(x)\cdot\nabla\varphi(x) dx}\notag\\
	&\leq 2\abs{n_1}_{L^4}\abs{\bw}_{L^4}\abs{\nabla\varphi}_{L^2}\notag\\
	&\leq\frac{\delta}{4}\abs{\nabla\varphi}_{L^2}^2+\bk\abs{\bw}_{L^2}\abs{\nabla\bw}_{L^2}\left(\abs{\nabla n_1}_{L^2}\abs{n_1}_{L^2}+\abs{n_1}_{L^2}^2\right)\label{5.18}\\
	&\leq\frac{\delta}{4}\abs{\nabla\varphi}_{L^2}^2+ \frac{\eta}{5}\abs{\nabla\bw}^2_{L^2}+\bk\left(\abs{\nabla n_1}^2_{L^2}\abs{n_1}^2_{L^2}+\abs{n_1}_{L^2}^4\right)\abs{\bw}^2_{L^2}\notag.
	\end{align}
By applying  the Young and Galiardo-Nirenberg-Sobolev  inequalities we obtain
	\begin{align}
	2\abs{r_2(\varphi,c_1,\varphi)}
	&\leq 2\abs{\varphi}_{L^4}\abs{\nabla c_1}_{L^4}\abs{\nabla\varphi}_{L^2}\notag\\
	&\leq \frac{\delta}{4}\abs{\nabla\varphi}_{L^2}^2+\bk \abs{\varphi}_{L^4}^2\abs{\nabla c_1}_{L^4}^2\notag\\
	&\leq\frac{\delta}{4}\abs{\nabla\varphi}_{L^2}^2+\bk\left(\abs{\nabla\varphi}_{L^2}\abs{\varphi}_{L^2}+\abs{\varphi}_{L^2}^2\right)\left(\abs{c_1}_{H^2}\abs{\nabla c_1}_{L^2}+\abs{\nabla c_1}_{L^2}^2\right)\label{5.19}\\
	&\leq \frac{\delta}{2}\abs{\nabla\varphi}_{L^2}^2+\bk\left(\abs{c_1}^2_{H^2}\abs{\nabla c_1}_{L^2}^2+\abs{\nabla c_1}_{L^2}^4+\abs{c_1}_{H^2}\abs{\nabla c_1}_{L^2}+\abs{\nabla c_1}_{L^2}^2\right)\abs{\varphi}_{L^2}^2\notag.
	\end{align}
In a similarly way we have that
	\begin{align}
2\abs{r_2(n_2,\psi,\varphi)}
	&\leq 2\abs{n_2}_{L^4}\abs{\nabla \psi}_{L^4}\abs{\nabla\varphi}_{L^2}\notag\\
	&\leq \frac{\delta}{4}\abs{\nabla\varphi}_{L^2}^2+\bk\abs{n_2}^2_{L^4}\left(\abs{\psi}_{H^2}\abs{\nabla \psi}_{L^2}+\abs{\nabla \psi}_{L^2}^2\right)\notag\\
	&\leq \frac{\delta}{4}\abs{\nabla\varphi}_{L^2}^2+\frac{\bk_0^{-1}\xi}{6}\abs{\psi}_{H^2}^2+\bk\abs{n_2}^4_{L^4}\abs{\nabla \psi}_{L^2}^2+\bk\abs{n_2}^2_{L^4}\abs{\nabla \psi}_{L^2}^2\label{5.20}\\
	&\leq\frac{\delta}{4}\abs{\nabla\varphi}_{L^2}^2+\frac{\xi}{6}\abs{A_1\psi}_{L^2}^2+\frac{\xi}{6}\abs{ \psi}_{H^1}^2+\bk\abs{n_2}_{L^2}^4\abs{ \psi}_{H^1}^2\notag\\
	& \quad{}+\bk\left(\abs{\nabla n_2}_{L^2}\abs{n_2}_{L^2}+\abs{n_2}_{L^2}^2\abs{\nabla n_2}^2_{L^2}+\abs{n_2}^2_{L^2}\right)\abs{ \psi}_{H^1}^2.\notag
	\end{align}
By	using \re{4.46*} we derive that
	\begin{align}
	\gamma^2\abs{\nabla\phi(\psi)}^2_{\mathcal{L}^2(\mathbb{R}^2;L^2)}&=\gamma^2\sum_{k=1}^2\int_\bo\abs{\nabla(\sigma_k(x)\cdot\nabla \psi(x))}^2dx\notag\\
	&\leq2\gamma^2\sum_{k=1}^2\abs{\sigma_k}^2_{W^{1,\infty}}\abs{\nabla\psi}^2_{L^2}+2\gamma^2\sum_{k=1}^2\abs{\sigma_k}^2_{L^{\infty}}\abs{\psi}^2_{H^2}\label{5.21}\\
	&\leq (1+\bk_0)2\gamma^2\abs{\sigma}^2_{W^{1,\infty}}\abs{\nabla\psi}^2_{L^2}+2\gamma^2\bk_0\abs{\sigma}^2_{L^{\infty}}\abs{A_1\psi}^2_{L^2}\notag\\
	&\leq\frac{\xi}{6}\abs{A_1\psi}^2_{L^2}+(1+\bk_0)2\gamma^2\abs{\sigma}^2_{W^{1,\infty}}\abs{\psi}^2_{H^1}\notag.
	\end{align}
Now, for $t\in [0,T]$ and $s\in [0,t]$,  let us set $$\mathcal{Y}(t):=\abs{\bu(t)}^2_{L^2}+\abs{c(t)}^2_{H^1}+\abs{\varphi(t)}^2_{L^2},$$
	\begin{align}
	\bz(s)&:=\bk\abs{\nabla u_1(s)}_{L^2}^2+\bk\abs{\nabla c_1(s)}^2_{L^2}+\bk\left(\abs{\nabla n_1(s)}_{L^2}\abs{n_1(s)}_{L^2}+\abs{n_1(s)}_{L^2}^2\right)\notag\\
	&\qquad{}+\bk\left(\abs{c_1(s)}^2_{H^2}\abs{\nabla c_1(s)}_{L^2}^2+\abs{\nabla c_1(s)}_{L^2}^4\right)+\bk\left(\abs{\bu_2(s)}^2_{L^2}\abs{\nabla\bu_2(s)}^2_{L^2}+\abs{\nabla\bu_2(s)}^2_{L^2}\right)\notag\\
	&\qquad{}+\bk\left(\abs{\nabla n_1(s)}_{L^2}\abs{n_1(s)}_{L^2}+\abs{n_1(s)}_{L^2}^2\right)+\bk\left(\abs{\nabla n_1(s)}^2_{L^2}\abs{n_1(s)}^2_{L^2}+\abs{n_1(s)}_{L^2}^4\right)\label{5.22}\\
	&\qquad{}+\bk\left(\abs{c_1(s)}^2_{H^2}\abs{\nabla c_1(s)}_{L^2}^2+\abs{\nabla c_1(s)}_{L^2}^4+\abs{c_1(s)}_{H^2}\abs{\nabla c_1(s)}_{L^2}+\abs{\nabla c_1(s)}_{L^2}^2\right)\notag
\\
	&\qquad{}+\bk\left(\abs{\nabla n_2(s)}_{L^2}\abs{n_2(s)}_{L^2}+\abs{n_2(s)}_{L^2}^2\abs{\nabla n_2(s)}^2_{L^2}+\abs{n_2(s)}^2_{L^2}+\abs{n_2(s)}_{L^2}^4\right),\notag
	\end{align}
and
	$$\theta(t):=\exp\left(-\int_0^t\bz(s)ds\right).$$
	Applying the It\^o formula to  $t\mapsto\theta(t)\abs{\bu(t)}^2_{L^2}$, we derive that
	\begin{align}
	\theta(t)\abs{\bw(t)}_{L^2}^2+2\eta\int_0^t\theta(s)\abs{\nabla\bw(s)}^2_{L^2}ds&\leq 2\int_0^t\theta(s)(B_0(\bw(s),\bu_1(s)),\bw(s))ds\notag\\
	& +2\int_0^t\theta(s)(R_0(\varphi(s)),\bw(s))ds+\int_0^t\theta'(s)\abs{\bw(s)}_{L^2}^2ds\notag\\
	&+\int_0^t\theta(s)\abs{g(\bu_1(s),c_1(s))-g(\bu_2(s),c_2(s))}^2_{\mathcal{L}^2(\buc,H)}ds\label{5.23}\\
	&+ 2\int_0^t\theta(s)(g(\bu_1(s),c_1(s))-g(\bu_2(s),c_2(s)),\bw(s)) dW_s.\notag
	\end{align}
	Applying the It\^o formula once more to  $t\mapsto\theta(t)(\abs{\varphi(t)}^2_{L^2}+\abs{\psi(t)}^2_{H^1})$ and adding the result with \re{5.23} after taking into account the estimates \re{5.8}-\re{5.10} and \re{5.13}-\re{5.21}, we arrive at
	\begin{align}
	\theta(t)\mathcal{Y}(t)&+\int_0^t\theta(s)\left(\eta\abs{\nabla\bw(s)}^2_{L^2}+\mu\abs{\nabla\psi(s)}^2_{L^2}+\xi\abs{A_1\psi(s)}^2_{L^2}\right)ds\notag\\
	&\leq \left( \bk\abs{\varPhi}^2_{W^{1,\infty}}+L_{Lip}^2+2\bk_f+\frac{\xi}{3}+(1+\bk_0)2\gamma^2\abs{\sigma}^2_{W^{1,\infty}}\right)\int_0^t\theta(s)\mathcal{Y}(s)ds\label{5.24}\\
	&\qquad{}+2\gamma\int_0^t\theta(s)(\nabla\phi(\psi(s)),\nabla\psi(s))d\beta_s\notag\\
	&\qquad{}+ 2\int_0^t\theta(s)(g(\bu_1(s),c_1(s))-g(\bu_2(s),c_2(s)),\bw(s)) dW_s.\notag
	\end{align}
	Next, taking the mathematical expectation yields
	\begin{align}
	\be\theta(t)\mathcal{Y}(t)&+\be\int_0^t\theta(s)\left(\eta\abs{\nabla\bw(s)}^2_{L^2}+\mu\abs{\nabla\psi(s)}^2_{L^2}+\xi\abs{A_1\psi(s)}^2_{L^2}\right)ds\notag\\
	&\leq \left( \bk\abs{\varPhi}^2_{W^{1,\infty}}+L_{Lip}^2+2\bk_f+\frac{\xi}{3}+(1+\bk_0)2\gamma^2\abs{\sigma}^2_{W^{1,\infty}}\right)\be\int_0^t\theta(s)\mathcal{Y}(s)ds\label{5.25}.
	\end{align}
	From which along with the Gronwall inequality we infer that for any $t\in[0,T]$ $$\be\theta(t)\mathcal{Y}(t)=0.$$
	It follows that for all $t \in [0,T ]$,  $\mathcal{Y}(t) = 0$ $\mathbb{P}$-a.s. Since the paths of $(\bu_i, c_i,n_i)$, $i=1,2$ are continuous $\mathbb{P}$-a.s., then   $$(\bu_1(t),c_1(t),n_1(t)) = (\bu_2(t),c_2(t),n_2(t)),\ \  \mathbb{P}\text{-a.s., for all } t \in [0,T ].$$
\end{proof}

With the existence and pathwise uniqueness results at hand we now prove the existence of strong solution stated in Theorem \ref{theo4.2}.

\begin{proof}[\textbf{Proof of Theorem \ref{theo4.2}.} ]
The existence of a probabilistic weak solution to the problem \re{1.1} is shown in Proposition \ref{theo4.1}. The pathwise uniqueness of probabilistic weak solutions is given by Proposition \ref{lem5.1}. Thus, the existence and uniqueness of a probabilistic strong solution to the problem \re{1.1} follows from the Yamada-Watanabe Theorem  (see \cite[Theorem E.1.8]{Pre}), which states that the existence of weak probabilistic solution and the pathwise uniqueness imply the existence of a unique probabilistic strong solution.
\end{proof}

\section{Proof of Proposition \ref{theo4.1}} \label{sect5}
In this section, we will show  Proposition \ref{theo4.1}. We introduce a Galerkin approximation first. We then discuss the existence of the Galerkin approximation and prove the mass conservation property, the  non-negativity property  and the $L^\infty$-norm satibility in finite dimension.  Using these properties, we prove  priori estimates and by these a priori estimates, we show the tightness of the family of approximations,
and pass in a second step, to the limit in the deterministic terms and the  construction of  the noise terms by exploiting the usual martingale representation theorem proved in \cite[Theorem 8.2]{Da}.

\subsection{Galerkin approximation and a priori uniform estimates}

In this subsection, we will construct a family of approximations of the solutions and prove some crucial estimates satisfied uniformly by the approximations.  For this propose, let us recall that there exists an orthonormal basis $\{\bw_i\}_{i=1}^\infty$ of $H$ consisting of the eigenfunctions of the Stokes operator $A_0$ and  an orthonormal basis $\{\varphi_i\}_{i=1}^\infty\subset \mathcal{C}^\infty(\bo)$ of $L^2(\bo)$ consisting of  the eigenfunctions of the Neumann Laplacian operator $A_1$. For $m\in \mathbb{N}$, we will consider the following finite-dimensional spaces
\begin{equation*}
	\bh_m=\text{spam}\{\bw_1,...,\bw_m \},\qquad H_m=\text{spam}\{\varphi_1,...,\varphi_m \},\qquad \bhc_m=\bh_m\times H_m\times H_m,
\end{equation*}
where we endow $\bhc_m$ with the following norm $$\abs{(\bu,c,n)}^2_{\bhc_m}=\abs{\bu}^2_{L^2}+\abs{c}^2_{L^2}+\abs{n}^2_{L^2}, \quad (\bu,c,n)\in \bhc_m.$$  Owing to the fact that $\bhc_m$ is a finite dimensional space, the $L^2(\bo)$, $H^1(\bo)$ and $H^2(\bo)$-norms are equivalent on this space.
We choose as in \cite[P. 335]{Zhai} $n_0^m$, $c_0^m$ and $\bu_0^m$ such that
\begin{equation}\label{4.42}
	\begin{split}
		&n_0^m>0,\ n_0^m\to n_0 \text{  in  } L^2(\bo), \ n_0^m\ln n_0^m\to n_0\ln n_0\text{  in  } L^1(\bo),\\
		&c_0^m>0,\ \abs{c_0^m}_{L^\infty}\leq\abs{c_0}_{L^\infty},\  c_0^m\to c_0 \text{  in  } H^1(\bo),\\
		&\text{ and }\bu_0^m\to \bu_0 \text{  in  } H.
	\end{split}
\end{equation}

We then consider on the filtered probability space $(\Omega,\mathcal{F},\{\mathcal{F}_t\}_{t\in[0,T]},\mathbb{P})$  the following finite dimensional problem. For all $t\in [0,T]$
\begin{equation}\label{3.40}
	\begin{split}
		&\bu_m(t) +\Int_0^t[\eta A_0\bu_m(s)+\bp_m^1B_0(\bu_m(s),\bu_m(s))]ds\\
		&=\bu_0^m+\Int_0^t\bp_m^1R_0(n_m(s),\varPhi)ds + \Int_0^t\bp_m^1g(\bu_m(s),c_m(s)) dW_s,\\
		&	c_m(t) + \Int_0^t[\xi A_1c_m(s)+\bp_m^2B_1(\bu_m(s), c_m(s))]ds\\
		&= c_0^m-\Int_0^t \bp_m^2R_1(n_m(s),c_m(s))ds+\gamma\int_0^t\bp_m^2\phi(c_m(s))d\beta_s,\\
		&	n_m(t) +\Int_0^t[\delta A_1n_m(s)+\bp_m^2B_1(\bu_m(s),n_m(s))]ds =n_0^m-\Int_0^t \bp_m^2R_2(n_m(s),c_m(s))ds,
	\end{split}
\end{equation}
where $\bp_m^1$ and $\bp_m^2$ are the projection from $H$ and 	$L^2(\bo)$ onto $\bh_m$ and $H_m$, respectively, and their operator norms are equal to $1$.

For each $m$, we consider the following mapping $\Psi_m:\bhc_m\to \bhc_m$ defined by
\begin{equation*}
	\Psi_m(\bu,c,n)=	\begin{pmatrix}
		\eta A_0\bu+\bp_m^1B_0(\bu,\bu)-\bp_m^1R_0(n,\varPhi) \\
		\xi A_1c+\bp_m^2B_1(\bu, c)+\bp_m^2R_1(n,c) \\
		\delta A_1n+\bp_m^2B_1(\bu,n)+\bp_m^2R_2(n,c)
	\end{pmatrix}.
\end{equation*}
In the following lemma, we are going to state an important property of  the  mappings $\Psi_m$, $m\in \mathbb{N}$.
\begin{lemma}\label{lemma3.3}
Let Assumption \ref{ass_1} and  Assumption  \ref{ass_3}	be satisfied.  For each $m\in\mathbb{N}$, the mapping $\Psi_m$ is locally Lipschitz continuous. To be more precise, for each $m\in\mathbb{N}$ and every $r > 0$, there exists a constant $\bk_r$ such that
	\begin{equation}\label{3.41}
		\abs{\Psi_m(\bv_1)-\Psi_m(\bv_2)}_{\bhc_m}\leq \bk_r\abs{\bv_1-\bv_2}_{\bhc_m},
	\end{equation}
	for $\bv_1=(\bu_1,c_1,n_1)$, $\bv_2=(\bu_2,c_2,n_2)\in \bhc_m$ with $\abs{\bv_i}_{\bhc_m} \leq r$, $i=1,2$.
\end{lemma}
\begin{proof}
	Let $\bv_1=(\bu_1,c_1,n_1)$, $\bv_2=(\bu_2,c_2,n_2)\in\bhc_m$ and $\bv=(\bu,c,n)\in\bhc_m$. We assume that $\abs{\bv_i}_{\bhc_m} \leq r$, $i=1,2$. We have
	\begin{align}
		\left(\Psi_m(\bv_1)-\Psi_m(\bv_2),\bv \right)_{\bhc_m}	&=(\eta A_0(\bu_1-\bu_2)+B_0(\bu_1,\bu_1)-B_0(\bu_2,\bu_2)-R_0(n_1,\varPhi)+R_0(n_2,\varPhi),\bu)\notag\\
		&+(\xi A_1(c_1-c_2)+B_1(\bu_1, c_1)-B_1(\bu_2, c_2)+R_1(n_1,c_1)-R_1(n_2,c_2),c)\notag\\
		&+(\delta A_1(n_1-n_2)+B_1(\bu_1,n_1)-B_1(\bu_2,n_2)+R_2(n_1,c_1)-R_2(n_2,c_2),n).\label{3.42}
	\end{align}
	Using the bilinearity  of the operator $B_0$, we see that
	\begin{align}
		\abs{(B_0(\bu_1,\bu_1)-B_0(\bu_2,\bu_2),\bu)}&\leq\abs{(B_0(\bu_1-\bu_2,\bu_1),\bu)}+\abs{(B_0(\bu_2,\bu_1-\bu_2),\bu)}\notag\\
		&\leq 2\bk_r\abs{\bu_1-\bu_2}_{L^2}\abs{\bu}_{L^2}.\notag
	\end{align}
	By the H\"older inequality we also note that
	\begin{align}
		(R_0(n_1,\varPhi)-R_0(n_2,\varPhi),\bu)&\leq \int_\bo\abs{n_1-n_2}\abs{\nabla\varPhi}\abs{\bu}dx\notag\\
		&\leq \abs{\nabla\varPhi}_{L^\infty}\abs{n_1-n_2}_{L^2}\abs{\bu}_{L^2}.\notag
	\end{align}
	Since the space $H^1(\bo)$ is continuously embedded in the space  $L^q(\bo)$ for any $q\geq 2$, we have
	\begin{align}
		\abs{(B_1(\bu_1, c_1)-B_1(\bu_2, c_2),c)}&\leq \abs{(B_1(\bu_1-\bu_2, c_1),c)}+\abs{(B_1(\bu_2, c_1-c_2),c)}\notag\\
		&\leq \abs{\bu_1-\bu_2}_{L^4}\abs{\nabla c_1}_{L^2}\abs{c}_{L^4}+\abs{\bu_2}_{L^4}\abs{\nabla (c_1-c_2)}_{L^2}\abs{c}_{L^4}\notag\\
		&\leq (\abs{\nabla(\bu_1-\bu_2)}_{L^2}\abs{\nabla c_1}+\abs{\nabla\bu_2}_{L^2}\abs{\nabla (c_1-c_2)}_{L^2})\abs{c}_{H^1}\notag\\
		&\leq \bk_r(\abs{\nabla(\bu_1-\bu_2)}_{L^2}+\abs{\nabla (c_1-c_2)}_{L^2})\abs{c}_{H^1}.\notag
	\end{align}
In a similar way we show that
	\begin{equation}\
		\abs{(B_1(\bu_1,n_1)-B_1(\bu_2,n_2),n)}\leq \bk_r(\abs{\nabla(\bu_1-\bu_2)}_{L^2}+\abs{\nabla (n_1-n_2)}_{L^2})\abs{n}_{H^1}.\notag
	\end{equation}
	Owing to the fact that $H_m\subset\mathcal{C}^\infty(\bo)$ and  $f(0)=0$ as well as $f\in C^1([0,\infty))$, we derive that
	\begin{align}
		\abs{(R_1(n_1,c_1)-R_1(n_2,c_2),c)}&\leq \int_\bo\abs{n_1-n_2}f(c_1)\abs{c}dx+\int_\bo \abs{n_2}\abs{f(c_1)-f(c_2)}\abs{c}dx\notag\\
		&\leq \max_{0\leq c\leq\abs{c_1}_{L^\infty}}f(c) \int_\bo\abs{n_1-n_2}\abs{c}dx\notag\\
		&\qquad{}+\max_{0\leq c\leq\max(\abs{c_1}_{L^\infty},\abs{c_2}_{L^\infty})}f'(c)\int_\bo \abs{n_2}\abs{c_1-c_2}\abs{c}dx\notag\\
		&\leq \max_{0\leq c\leq r}f(c) \abs{n_1-n_2}_{L^2}\abs{c}_{L^2}+\max_{0\leq c\leq r}f' \abs{n_2}_{L^4}\abs{c_1-c_2}_{L^4}\abs{c}_{L^2}\notag\\
		&\leq \bk_r (\abs{n_1-n_2}_{L^2}+\abs{c_1-c_2}_{H^1})\abs{c}_{L^2}.\notag
	\end{align}
	Also, we note that
	\begin{align}
		\abs{(R_2(n_1,c_1)-R_2(n_2,c_2),n)}&\leq \int_\bo\abs{n_1-n_2}\abs{\nabla c_1}\abs{\nabla n}dx+\int_\bo\abs{n_2}\abs{\nabla (c_1-c_2)}\abs{\nabla n}dx\notag\\
		&\leq \abs{n_1-n_2}_{L^2}\abs{\nabla c_1}_{L^4}\abs{\nabla n}_{L^2}+\abs{n_2}_{L^4}\abs{\nabla (c_1-c_2)}_{L^4}\abs{\nabla n}_{L2}\notag\\
		&\leq \bk_r(\abs{n_1-n_2}_{L^2}+\abs{ c_1-c_2}_{H^2})\abs{ n}_{H^1}.\notag
	\end{align}
	Taking into account the fact that all norms are equivalent in finite dimensional space, and  the fact that the operators $A_0$ and $A_1$ are linear, we infer  these previous  inequalities and equality \re{3.42}.
\end{proof}
The existence of solutions to the finite dimensional problem \re{3.40} is classical. In fact, due to Lemma \ref{lemma3.3}, the mapping $\Psi_m$ is locally Lipschitz. Also by the inequality \re{2.5}, $\bp_m^1g(\cdot,\cdot)$ is locally Lipschitz. From the linearity of $\phi(.)$, we can easily see that $\bp_m^2\phi(\cdot)$ is  Lipschitz. Hence,  by  well known theory for finite dimensional 
stochastic differential equations with locally Lipschitz coefficients  (see  \cite[Theorem 38, P. 303]{Pro} for full details) there exists a local solution of system \re{3.40} with continuous paths in $\bhc_m$. That  is, there exists a stopping time $\tau_m$, a process  $t\mapsto(\bu_m(t),c_m(t),n_m(t))$ such that $\tau_m>0$ $\mathbb{P}$-a.s., and the stopped process $$t\mapsto(\bu_m(t\wedge\tau_m),c_m(t\wedge\tau_m),n_m(t\wedge\tau_m))$$ satisfies the system of It\^o equation  \re{3.40} and has  continuous paths in $\bhc_m$.  Moreover, if a process $$t\mapsto(\bar{\bu}_m(t),\bar{c}_m(t),\bar{n}_m(t)),$$ 
and a  stopping time  $\sigma_m$ constitute  another local solution, then
\begin{equation*}
(\bu_m(\cdot),c_m(\cdot),n_m(\cdot))=(\bar{\bu}_m(\cdot),\bar{c}_m(\cdot),\bar{n}_m(\cdot)),\ \ \mathbb{P}\text{-a.s. on } [0,\tau_m\wedge\sigma_m].
\end{equation*}
We will show in what follows that the solutions $(\bu_m,c_m,n_m)$ exist almost surely for every $t\in [0,T]$. For this goal, it will be enough to show that 
\begin{equation}
\tau_m(\omega)>T, \ \text{for almost  all } \omega\in \Omega, \text{and all } m\in\mathbb{N}.\label{4.10}
\end{equation}
To this aim, we will use some idea from \cite[P. 132, Proof of Theorem 12.1]{Rog}. Since for all $m\in \mathbb{N}$, the deterministic integrand $\Psi_m$ and  the stochastic integrand $\bp_m^1g$ are locally Lipschitz, for each $N\in \mathbb{N}$, we can define the integrands  $\Psi_m^N$ and $\bp_m^1g^N$, agreeing respectively with $\Psi_m$ and $\bp_m^1g$ on  the ball 
$$\mathbb{B}^N_{\bhc_m}:=\left\{ (\bv,\varphi,\psi)\in \bhc_m: \abs{(\bv,\varphi,\psi)}_{\bhc_m}< N \right\},$$
such that $\Psi_m^N$ and $\bp_m^1g^N$ are globally Lipschitz.  As consequence, since  $\bp_m^2\phi$ is already globally Lipschitz, \cite[P. 128, Theorem 11.2]{Rog} guarantees that there is  a unique solution $(\bu^N_m,c^N_m,n^N_m)$ to a system associated to the system  \re{3.40} with $\Psi_m^N$ and $\bp_m^1g^N$ (instead of $\Psi_m$ and $\bp_m^1g$) and defined on $[0,+\infty)$ almost surely. We then define a sequence of stopping times as follows
 for all $m,N\in\mathbb{N}$ 
\begin{equation}\label{3.46}
	\tau_N^m:=
	\inf \{ t>0: \sqrt{\abs{n^N_m(t)}^2_{L^2}+\abs{\bu^N_m(t)}_{L^2}^2+\abs{c^N_m(t)}_{H^1}^2}\geq N   \}\wedge N,
\end{equation}
where $a\wedge b:=\min\{a,b\}$ for any real numbers $a$ and $b$.

For any fixed $m\in \mathbb{N}$, the sequence $\{\tau_N^m\}_{N\in\mathbb{N}}$ is obviously increasing. Moreover \cite[P. 131, Corollary 11.10]{Rog} implies that  for all $N\in \mathbb{N}$,
\begin{equation}
(\bu_m,c_m,n_m)=(\bu^N_m,c^N_m,n^N_m)\text{ on } [0,\tau_N^m].\notag
\end{equation}
From this last equality, we infer that the solution $(\bu_m,c_m,n_m)$ of system  \re{3.40} is defined on  $[0,\tau_N^m]$ for all $N\in \mathbb{N}$ and hence, $\tau_m>\tau^m_N$ almost surely  for all $N\in \mathbb{N}$.  Therefore, 
\begin{equation}
	\tau_m\geq\sup_{N\in \mathbb{N}}\tau^m_N, \  \mathbb{P}\text{-a.s}.\notag
\end{equation}
In order to prove the inequality \re{4.10}, it is sufficient to prove that 
\begin{equation}
\sup_{N\in \mathbb{N}}\tau^m_N>T, \  \mathbb{P}\text{-a.s}.\label{4.14***}
\end{equation}
Before proving this, in the following lemma, we prove some properties of the  local solution $(\bu_m,c_m,n_m)$ of system  \re{3.40}.
\begin{lemma}\label{lem3.4}
Assumption \ref{ass_1} and  Assumption  \ref{ass_2}. Then for  all $m,N\in\mathbb{N}$, the following equality and inequalities hold  $\mathbb{P}$-a.s.
\begin{equation}
	\int_\mathcal{O}n_m(t\wedge\tau_N^m,x)dx=\int_\mathcal{O}n^m_0(x)dx, \text{   for all } t\in [0,T],\label{2.14*}
\end{equation}
	\begin{equation}
		n_m(t\wedge\tau_N^m)>0,\text{ and } c_m(t\wedge\tau_N^m)> 0, \text{   for all } t\in [0,T],\label{3.45}
	\end{equation}
	and 
	\begin{equation}
		\abs{c_m(t\wedge\tau_N^m)}_{L^\infty}\leq\abs{c_0}_{L^\infty}, \text{   for all } t\in [0,T].\label{4.46}
	\end{equation}
\end{lemma}
\begin{proof}
In order to prove the non-negativity of $n_m(t\wedge\tau_N^m)$ and $c_m(t\wedge\tau_N^m)$,  we  will follow the idea of the proof of Lemma \ref{lem3.6}. But, instead of the Gagliardo-Niremberg-Sobolev inequality, we will use the equivalence of the norms on finite dimensional space. 

 Let $N,m\in\mathbb{N}$ and  $t\in [0,T]$ be arbitrary but fixed.  For all $s\in[0,t]$ define
$$n_{m_-(s\wedge\tau_N^m)} := \max(-n_m(s\wedge\tau_N^m), 0).$$ We remark that $n_{m_-}(s\wedge\tau_N^m)\in W^{2,2}(\bo)$ and
\begin{align*}
	&n_{m_-}(s\wedge\tau_N^m)=0\cdot 1_{\{n_m(s\wedge\tau_N^m)\geq 0\}}-n_m(s\wedge\tau_N^m)\cdot1_{\{n_m(s\wedge\tau_N^m)< 0\}},\\
	& \nabla n_{m_-}(s\wedge\tau_N^m)=0\cdot 1_{\{n_m(s\wedge\tau_N^m)\geq 0\}}-\nabla n_m(s\wedge\tau_N^m)\cdot1_{\{n_m(s\wedge\tau_N^m)< 0\}},\\
	&\Delta
	n_{m_-}(s\wedge\tau_N^m)=0\cdot1_{\{n_m(s\wedge\tau_N^m)\geq 0\}}-\Delta n_m(s\wedge\tau_N^m)\cdot1_{\{n_m(s\wedge\tau_N^m)< 0\}}.
\end{align*}
We can easily see also  that for all $s\in[0,t]$,
\begin{align*}
	&	\frac{dn_m(s\wedge\tau_N^m)}{dt}n_{m_-}(s\wedge\tau_N^m)=-\frac{dn_{m_-}(s\wedge\tau_N^m)}{dt}n_{m_-}(s\wedge\tau_N^m),\\
	& n_{m_-}(s\wedge\tau_N^m)\nabla n_m(s\wedge\tau_N^m) =-n_{m_-}(s\wedge\tau_N^m)\nabla n_{m_-}(s\wedge\tau_N^m),\\
	&\Delta n_m(s\wedge\tau_N^m) n_{m_-}(s\wedge\tau_N^m)=-\Delta n_{m_-}(s\wedge\tau_N^m) n_{m_-}(s\wedge\tau_N^m).
\end{align*}
 Hence,  we multiply equation $(\ref{3.1})_3$ by $n_{m_-}(s\wedge\tau_N^m)$ for any $s\in [0,t]$, integrate over $\mathcal{O}$, and use an integration-by-parts with the fact that $\nabla\cdot\bu_m=0$ to obtain
\begin{align}
	&\frac{1}{2}\frac{d}{dt}\abs{n_{m_-}(s\wedge\tau_N^m)}^2_{L^2}\notag\\
	&=-\int_\mathcal{O}\bu_m(s\wedge\tau_N^m,x)\cdot\nabla n_{m_-}(s\wedge\tau_N^m,x)n_{m_-}(s\wedge\tau_N^m,x)dx-\delta\abs{\nabla n_{m_-}(s\wedge\tau_N^m)}^2_{L^2}\notag\\
	&\qquad -\chi\int_\mathcal{O}n_m(s\wedge\tau_N^m,x)\nabla c_m(s\wedge\tau_N^m,x)\nabla n_{m_-}(s\wedge\tau_N^m,x)dx\notag\\
	&=\frac{1}{2}\int_\mathcal{O}n_{m_-}^2(s\wedge\tau_N^m,x)\nabla\cdot\bu_m(s\wedge\tau_N^m,x) dx -\delta\abs{\nabla n_{m_-}(s\wedge\tau_N^m,x)}^2_{L^2}\notag\\
	&\qquad+\chi\int_\mathcal{\bo}n_{m_-}(s\wedge\tau_N^m,x)\nabla c_m(s\wedge\tau_N^m,x)\nabla n_-(s\wedge\tau_N^m,x)dx\notag\\
	&\leq -\delta\abs{\nabla n_{m_-}(s\wedge\tau_N^m)}^2_{L^2}+\chi\abs{n_{m_-}(s\wedge\tau_N^m)}_{L^4}\abs{\nabla c_m(s\wedge\tau_N^m)}_{L^4}\abs{\nabla n_{m_-}(s\wedge\tau_N^m)}_{L^2}\notag\\
	&\leq \bk\abs{n_{m_-}(s\wedge\tau_N^m)}_{H^1}^2\abs{ c_m(s\wedge\tau_N^m)}_{H^2}.\notag
\end{align}
In the last line we have used the continuous embedding of  $H^1(\bo)$ into $L^4(\bo)$. Since the $L^2(\bo)$, $H^1(\bo)$ and $H^2(\bo)$-norms are equivalent on $H_m$,  we then infer from this last inequality that for all $s\in [0,t]$,
\begin{equation}
		\frac{1}{2}\frac{d}{dt}\abs{n_{m_-}(s\wedge\tau_N^m)}^2_{L^2}\leq \bk(m)\abs{n_{m_-}(s\wedge\tau_N^m)}_{L^2}^2\abs{ c_m(s\wedge\tau_N^m)}_{L^2},\label{2.15**}
\end{equation}
where $\bk(m)$ is a constant depending of $m$ which is the dimension of the space $H_m$.
Owing to the fact that $\mathbb{P}$-a.s. the paths of $c_m$ are continuous, we derive that $$ \Sup_{0\leq s\leq t}\abs{c_m(s\wedge\tau_N^m)}_{L^2}<\infty,\ \  \mathbb{P}\text{-a.s.}$$
Hence, integrating \re{2.15**} over $ [0,t]$ we arrive at
\begin{equation}\label{4.19}
	\abs{n_{m_-}(t\wedge\tau_N^m)}^2_{L^2}\leq\abs{(n^m_{0})_-}^2_{L^2}+\bk \int_0^{t }\abs{c_m(s\wedge\tau_N^m)}_{L^2}\abs{n_{m_-}(s\wedge\tau_N^m)}^2_{L^2}ds.
\end{equation}
Thanks to the  Gronwall inequality, we derive from the inequality \re{4.19} that
\begin{equation*}
	\abs{n_{m_-}(t\wedge\tau_N^m)}^2_{L^2}\leq\abs{(n^m_{0})_-}^2_{L^2}\exp\left(\bk\int_0^t\abs{ c_m(s\wedge\tau_N^m)}_{L^2}ds\right),
\end{equation*}
which  implies that $\mathbb{P}$-a.s,  $n_{m_-}(t\wedge\tau_N^m) = 0$ for  all $t\in [0,T]$ since by the relation \re{4.42}, $n_{0}^m >0$.

The non-negativity property of $c_m(t\wedge\tau_N^m)$ is quite similar to the proof of Lemma \ref{lem3.6}. We consider the function $\Psi:H_m\to \mathbb{R}$ defined by $\Psi(c)=\int_\mathcal{O}c_-^2(x)dx$ where $c_-=\max(-c; 0)$. Let  $\{\psi_h\}_{h\in\mathbb{N}}$ be a sequence of smooth functions defined by $\psi_h (y)=y^2\varphi(hy)$, for all $y\in \mathbb{R}$ and $h\in\mathbb{N}$, where the function $\varphi$ is defined by \re{4.14**}. 
We consider for any $h\geq 1$, the following sequence of function $\Psi_h:H_m\to \mathbb{R}$ defined by $\Psi_h=\int_\mathcal{O}\psi_h(c(x))dx$, for $c\in H_m$. The mapping $\Psi_h$ is twice (Fr\'echet) differentiable and its first and second derivatives are given by
\begin{equation*}
	\Psi'_h(c)(z)=2\int_\mathcal{O}c(x)\varphi(hc(x))z(x)dx+h\int_\mathcal{O}c^2(x)\varphi'(hc(x))z(x)dx, \quad \forall c, z\in H_m,
\end{equation*}
and
\begin{align*}
	\Psi_h^{''}(c)(z,k)&=h^2\int_\mathcal{O}c^2(x)\varphi^{''}(hc(x))z(x)k(x)dx\\
	&\qquad{}+4h\int_\bo c(x)\varphi'(hc(x))z(x)k(x)dx+2\int_\mathcal{O}\varphi(hc(x))z(x)k(x)dx,\quad \forall c,z,  k\in H_m.
\end{align*}
Applying the It\^o formula to  $t\mapsto\Psi_h(c_m(t\wedge\tau_N^m))$, we obtain for all $t\in [0,T]$,
\begin{align}
	\Psi_h(c_m(t\wedge\tau_N^m))-\Psi_h(c_m(0))&=\int_0^{t\wedge\tau_N^m}\Psi'_h(c_m(s))\left(\bu_m(s)\cdot \nabla c_m(s)+\xi\Delta c_m(s) -n_m(s)f(c_m(s)) \right)ds\notag\\
	&\qquad{}+\frac{1}{2}\int_0^{t\wedge\tau_N^m}\sum_{k=1}^2\Psi^{''}_h(c_m(s))\left(\gamma\phi_k(c_m(s)),\gamma\phi_k(c_m(s))\right)ds\notag\\
	&\qquad{}+\gamma \sum_{k=1}^2\int_0^{t\wedge\tau_N^m}\Psi'_h(c_m(s))(\phi_k(c_m(s)))d \beta^k_s.\notag
\end{align}
	Similarly to \re{4.16***}, \re{4.18}, \re{4.11*}, \re{2.24*} and \re{4.13*}, we can infer from this last equality  that
\begin{align}
	\int_\bo\psi_h(c_m(t\wedge\tau_N^m,x))dx&-\int_\bo\psi_h(c_0^m(x))dx\notag\\
	&=\int_0^{t\wedge\tau_N^m}\Psi'_h(c_m(s))\left(\bu_m(s)\cdot \nabla c_m(s)+\eta\Delta c_m(s) -n_m(s)f(c_m(s)) \right)ds.\label{4.14*}
\end{align}
Now, observe that from the assumptions on the function $\varphi$, we infer that  for all $y\in \mathbb{R}$ we have
\begin{equation}
	\lim_{h\longrightarrow\infty}\psi_h(y)=-y^2\cdot1_{\{y<0\}}=-y_-^2 \ \text{ and } \ \lim_{h\longrightarrow\infty}2y\varphi(hy)=-2y\cdot1_{\{y<0\}}.\label{4.17*}
\end{equation}
We note that for any $y\in \mathbb{R}$, we have
\begin{equation}
	\lim_{h\longrightarrow\infty}h\varphi'(hy)=0,\label{4.16****}
\end{equation}
and also that
\begin{equation}
	\abs{\psi_h(y)}\leq \bk y^2\ \text{ and }\ \abs{h\varphi'(hy)}\leq \bk\abs{y},\label{4.16**}
\end{equation}
for any $y\in\mathbb{R}$ and for all $h\geq 1$, where $\bk>0$ is a constant.

Using \re{4.17*}-\re{4.16**}  and applying the Lebesgue Dominated Convergence Theorem, we can  pass to the limit as $h$ tends to infinity in \re{4.14*}.
In this way, we derive that
\begin{align}
	&-\int_\bo c^2_{m_-}(t\wedge\tau_N^m,x)dx+\int_\bo (c_0^m(x))_-^2dx\notag\\
	&=-2\int_0^{t\wedge\tau_N^m}\int_\mathcal{O}\left((\bu_m(s,x)\cdot \nabla c_m(s,x)+\eta\Delta c_m(s,x)) )\right)c_m(s,x)1_{\{c_m(s,x)<0\}}dxds\notag\\
&\qquad+2\int_0^{t\wedge\tau_N^m}\int_\mathcal{O}n_m(s,x)f(c_m(s,x))c_m(s,x)1_{\{c_m(s,x)<0\}}dxds\label{2.21***}\\	&=2\int_0^{t\wedge\tau_N^m}\int_\mathcal{O}\left(\eta\abs{\nabla c_m(s,x)}^2+n_m(s,x)f(c_m(s,x))c_m(s,x) \right)1_{\{c_m(s,x)<0\}}dxds,\notag
\end{align}
where we have used integration-by-parts and the fact that $\nabla\cdot\bu_m=0$.
By the mean value theorem we know that, for all $x\in \bo$, there exists a number  $\lambda_m(x)\in(\min(0,c_m(x)), \max(0,c_m(x)))$ such that \\
$$f(c_m(x))-f(0)=c_m(x)f'(\lambda_m(x)).
$$
By the fact that $f(0)=0$, we  infer from \re{2.21***} that
\begin{equation*}
	\abs{c_{m_-}(t\wedge\tau_N^m)}^2_{L^2}-\abs{(c_0^m)_-}^2_{L^2}=-2\int_0^{t\wedge\tau_N^m}\int_\mathcal{O}n_m(s,x)f'(\lambda_m(s,x))c_m^2(s,x) 1_{\{c_m(s,x)<0\}}dxds.
\end{equation*}
Since $f'>0$ and $1_{\{c_m<0\}}>0$ as well as on $[0,t\wedge\tau_N^m]$, $c_m^2>0$ and $n_m> 0$, we deduce that $\abs{c_{m_-}(t\wedge\tau_N^m)}^2_{L^2}\leq\abs{(c_0^m)_-}^2_{L^2}$. Owing to the fact that by the relation \re{4.42} we have  $c^m_0>0$, we derive that $(c^m_0)_-=0$ and therefore $\abs{c_{m_-}(t\wedge\tau_N^m)}^2_{L^2}=0$. This  gives $c_{m_-}(t\wedge\tau_N^m)=0$ and implies that  for all $t\in[0,T]$, $\mathbb{P}$-a.s, $c_m(t\wedge\tau_N^m)>0$. 

It remains to prove the inequality \re{4.46}.  The  proof is similar to the proof of Corollary \ref{lem3.7}. Let $p\geq2$ be an integer. Let $\Psi: H_m\to \mathbb{R}$ be the functional defined by $\Psi(c)=\int_\mathcal{O}c^p(x)dx$.  Note that the mapping $\Psi$ is twice (Fr\'echet) differentiable and its first and second derivatives are given by
\begin{align*}
	&\Psi'(c)(z)=p\int_\mathcal{O}c^{p-1}(x)z(x)dx, \quad \forall c, z\in H_m,\\
	&\Psi^{''}(c)(z,k)=p(p-1)\int_\mathcal{O}c^{p-2}(x)z(x)k(x)dx, \quad \forall c,z,  k\in H_m.
\end{align*}
By applying the  It\^o formula to the process $t\mapsto\Psi(c_m(t\wedge\tau_N^m))$, we derive that for all $t\in [0,T]$,
\begin{align}
	\Psi(c_m(t\wedge\tau_N^m))-\Psi(c_m(0))&=\int_0^{t\wedge\tau_N^m}\Psi'(c_m(s))\left(\bu(s)\cdot \nabla c_m(s)+\xi\Delta c_m(s) -n_m(s)G(c_m(s)) \right)ds\notag\\
	&\quad{}+\frac{1}{2}\int_0^{t\wedge\tau_N^m}\sum_{k=1}^2\Psi^{''}(c_m(s))\left(\gamma\phi_k(c_m(s)),\gamma\phi_k(c_m(s))\right)ds\notag\\
	&\qquad{}+\gamma \sum_{k=1}^2\int_0^{t\wedge\tau_N^m}\Psi'(c_m(s))(\phi_k(c_m(s)))d \beta^k_s,\notag
\end{align}
	from which and  calculations similar to \re{2.24*}, \re{2.30*}, \re{2.31*} and \re{2.32*} we derive from the last equality  that
\begin{align*}
&	\Psi(c_m(t\wedge\tau_N^m))-\Psi(c_0^m)\\
&=\int_0^{t\wedge\tau_N^m}\int_\mathcal{O}\left(-p(p-2)\abs{\nabla c_m(s,x)}^2c_m^{p-2}(s,x)-pn_m(s,x)f(c_m(s,x))c_m^{p-1}(s,x) \right)dxds.
\end{align*}
Since for all $s\in [0,t]$ the quantities $n_m(s\wedge\tau_N^m)$, $f(c_m(s\wedge\tau_N^m))$ and $c_m(s\wedge\tau_N^m)$ are positive $\mathbb{P}$-a.s, we infer from the last equality  that for all $t \in [0,T]$,		$\Psi(c_m(t\wedge\tau_N^m))\leq\Psi(c_0^m).$
This implies that $\abs{c_m(t\wedge\tau_N^m)}_{L^p}\leq\abs{c^m_0}_{L^p}$ for all $p\geq 2$. Using the fact that $\abs{.}_{L^p}\to \abs{.}_{L^\infty}$ as $p\to +\infty$  and the inequality \re{4.42}, we obtain the result.
\end{proof}
Next, we introduce for any $t\in [0,T]$ and $m,N\in \mathbb{N}$, the following Lyapunov functional
\begin{align*}
	\bec (n_m,c_m,\bu_m)(t\wedge\tau_N^m)&=\int_\bo n_m(t\wedge\tau_N^m)\ln n_m(t\wedge\tau_N^m)dx +\bk_f\abs{\nabla c_m(t\wedge\tau_N^m)}^2_{L^2}\\	&\qquad+\frac{\bk_4}{\eta}\abs{\bu_m(t\wedge\tau_N^m)}^2_{L^2}+e^{-1}\abs{\bo},
\end{align*}
where $\bk_4$ is some positive constant to be given later and $\bk_f$ is defined in \eqref{Eq:K-f}.
Since $x\ln x\geq -e^{-1}$ for any $x> 0$, we can easily see that for all $t\in [0,T]$,  $\bec (n_m,c_m,\bu_m)(t\wedge\tau_N^m)\geq 0$.
As in \cite{Zhai} the property \re{4.42} implies that
\begin{equation}
	\bec (n_0^m,c_0^m,\bu_0^m)\leq \bec (n_0,c_0,\bu_0),\qquad\text{for all } m\geq 1.\label{4.29}
\end{equation}
 In addition, taking into account the inequality \re{4.46} and setting $\bk=\min(\bk_f,\frac{\bk_4}{\eta})$ the following holds  for all $t\in [0,T]$,
\begin{align}
	\abs{(\bu_m(t),c_m(t\wedge\tau_N^m)}_{\bhc}^2&\leq \bk^{-1}\bec (n_m,c_m,\bu_m)(t\wedge\tau_N^m)+\bk^{-1}\abs{c_m(t\wedge\tau_N^m)}_{L^2}^2\notag\\
	&\leq \bk^{-1}\bec (n_m,c_m,\bu_m)(t\wedge\tau_N^m)+\bk^{-1}\abs{\bo}\abs{c_0}_{L^\infty}^2,\quad\mathbb{P}\text{-a.s.}\label{3.20*}
\end{align}
We now proceed to establish some   uniform bounds  for $\bu_m$, $c_m$, and $n_m$ in some suitable spaces. For this purpose, we  recall that hereafter, $\bk$ will denote a positive constant independent of $m$ and $N$, which may change from one term to the next. 
\begin{lemma}\label{proposition 3.2}
	Under the same assumptions as in Proposition \ref{theo4.1}, there exists a positive constant $\bk$ such that for all $m\in\mathbb{N}$ and $N\in\mathbb{N}$,
	\begin{equation}\label{3.47}
		\sup_{0\leq s\leq T}\abs{c_m(s\wedge\tau_N^m)}^{2}_{L^2}+2\eta\int_0^{T\wedge\tau_N^m}\abs{\nabla c_m(s)}_{L^2}^2ds\leq \abs{\bo}\abs{c_0}^2_{L^\infty},\qquad \mathbb{P}\text{-a.s.}
	\end{equation}
	\begin{equation}\label{3.48}
		\begin{split}
			&\be\sup_{0\leq s\leq T}	\bec (n_m,c_m,\bu_m)(s\wedge\tau_N^m)\leq\bk,\\
			&\be\int_0^{T\wedge\tau_N^m}\left(\abs{\nabla\sqrt{n_m(s)}}^2_{L^2}
			+\abs{\Delta c_m(s)}^2_{L^2}+\abs{\nabla \bu_m(s)}^2_{L^2}\right)ds
			\leq\bk.
		\end{split}
	\end{equation}
\end{lemma}
\begin{proof}
Let $t\in [0,T]$ be arbitrary but fixed.	We start by proving the estimate \re{3.47}. To do this, we take $(m,N)\in\mathbb{N}^2$ arbitrary and apply the It\^o formula to $t\mapsto\abs{c_m(t\wedge\tau_N^m)}^2_{L^2}$   to get
	\begin{align}
		&\abs{c_m(t\wedge\tau_N^m)}^2_{L^2}+2\xi\int_0^{t\wedge\tau_N^m}\abs{\nabla c_m(s)}^2_{L^2}ds\notag\\
		&=\abs{c^m_0}^2_{L^2}-2\int_0^{t\wedge\tau_N^m}(B_1(\bu_m(s),c_m(s)),c_m(s))ds-2\int_0^{t \wedge\tau_N^m}(R_1(n_m(s),c_m(s)),c_m(s))ds\label{3.29**}\\
		&\qquad{}+\gamma^2\int_0^{t\wedge\tau_N^m}\abs{\phi(c_m(s))}^2_{\mathcal{L}^2(\mathbb{R}^2;L^2)}\notag+2\gamma\int_0^{t\wedge\tau_N^m}(\phi(c_m(s)),c_m(s))d\beta_s.\notag
	\end{align}
	By integration by part, we derive that
	\begin{equation*}
		(B_1(\bu_m,c_m),c_m)=\frac{1}{2}\int_\bo\bu_m(x)\cdot\nabla c_m^2(x)dx=-\frac{1}{2}\int_\bo c_m^2(x)\nabla\cdot\bu_m(x) dx=0.
	\end{equation*}
	By the free divergence property of $\sigma_k$ and the fact that $\sigma_k=0$ on $\partial\bo$, $k=1,2$, we get
	\begin{align*}
		(\phi(c_m),c_m)&=\sum_{k=1}^2\int_\bo\sigma_k(x)\cdot\nabla c_m(x)c_m(x)dx\notag\\
		&=\frac{1}{2}\sum_{k=1}^2\int_\bo\sigma_k(x)\cdot\nabla c_m^2(x)dx\notag\\
		&=-\frac{1}{2}\sum_{k=1}^2\int_\bo c_m^2(x)\nabla\cdot\sigma_k(x) dx+\frac{1}{2}\sum_{k=1}^2\int_{\partial\bo }c_m^2(\sigma)\sigma_k(\sigma)\cdot\nu d\sigma\notag\\
		&=0.
	\end{align*}
	Taking into account the equality \re{2.24*}, we infer that
	\begin{equation*}
		\abs{\phi(c_m)}^2_{\mathcal{L}^2(\mathbb{R}^2;L^2)}=\sum_{k=1}^2\int_\bo\abs{\sigma_k(x)\cdot\nabla c_m(x)}^2_{L^2}dx=\abs{\nabla c_m}^2_{L^2}.
	\end{equation*}
	Using these three last equalities and the fact that $\abs{c^m_0}^2_{L^2}\leq\abs{\bo}\abs{c_0}^2_{L^\infty}$ (since  by the relation \re{4.42}, $\abs{c^m_0}^2_{L^\infty}\leq\abs{c_0}^2_{L^\infty}$), we infer from the equality \re{3.29**} that  for all $t\in [0,T]$,
		\begin{equation}
		\abs{c_m(t\wedge\tau_N^m)}^2_{L^2}+2\eta\int_0^{t\wedge\tau_N^m}\abs{\nabla c_m(s)}^2_{L^2}ds+2\int_0^{t\wedge\tau_N^m} \int_{\bo}n_m(s,x)f(c_m(s,x))c_m(s,x)dxds\leq\abs{\bo}\abs{c_0}^2_{L^\infty}.\label{4.35}
	\end{equation}
Thanks to the non-negativity of $n_m(s\wedge\tau_N^m)$, $c_m(s\wedge\tau_N^m)$ and $f$  over the interval $[0,t]$ given in  Lemma  \ref{lem3.4} and Assumption \ref{ass_1}, we can deduce from the inequality \re{4.35} that 
	\begin{equation}\label{4.36}
	\sup_{0\leq 
		t\leq T}\abs{c_m(t\wedge\tau_N^m)}^{2}_{L^2}+2\eta\int_0^{T\wedge\tau_N^m}\abs{\nabla c_m(s)}_{L^2}^2ds\leq \abs{\bo}\abs{c_0}^2_{L^\infty},\qquad \mathbb{P}\text{-a.s.}
\end{equation}
Let us now move to the proof  of the estimate \re{3.48}.	

Multiplying equation \re{3.1}$_3$ by $1+\ln n_m(s\wedge\tau_N^m)$ for $s\in [0,t]$ and integrate the resulting equation in $\bo$ and using an integration-by-parts as well as the divergence free property of $\bu_m$, we have
	\begin{align}\label{3.27***}
		&\frac{d}{dt}\int_\bo n_m(s\wedge\tau_N^m,x)\ln n(s\wedge\tau_N^m,x)dx+\delta\int_\bo\dfrac{\abs{\nabla n_m(s\wedge\tau_N^m,x)}^2}{n_m(s\wedge\tau_N^m,x)}dx\notag\\
		&=\chi\int_\bo \nabla n_m(s\wedge\tau_N^m,x)\cdot\nabla c_m(s\wedge\tau_N^m,x)dx.
	\end{align}
In the equality \re{3.27***}, we have used the fact that $\bu_m=\frac{\partial n_m}{\partial \nu}=0$ on $\partial\bo$ and the fact that
\begin{align*}
-\int_\bo\Delta n_m(x)\ln (n_m(x))dx&=\int_\bo\nabla n_m(x)\cdot\nabla\ln (n_m(x))dx-\int_{\partial\bo}\frac{ \partial n_m(\sigma)}{\partial\nu}\ln (n_m(\sigma))d\sigma\\
&=\int_\bo\frac{\nabla n_m(x)\cdot\nabla n_m(x)}{n_m(x)}dx, 
\end{align*}
as well as
\begin{align*}
\int_\bo \bu_m(x)\cdot\nabla n_m(x)\ln (n_m(x))dx
&=-\int_\bo n_m(x)\nabla\cdot(\bu_m(x)\ln (n_m(x)))dx\\
&\qquad+\int_{\partial\bo} n_m(\sigma)\ln (n_m(\sigma))\bu_m(\sigma)\cdot \nu d\sigma\\
&=-\int_\bo n_m(x)\bu_m(x)\cdot\nabla\ln(n_m(x))dx\\
&\qquad-\int_\bo n_m(x)\ln (n_m(x))\nabla\cdot\bu_m(x)dx\\
&=-\int_\bo\bu_m(x)\cdot\nabla n_m(x)dx.
\end{align*}
It follows from the Young  inequality and the Cauchy-Schwarz inequality that
	\begin{equation*}
		\chi\int_\bo \nabla n_m(x)\cdot \nabla c_m(x)dx\leq\frac{\delta}{2}\int_\bo\dfrac{\abs{\nabla n_m(x)}^2}{n_m(x)}dx+\frac{\chi^2}{2\delta}\int_\bo n_m(x)\abs{\nabla c_m(x)}^2dx.
	\end{equation*}
	Since $$\int_\bo\dfrac{\abs{\nabla n_m(x)}^2}{n_m(x)}dx=4\int_\bo\abs{\nabla\sqrt{n_m(x)}}^2dx,$$ we may combine the last inequality with equality \re{3.27***} to obtain
	\begin{align}\label{3.27****}
		&\int_\bo n_m(t\wedge\tau_N^m,x)\ln n_m(t\wedge\tau_N^m,x)dx+2\delta\int_0^{t\wedge\tau_N^m}\abs{\nabla\sqrt{n_m(s)}}^2_{L^2}ds\notag\\
		&\leq	\int_\bo n^m_0(x)\ln n^m_0(x)dx+\frac{\chi^2}{2\delta} \int_0^{t\wedge\tau_N^m}\abs{\sqrt{n_m(s)}\nabla c_m(s)}_{L^2}^2ds.
	\end{align}
	Applying the It\^o formula once more to $t\mapsto\abs{\nabla c_m(t\wedge\tau_N^m)}^2_{L^2}$, yields
	\begin{align}
		&\abs{\nabla c_m(t\wedge\tau_N^m)}^2_{L^2}+2\xi\int_0^{t\wedge\tau_N^m}\abs{\Delta c_m(s)}^2_{L^2}ds\notag\\
		&=\abs{\nabla c^m_0}^2_{L^2}-2\int_0^{t\wedge\tau_N^m}(\nabla B_1(\bu_m(s),c_m(s)),\nabla c_m(s))ds\notag\\
		&\qquad{}-2\int_0^{t\wedge\tau_N^m} (\nabla R_1(n_m(s),c_m(s)),\nabla c_m(s))ds\label{3.31**}\\
		&\qquad{}+\gamma^2\int_0^{t\wedge\tau_N^m}\abs{\nabla\phi(c_m(s))}^2_{\mathcal{L}^2(\mathbb{R}^2;L^2)}+2\gamma\int_0^{t\wedge\tau_N^m}(\nabla\phi(c_m(s)),\nabla c_m(s))d\beta_s.\notag
	\end{align}
	Since $\bu_m$ is solenoidal and vanishes on $\partial\bo$, we derive that
	\begin{align}
		(\nabla B_1(\bu_m,c_m),\nabla c_m)&=\int_\bo\nabla(\bu_m(x)\cdot\nabla c_m(x))\cdot\nabla c_m(x)dx\notag\\
		&=\int_\bo\nabla \bu_m(x)\nabla c_m(x)\cdot\nabla c_m(x)dx\notag\\
		&\qquad{}+\int_\bo\nabla c_m(x)\cdot D^2c_m(x)\bu_m(x) dx\notag\\
		&\leq\int_\bo\abs{\nabla \bu_m(x)}\abs{\nabla c_m(x)}^2dx+\frac{1}{2}\int_\bo u_m(x)\cdot\nabla\abs{\nabla c_m(x)}^2dx\label{3.26*}\\
		&\leq \abs{\nabla \bu_m}_{L^2}\abs{\nabla c_m}^2 _{L^4}.\notag
	\end{align}
	We use the Gagliardo-Niremberg inequality to obtain
	\begin{equation*}
		\abs{\nabla c_m}^4_{L^4}\leq\bk_{GN} \abs{c_m}^2_{L^\infty}\abs{D^2c_m}^2_{L^2}+\bk_{GN}\abs{c_m}_{L^\infty}^{4},
	\end{equation*}
	To cancel $\abs{D^2c_m}_{L^2}$, we invoke the pointwise identity $$\abs{\Delta c_m}^2=\nabla\cdot(\Delta c_m\nabla c_m)-\nabla c_m\cdot \nabla\Delta c_m,$$ and $\Delta\abs{\nabla c_m}^2=2\nabla c_m\cdot\nabla\Delta c_m+2\abs{D^2c_m}^2$, as well as the integration-by-parts to rewrite $\abs{\Delta c_m}^2_{L^2}$ as
	\begin{align}
		\abs{\Delta c_m}^2_{L^2}&=-\int_\bo\nabla c_m(x)\cdot \nabla\Delta c_m(x)dx\notag\\
		&= \abs{D^2c_m}^2_{L^2}-\frac{1}{2}\int_\bo\Delta\abs{\nabla c_m(x)}^2dx\label{3.27*****}\\
		&=\abs{D^2c_m}^2_{L^2}-\frac{1}{2}\int_{\partial\bo}\frac{\partial \abs{\nabla c_m(\sigma)}^2}{\partial\nu}d\sigma.\notag
	\end{align}
	Invoking \cite[Lemma 4.2]{miz} we obtain
	\begin{equation}\label{3.28*}
		\frac{1}{2}\int_{\partial\bo}\frac{\partial \abs{\nabla c_m(\sigma)}^2}{\partial\nu}d\sigma\leq\kappa(\bo)\int_{\partial\bo}\abs{\nabla c_m(\sigma)}^2d\sigma,
	\end{equation}
	where $\kappa(\bo)$ is an upper bound for the curvatures of $\partial\bo$.
	
	Thanks to the trace theorem (see \cite[(ii) of Proposition 4.22 with (i) of Theorem 4.24]{Har}), it holds that
	\begin{equation*}
		\int_{\partial\bo}\abs{\nabla c_m(\sigma)}^2d\sigma\leq \bk(\bo,\varsigma)\abs{c_m}^2_{H^{\frac{3+\varsigma}{2}}} \qquad \text{for any } \varsigma\in (0,1),
	\end{equation*}
	where $\bk(\bo,\varsigma)>0$ depends only on $\bo$ and $\varsigma$, which can be fixed for instance $\varsigma=1/2$. On the other hand, the interpolation inequality, the Young inequality and the inequality  \re{4.46} of Lemma \ref{lem3.4} imply the existence of $\bk_1$ and $\bk_2$ depending on $\bo$  such that
	\begin{align}
		\kappa(\bo)\bk(\bo,\varsigma)\abs{c_m}^2_{H^{\frac{7}{4}}}&\leq \bk_1(\abs{D^2c_m}^{7/4}_{L^2}\abs{c_m}^{1/4}_{L^2}+\abs{c_m}^2_{L^2})\notag\\
		&\leq \frac{1}{4}\abs{D^2c_m}^{2}_{L^2}+\bk_2\abs{c_0}^2_{L^\infty}.\notag
	\end{align}
	Using this previous inequality and   \re{3.28*}, we infer from the equality \re{3.27*****} that
	\begin{equation}\label{3.30**}
		\abs{D^2c_m}_{L^2}^2\leq \frac{4}{3}	\abs{\Delta c_m}_{L^2}^2+\frac{4\bk_2}{3}\abs{c_0}^2_{L^\infty},
	\end{equation}
	and therefore
	\begin{equation*}
		\abs{\nabla c_m}^4_{L^4}\leq\frac{4\bk_{GN}\abs{c_0}^2_{L^\infty}}{3} \abs{\Delta c_m}^2_{L^2}+\left(\frac{4\bk_2}{3}+1\right)\bk_{GN}\abs{c_0}_{L^\infty}^{4}.
	\end{equation*}
	By the inequality  \re{3.26*} and the  Young inequality, we infer that
	\begin{align}
		(\nabla B_1(\bu_m,c_m),\nabla c_m)&\leq \abs{\nabla \bu_m}_{L^2}\abs{\nabla c_m}^2_{L^4}\notag\\
		&\leq \frac{3\xi}{16\bk_{GN}\abs{c_0}^2_{L^\infty}}\abs{\nabla c_m}^4                                                                                                     _{L^4}+\frac{4\bk_{GN}\abs{c_0}^2_{L^\infty}}{3\xi} \abs{\nabla \bu_m}^2_{L^2}\notag\\
		&\leq \frac{\xi}{4}\abs{\Delta c_m}_{L^2}^2+\frac{4\bk_{GN}\abs{c_0}^2_{L^\infty}}{3\xi} \abs{\nabla \bu_m}^2_{L^2}+\frac{\xi(4\bk_2+3)}{16}\abs{c_0}_{L^\infty}^{2}.\notag
	\end{align}
	Due to the Assumption $1$  and the inequality  \re{4.46} of Lemma \ref{lem3.4}, we note that
	\begin{align}
		- (\nabla R_1(n_m,c_m),\nabla c_m)
		&=-\int_\bo \nabla (n_m(x)f(c_m(x)))\cdot\nabla c_m(x)dx\notag\\
		&=-\int_\bo f'(c_m(x))\abs{\nabla c_m(x)}^2n_m(x)dx-\int_\bo f(c_m(x))\nabla c_m(x)\cdot\nabla n_m(x)dx\notag\\
		&\leq -\frac{\Min_{0\leq c\leq \abs{c_0}_{L^\infty}}f'(c)}{2}\int_\bo n_m(x)\abs{\nabla c_m(x)}^2dx\notag\\
		&\qquad{}+\frac{1}{2\Min_{0\leq c\leq \abs{c_0}_{L^\infty}}f'(c)}\int_\bo f^2(c_m(x))\frac{\abs{\nabla n_m(x)}^2}{n_m(x)}dx\notag\\
		&\leq -\frac{\Min_{0\leq c\leq \abs{c_0}_{L^\infty}}f'(c)}{2}\abs{\sqrt{n_m}\nabla c_m}_{L^2}^2+\frac{2\Max_{0\leq c\leq \abs{c_0}_{L^\infty}}f^2(c)}{\Min_{0\leq c\leq \abs{c_0}_{L^\infty}}f'(c)} \abs{\nabla \sqrt{n_m}}_{L^2}^2.\notag
	\end{align}
	Combining these two last inequalities, we derive from equality \re{3.31**} that
	\begin{align}
		\abs{\nabla c_m(t\wedge\tau_N^m)}^2_{L^2}	&+\frac{3\xi}{2}\int_0^{t\wedge\tau_N^m}\abs{\Delta c_m(s)}^2_{L^2}ds+\Min_{0\leq c\leq \abs{c_0}_{L^\infty}}f'(c)\int_0^{t\wedge\tau_N^m}\abs{\sqrt{n_m(s)}\nabla c_m(s)}_{L^2}^2ds\notag\\
		&\leq\abs{\nabla c^m_0}^2_{L^2}+\frac{\xi(4\bk_2+3)}{8}\abs{c_0}_{L^\infty}^{2}t+\frac{8\bk_{GN}\abs{c_0}^2_{L^\infty}}{3\xi} \int_0^{t\wedge\tau_N^m}\abs{\nabla \bu_m(s)}^2_{L^2}ds\notag\\
		&\qquad{}+\frac{4\Max_{0\leq c\leq \abs{c_0}_{L^\infty}}f^2(c)}{\Min_{0\leq c\leq \abs{c_0}_{L^\infty}}f'(c)} \int_0^{t\wedge\tau_N^m}\abs{\nabla \sqrt{n_m(s)}}_{L^2}^2ds\notag\\
		&\qquad{}+\gamma^2\int_0^{t\wedge\tau_N^m}\abs{\nabla\phi(c_m(s))}^2_{\mathcal{L}^2(\mathbb{R}^2;L^2)}ds+2\gamma\int_0^{t\wedge\tau_N^m}(\nabla\phi(c_m(s)),\nabla c_m(s))d\beta_s.\notag
	\end{align}
	Multiplying this last inequality by $\bk_f$ and adding the  result with  inequality \re{3.27****} we obtain
	\begin{align}
	&	\int_\bo n_m(t\wedge\tau_N^m,x)\ln n_m(t\wedge\tau_N^m,x)dx +\bk_f\abs{\nabla c_m(t\wedge\tau_N^m)}^2_{L^2}	+\frac{\xi\bk_f}{4}\int_0^{t\wedge\tau_N^m}\abs{\Delta c_m(s)}^2_{L^2}ds\notag\\
		&\qquad{}+2\delta\int_0^{t\wedge\tau_N^m}\abs{\nabla\sqrt{n_m(s)}}^2_{L^2}ds+\int_0^t\abs{\sqrt{n_m(s)}\nabla c_m(s)}_{L^2}^2ds\notag\\
		&\leq\bk_f\abs{\nabla c^m_0}^2_{L^2}+\int_\bo n^m_0(x)\ln n^m_0(x)dx+\frac{\bk_f\xi(4\bk_f\bk_2+3)}{8}\abs{c_0}_{L^\infty}^{2}t\notag\\
		&\qquad{}+\frac{8\bk_f\bk_{GN}\abs{c_0}^2_{L^\infty}}{3\xi} \int_0^{t\wedge\tau_N^m}\abs{\nabla \bu_m(s)}^2_{L^2}ds+\frac{4\bk_f\Max_{0\leq c\leq \abs{c_0}_{L^\infty}}f^2(c)}{\Min_{0\leq c\leq \abs{c_0}_{L^\infty}}f'(c)} \int_0^{t\wedge\tau_N^m}\abs{\nabla \sqrt{n_m(s)}}_{L^2}^2ds\notag\\
		&\qquad{}+\gamma^2\bk_f\int_0^{t\wedge\tau_N^m}\abs{\nabla\phi(c_m(s))}^2_{\mathcal{L}^2(\mathbb{R}^2;L^2)}ds
		+2\gamma\bk_f\int_0^{t\wedge\tau_N^m}(\nabla\phi(c_m(s)),\nabla c_m(s))d\beta_s.\notag
	\end{align}
By using the first  inequality of   \re{4.46*}, we see that the previous inequality reduces to
	\begin{align}
		&\int_\bo n_m(t\wedge\tau_N^m,x)\ln n_m(t\wedge\tau_N^m,x)dx +\bk_f\abs{\nabla c_m(t\wedge\tau_N^m)}^2_{L^2}	+\frac{3\xi\bk_f}{2}\int_0^{t\wedge\tau_N^m}\abs{\Delta c_m(s)}^2_{L^2}ds\notag\\
		&\qquad{}+2\delta\int_0^{t\wedge\tau_N^m}\abs{\nabla\sqrt{n_m(s)}}^2_{L^2}ds+\int_0^{t\wedge\tau_N^m}\abs{\sqrt{n_m(s)}\nabla c_m(s)}_{L^2}^2ds\notag\\
		&\leq\bk_f\abs{\nabla c^m_0}^2_{L^2}+\int_\bo n^m_0(x)\ln n^m_0(x)dx+\frac{\bk_f\xi(4\bk_f\bk_2+3)}{8}\abs{c_0}_{L^\infty}^{2}t\notag\\
		&\qquad{}+\frac{8\bk_f\bk_{GN}\abs{c_0}^2_{L^\infty}}{3\xi} \int_0^{t\wedge\tau_N^m}\abs{\nabla \bu_m(s)}^2_{L^2}ds+\gamma^2\bk_f\int_0^{t\wedge\tau_N^m}\abs{\nabla\phi(c_m(s))}^2_{\mathcal{L}^2(\mathbb{R}^2;L^2)}ds\label{3.36*}\\
		&\qquad{}+2\gamma\bk_f\int_0^{t\wedge\tau_N^m}(\nabla\phi(c_m(s)),\nabla c_m(s))d\beta_s.\notag
	\end{align}
Now,  we use the equality \re{2.14*} of Lemma \ref{lem3.4} and the  inequality \re{4.4.} to  obtain that 
	\begin{align}
		\abs{n_m}_{L^2}&\leq \bk_{GN}\left( \abs{\sqrt{n_m}}_{L^2}\abs{\nabla\sqrt{n_m}}_{L^2}+ \abs{\sqrt{n_m}}^2_{L^2}\right)\notag\\
		&\leq \bk_{GN}\left( \abs{n^m_0}^{\frac{1}{2}}_{L^1}\abs{\nabla\sqrt{n_m}}_{L^2}+ \abs{n^m_0}_{L^1}\right),\label{3.37*}
	\end{align}
By the relation \re{4.42}, we have $n_0^m\to n_0 \text{  in  } L^2(\bo)$.  Thanks to the continuous embedding of $L^2(\bo)$ into $L^1(\bo)$, we derive that $n_0^m\to n_0 \text{  in  } L^1(\bo)$ and therefore the sequence $\{n_0^m\}_{m\geq 1}$ is bounded in $L^1(\bo)$.  This implies that the inequality \re{3.37*} can be controlled as follows
	\begin{align}
	\abs{n_m}_{L^2}\leq \bk_{GN}\bk^{1/2}\abs{\nabla\sqrt{n_m}}_{L^2}+ \bk,\label{4.53}
\end{align}
where $\bk$ is a constant independent of $m$ and $N$.

Next,  applying the It\^o formula to $t\mapsto\abs{\bu_m(t\wedge\tau_N^m)}^2_{L^2}$ and using the estimation \re{4.53}, we infer the existence of $\bk_3>0$ such that
	\begin{align}
		&\abs{\bu_m(t\wedge\tau_N^m)}^2_{L^2}+2\eta\int_0^{t\wedge\tau_N^m}\abs{\nabla \bu_m(s)}^2_{L^2}ds\notag\\
		&\leq 2\int_0^{t\wedge\tau_N^m}\abs{\nabla\Phi}_{L^\infty}\abs{n_m(s)}_{L^2}\abs{\bu_m(s)}_{L^2}ds\notag\\
		&\qquad{}+\int_0^{t\wedge\tau_N^m}\abs{g(\bu_m(s),c_m(s))}^2_{\mathcal{L}^2(\buc;H)}ds+2\int_0^{t\wedge\tau_N^m}(g(\bu_m(s),c_m(s)),\bu_m(s))dWs\notag\\
		&\leq\abs{\bu^m_0}^2_{L^2}+ \frac{\delta\eta}{\bk_4}\int_0^{t\wedge\tau_N^m}\abs{\nabla\sqrt{n_m(s)}}_{L^2}^2ds+\bk_3\abs{\nabla\Phi}_{L^\infty}^2\int_0^{t\wedge\tau_N^m}\abs{\bu_m(s)}^2_{L^2}ds\notag\\
		&\qquad{}+ \frac{1}{2}t+\frac{1}{2}\abs{\nabla\Phi}_{L^\infty}^2\bk^2\int_0^{t\wedge\tau_N^m}\abs{\bu_m(s)}^2_{L^2}ds\notag\\
		&\qquad{}+\int_0^{t\wedge\tau_N^m}\abs{g(\bu_m(s),c_m(s))}^2_{\mathcal{L}^2(\buc;H)}ds+2\int_0^{t\wedge\tau_N^m}(g(\bu_m(s),c_m(s)),\bu_m(s))dWs,\notag
	\end{align}
	with $\bk_4=\frac{8\bk_f\bk_{GN}\abs{c_0}^2_{L^\infty}}{3\xi}$. Multiplying this inequality by $\frac{\bk_4}{\eta}$, and adding the result with inequality \re{3.36*} after using the inequality \re{4.29}, we see that there exists positive constants $\bk_5$ and $\bk_6$ such that for all $t\in[0,T]$, $\mathbb{P}$-a.s.
	\begin{align}
		&	\bec (n_m,c_m,\bu_m)(t\wedge\tau_N^m)+\delta\int_0^{t\wedge\tau_N^m}\abs{\nabla\sqrt{n_m(s)}}^2_{L^2}ds
		\notag\\
		&\qquad{}+\int_0^{t\wedge\tau_N^m}\left[\frac{3\xi\bk_f}{2}\abs{\Delta c_m(s)}^2_{L^2}+\bk_4\abs{\nabla \bu_m(s)}^2_{L^2}+\abs{\sqrt{n_m(s)}\nabla c_m(s)}_{L^2}^2\right]ds\notag\\
		&\leq \bec (n_0,c_0,\bu_0)+\bk_5T+\bk_6\int_0^{t\wedge\tau_N^m}\abs{\bu_m(s)}^2_{L^2}ds+\gamma^2\bk_f\int_0^{t\wedge\tau_N^m}\abs{\nabla\phi(c_m(s))}^2_{\mathcal{L}^2(\mathbb{R}^2;L^2)}ds\label{3.50}\\
		&\qquad{}+\frac{2\bk_4}{\eta}\int_0^{t\wedge\tau_N^m}(g(\bu_m(s),c_m(s)),\bu_m(s))dW_s\notag\\
		&\qquad{}+\frac{\bk_4}{\eta}\int_0^{t\wedge\tau_N^m}\abs{g(\bu_m(s),c_m(s))}^2_{\mathcal{L}^2(\buc,H)}ds+2\gamma\bk_f\int_0^{t\wedge\tau_N^m}(\nabla\phi(c_m(s)),\nabla c_m(s))d\beta_s.\notag
	\end{align}
	Now, since  $\gamma$ satisfies the relation \re{4.46*}, taking into account the inequality \re{3.30**}, we note that
	\begin{align}
		\gamma^2\bk_f\abs{\nabla\phi(c_m)}^2_{\mathcal{L}^2(\mathbb{R}^2;L^2)}
		&\leq 2\gamma^2\bk_f\sum_{k=1}^2\int_\bo\abs{\nabla\sigma_k(x)\nabla c(x)}^2dx+2\gamma^2\bk_f\sum_{k=1}^2\int_\bo\abs{D^2c(x)\sigma_k(x)}^2dx\notag\\
		&\leq 2\gamma^2\bk_f\abs{\nabla c}^2_{L^2}\sum_{k=1}^2\abs{\sigma_k}^2_{W^{1,\infty}}+\abs{\Delta c}^2_{L^2}\frac{8\gamma^2\bk_f}{3}\sum_{k=1}^2\abs{\sigma_k}^2_{L^{\infty}}\label{3.51}\\
		&+\frac{8\gamma^2\bk_f\bk_2}{3}\abs{c_0}_{L^\infty}\sum_{k=1}^2\abs{\sigma_k}^2_{L^{\infty}}\notag\\
		&\leq \bk\abs{\nabla c}^2_{L^2}+\frac{\xi\bk_f}{2}\abs{\Delta c}^2_{L^2}+\bk.\notag
	\end{align}
	By the inequalities \re{2.5} and \re{3.20*}, we also note that
	\begin{align}
		\abs{g(\bu_m,c_m)}^2_{\mathcal{L}^2(\buc,H)}
		&\leq 2L^2_g\abs{(\bu_m,c_m)}^2_{\bhc}+2L^2_g\notag\\
		&\leq \bk\bec (n_m,c_m,\bu_m)+\bk\abs{c_0}_{L^\infty}^2+2L^2_g.\label{3.52}
	\end{align}
	From the estimates  \re{3.50} until \re{3.52}, we derive that
	\begin{align}
		&\be\sup_{0\leq s\leq T}	\bec (n_m,c_m,\bu_m)(s\wedge\tau^m_N)+\delta\be\int_0^{T\wedge\tau^m_N}\abs{\nabla\sqrt{n_m(s)}}^2_{L^2}ds
		\notag\\
		&\qquad{}+\be\int_0^{T\wedge\tau^m_N}\left[\xi\bk_f\abs{\Delta c_m(s)}^2_{L^2}+\bk_4\abs{\nabla \bu_m(s)}^2_{L^2}+\abs{\sqrt{n_m(s)}\nabla c_m(s)}_{L^2}^2\right]ds\notag\\
		&\leq \bec (n_0,c_0,\bu_0)+\bk T+\bk\be\int_0^{T\wedge\tau^m_N}\bec (n_m(s),c_m(s),\bu_m(s))ds+2L^2_gT\notag\\
		&\qquad{}+2\gamma\bk_f\be\sup_{0\leq s\leq T}	\abs{\int_0^{s\wedge\tau^m_N}(\nabla\phi(c_m(s)),\nabla c_m(s))d\beta_s}\label{3.54}\\
		&\qquad{}+\frac{2\bk_4}{\eta}\be\sup_{0\leq s\leq T}	\abs{\sum_{k=1}^\infty\int_0^{s\wedge\tau^m_N}(g(\bu_m(s),c_m(s))e_k,\bu_m(s))dW^k_s}.\notag
	\end{align}
	Now, by making use of the Burholder-Davis-Gundy, Cauchy-Schwarz, Young inequalities and the fact that $\gamma$ satisfies the  relation \re{4.46*}, we infer that
	\begin{align}
		&2\gamma\bk_f\be\sup_{0\leq s\leq T}	\abs{\int_0^{s\wedge\tau^m_N}(\nabla\phi(c_m(s)),\nabla c_m(s))d\beta_s}\notag\\
		&\leq \bk\be\left(\int_0^{T\wedge\tau^m_N}\abs{(\nabla\phi(c_m(s)),\nabla c_m(s))}^2ds\right)^{1/2}\notag\\
		&\leq \bk\be\left(\int_0^{T\wedge\tau^m_N}\abs{\nabla\phi(c_m(s))}^2_{\mathcal{L}^2(\mathbb{R}^2;L^2)}\abs{\nabla c_m(s)}_{L^2}^2ds\right)^{1/2}\notag\\
		&\leq\frac{\bk_f}{4}\be\sup_{0\leq s\leq T}\abs{\nabla c_m(s\wedge\tau^m_N)}_{L^2}^2+\bk\be\int_0^{T\wedge\tau^m_N}\abs{\nabla\phi(c_m(s))}^2_{\mathcal{L}^2(\mathbb{R}^2;L^2)}ds\notag\\
		&\leq\frac{1}{4}\be\sup_{0\leq s\leq T}	\bec (n_m(s),c_m(s),\bu_m(s))(s\wedge\tau^m_N)\notag\\
		&\qquad{}+\frac{\xi\bk_f}{2}\be\int_0^{T\wedge\tau^m_N}\abs{\Delta c_m(s)}^2_{L^2}ds+\bk\be\int_0^{T\wedge\tau^m_N}\abs{\nabla c_m(s)}^2_{L^2}ds+\bk T.\notag
	\end{align}
Similarly,
	\begin{align}
		&\frac{2\bk_4}{\eta}\be\sup_{0\leq s\leq T}	\abs{\sum_{k=1}^\infty\int_0^{s\wedge\tau^m_N}(g(\bu_m(s),c_m(s))e_k,\bu_m(s))dW^k_s}\notag\\
		&\leq\frac{1}{4}\be\sup_{0\leq s\leq T}\bec (n_m,c_m,\bu_m)(s\wedge\tau^m_N)+\bk\be\int_0^{T\wedge\tau^m_N}\abs{g(\bu_m(s),c_m(s))}^2_{\mathcal{L}^2(\buc;H)}ds\notag\\
		&\leq\frac{1}{4}\be\sup_{0\leq s\leq T}\bec (n_m,c_m,\bu_m)(s\wedge\tau^m_N)+\bk\be\int_0^{T\wedge\tau^m_N}\abs{(\bu_m(s),c_m(s))}^2_{\bhc}ds+ \bk TL_g^2\notag.
	\end{align}
	It follows from the estimates  \re{3.54} that
	\begin{align}
		&\be\sup_{0\leq s\leq T}	\bec (n_m,c_m,\bu_m)(s\wedge\tau^m_N)\notag\\
		&+\be\int_0^{T\wedge\tau^m_N}\left[\abs{\nabla\sqrt{n_m(s)}}^2_{L^2}
		+\bk_f\abs{\Delta c_m(s)}^2_{L^2}+\abs{\nabla \bu_m(s)}^2_{L^2}+\abs{\sqrt{n_m(s)}\nabla c_m(s)}_{L^2}^2\right]ds\\
		&\leq \bk\bec (n_0,c_0,\bu_0)+\bk T+\bk\be\int_0^{T\wedge\tau^m_N}\bec (n_m,c_m,\bu_m)(s)ds+\bk,\notag
	\end{align}
where $\bk$ is a constant depending on the initial data and  $T$ but independent  of $m$ and $N$. 
Now, the Gronwall lemma yields
	\begin{align}
		\be&\sup_{0\leq s\leq T}	\bec (n_m,c_m,\bu_m)(s\wedge\tau^m_N)\notag\\
	&\qquad	+\be\int_0^{T\wedge\tau^m_N}\left[\abs{\nabla\sqrt{n_m(s)}}^2_{L^2}
		+\abs{\Delta c_m(s)}^2_{L^2}+\abs{\nabla \bu_m(s)}^2_{L^2}+\abs{\sqrt{n_m(s)}\nabla c_m(s)}_{L^2}^2\right]ds\leq \bk,\notag
	\end{align}
from which  we  deduce the estimates \re{3.48} and hence completing the proof of   Lemma \ref{proposition 3.2}.
\end{proof}
\begin{lemma}
	\label{proposition 3.3}
	Under the same assumptions as in Lemma \ref{proposition 3.2}, for all $p\geq 1$, there exists a positive constant $\bk$ such that we have for all $m\in\mathbb{N}$ and $N\in\mathbb{N}$,
	\begin{equation}\label{3.61}
		\sup_{0\leq s\leq T}\abs{c_m(s\wedge\tau^m_N)}^{2p}_{L^2}+\left(\int_0^{T\wedge\tau^m_N}\abs{\nabla c_m(s)}_{L^2}^2ds\right)^p\leq \abs{\bo}^p\abs{c_0}^{2p}_{L^\infty},\qquad \mathbb{P}\text{-a.s.},\\
	\end{equation}
	\begin{equation}
		\begin{split}
			&\be\sup_{0\leq s\leq T}	\bec^p (n_m,c_m,\bu_m)(s\wedge\tau^m_N)+\be\left(\int_0^{T\wedge\tau^m_N}\abs{\nabla\sqrt{n_m(s)}}^2_{L^2}ds\right)^p\leq \bk,\label{3.62}\\
			&\text{and }\be\left(\int_0^{T\wedge\tau^m_N}\abs{\Delta c_m(s)}^2_{L^2}ds\right)^p+\be\left(\int_0^{T\wedge\tau^m_N}\abs{\nabla \bu_m(s)}^2_{L^2}ds\right)^p
			\leq\bk.
		\end{split}
	\end{equation}
\end{lemma}
\begin{proof}
	The inequality \re{3.61} follows directly from the estimates \re{3.47} of Lemma \ref{proposition 3.2}. Next, we are going to derive estimate \re{3.62}.  We start with  the inequality \re{3.54} and invoke the Jensen inequality to derive  that for all $p\geq2$,
	\begin{align}
		&\be\sup_{0\leq s\leq T}	\bec ^p(n_m,c_m,\bu_m)(s\wedge\tau^m_N)+\be\left(\int_0^{T\wedge\tau^m_N}\abs{\nabla\sqrt{n_m(s)}}^2_{L^2}ds\right)^p\notag\\
		&\qquad{}+\be\left(\int_0^{T\wedge\tau^m_N}\xi\bk_f\abs{\Delta c_m(s)}^2_{L^2}ds\right)^p+\be\left(\int_0^{T\wedge\tau^m_N}\bk_4\abs{\nabla \bu_m(s)}^2_{L^2}ds\right)^p\notag\\
		&\leq \bec ^p(n_0,c_0,\bu_0)+\bk T^p+\bk\be\left(\int_0^{T\wedge\tau^m_N}\bec (n_m,c_m,\bu_m)(s)ds\right)^p\label{3.63}\\
		&\qquad{}+\bk^p+2^p\gamma^p\bk_f^p\be\sup_{0\leq s\leq T}	\abs{\int_0^{s\wedge\tau^m_N}(\nabla\phi(c_m(s)),\nabla c_m(s))d\beta_s}^p\notag\\
		&\qquad{}+\bk\be\sup_{0\leq s\leq T}	\abs{\sum_{k=1}^\infty\int_0^{s\wedge\tau^m_N}(g(\bu_m(s),c_m(s))e_k,\bu_m(s))dW^k_s}^p.\notag
	\end{align}
	Invoking the  H\"older inequality, we see that
	\begin{equation*}
		\bk\be\left(\int_0^{T\wedge\tau^m_N}\bec (n_m,c_m,\bu_m)(s)ds\right)^p\leq \bk T^{\frac{p}{p-1}}\be\int_0^{T\wedge\tau^m_N}\bec^p (n_m,c_m,\bu_m)(s)ds.
	\end{equation*}
	Thanks to the Burkholder-Davis-Gundy inequality, we see that
	\begin{align}
		&\bk\be\sup_{0\leq s\leq T}	\abs{\sum_{k=1}^\infty\int_0^{s\wedge\tau^m_N}(g(\bu_m(s),c_m(s))e_k,\bu_m(s))dW^k_s}^p\notag\\
		&\leq \bk\be\sum_{k=1}^\infty\left(\int_0^{T\wedge\tau^m_N}\abs{(g(\bu_m(s),c_m(s))e_k,\bu_m(s))}^2ds\right)^{p/2}\notag\\
		&\leq \bk\be\sup_{0\leq s\leq T}\abs{ u_m(s\wedge\tau^m_N)}^p_{L^2}\left(\int_0^{T\wedge\tau^m_N}\abs{g(\bu_m(s),c_m(s))}^2_{\mathcal{L}^2(\buc;H)}ds\right)^{p/2}\notag\\
		&\leq\frac{1}{4}\be\sup_{0\leq s\leq T}\bec^p (n_m,c_m,\bu_m)(s\wedge\tau^m_N)+\bk\be\left(\int_0^{T\wedge\tau^m_N}\abs{g(\bu_m(s),c_m(s))}^2_{\mathcal{L}^2(\buc;H)}ds\right)^p\notag\\
		&\leq\frac{1}{4}\be\sup_{0\leq s\leq T}\bec^p (n_m,c_m,\bu_m)(s\wedge\tau^m_N)+\bk\be\int_0^{T\wedge\tau^m_N}\abs{(\bu_m(s),c_m(s))}^{2p}_{\bhc}ds+ \bk T^{\frac{p}{p-1}}L_g^{2p}\notag.
	\end{align}
	Taking into account the fact that $\gamma$ is sufficiently small such that the relation  \re{4.46*} is satisfied, we also arrive at
	\begin{align}
		&2^p\gamma^p\bk_f^p\bk\be\sup_{0\leq s\leq T}	\abs{\int_0^{s\wedge\tau^m_N}(\nabla\phi(c_m(s)),\nabla c_m(s))d\beta_s}^p\notag\\
		&\leq 2^p\gamma^p\bk_f^p\be\left(\int_0^{T\wedge\tau^m_N}\abs{\nabla\phi(c_m(s))}^2_{\mathcal{L}^2(\mathbb{R}^2;L^2)}\abs{\nabla c_m(s)}_{L^2}^2ds\right)^{p/2}\notag\\
		&\leq\frac{1}{4}\be\sup_{0\leq s\leq T}\abs{\nabla c_m(s\wedge\tau^m_N)}_{L^2}^{2p}+2^{2p}\gamma^{2p}\bk_f^{2p}\be\left(\int_0^{T\wedge\tau^m_N}\abs{\nabla\phi(c_m(s))}^2_{\mathcal{L}^2(\mathbb{R}^2;L^2)}ds\right)^p\notag\\
		&\leq\frac{1}{4}\be\sup_{0\leq s\leq T}	\bec^p (n_m,c_m,\bu_m)(s\wedge\tau^m_N)\notag\\
		&\qquad{}+\frac{1}{2}\be\left(\int_0^{T\wedge\tau^m_N}\xi\bk_f\abs{\Delta c_m(s)}^2_{L^2}ds\right)^p+\bk T^{\frac{p}{p-1}}\be\int_0^{T\wedge\tau^m_N}\abs{\nabla c_m(s)}^{2p}_{L^2}ds+\bk T^p.\notag
	\end{align}
	It follows from the estimates \re{3.63}  that
	\begin{align}
		&\be\sup_{0\leq s\leq T}	\bec ^p(n_m,c_m,\bu_m)(s\wedge\tau^m_N)+\be\left(\int_0^{T\wedge\tau^m_N}\abs{\nabla\sqrt{n_m(s)}}^2_{L^2}ds\right)^p\notag\\
		&\qquad{}+\be\left(\int_0^{T\wedge\tau^m_N}\abs{\Delta c_m(s)}^2_{L^2}ds\right)^p+\be\left(\int_0^{T\wedge\tau^m_N}\abs{\nabla \bu_m(s)}^2_{L^2}ds\right)^p\notag\\
		&\leq \bk\bec ^p(n_0,c_0,\bu_0)+\bk T^p+\bk\be\int_0^{T\wedge\tau^m_N}\bec^p (n_m,c_m,\bu_m)(s)ds+\bk.\notag
	\end{align}
	Now, the Gronwall lemma yields
	\begin{align}
		&\be\sup_{0\leq s\leq T}	\bec ^p(n_m,c_m,\bu_m)(s\wedge\tau^m_N)+\be\left(\int_0^{T\wedge\tau^m_N}\abs{\nabla\sqrt{n_m(s)}}^2_{L^2}ds\right)^p\notag\\
		&\qquad{}+\be\left(\int_0^{T\wedge\tau^m_N}\abs{\Delta c_m(s)}^2_{L^2}ds\right)^p+\be\left(\int_0^{T\wedge\tau^m_N}\abs{\nabla \bu_m(s)}^2_{L^2}ds\right)^p\leq \bk,\notag
	\end{align}
and the estimate \re{3.62} follows directly from this last inequality. This completes the proof of Lemma \ref{proposition 3.3}.
\end{proof}
In order to control the process $t\mapsto n_m(t\wedge\tau_N^m)$, we prove the following lemma. 
\begin{lemma}\label{lemma3.5}
	Under the same assumptions as in Lemma \ref{proposition 3.2}, there exists a positive constant $\eta_0>1$  such that for all $m\in \mathbb{N}$, $N\in \mathbb{N}$ and   $\mathbb{P}\text{-a.s.}$,
	\begin{align}
		\sup_{0\leq s\leq T}	\abs{n_m(s\wedge\tau^m_N)}^2_{L^2}+\int_0^{T\wedge\tau^m_N}\abs{ n_m(s)}_{H^1}^2ds\leq \eta_0\exp\left(\bk	\int_0^{T\wedge\tau^m_N}\abs{ \nabla c_m(s)}_{L^4}^{4}ds\right). \label{3.73*}
	\end{align}
\end{lemma}
\begin{proof}
Let $t\in [0,T]$ be arbitrary but fixed. Multiplying the last equation of \re{3.40} by $n_m(s\wedge\tau^m_N)$ for $0\leq s\leq t$, and using the fact that $\nabla\cdot \bu_m=0$ and  the inequality \re{4.4.} as well as the  H\"older inequality and the Young inequality,  we obtain
	\begin{align}
		&\frac{1}{2}\frac{d}{dt}\abs{n_m(s\wedge\tau^m_N)}^2_{L^2}+\delta\abs{\nabla n_m(s\wedge\tau^m_N)}_{L^2}^2\notag\\
		&=\xi\int_\bo n_m(s\wedge\tau^m_N,x)\nabla c_m(s\wedge\tau^m_N,x)\cdot\nabla n_m(s\wedge\tau^m_N,x)dx\notag\\
		&\leq \xi\abs{n_m(s\wedge\tau^m_N)}_{L^4}\abs{\nabla c_m(s\wedge\tau^m_N)}_{L^4}\abs{\nabla n_m(s\wedge\tau^m_N)}_{L^2}\notag\\
		&\leq \bk(\abs{n_m(s\wedge\tau^m_N)}^{1/2}_{L^2}\abs{\nabla n_m(s\wedge\tau^m_N)}^{1/2}_{L^2}+\abs{n_m(s\wedge\tau^m_N)}_{L^2})\abs{\nabla c_m(s\wedge\tau^m_N)}_{L^4}\abs{\nabla n_m(s\wedge\tau^m_N)}_{L^2}\notag\\
		&\leq \bk\abs{n_m(s\wedge\tau^m_N)}^{1/2}_{L^2}\abs{\nabla c_m(s\wedge\tau^m_N)}_{L^4}\abs{\nabla n_m(s\wedge\tau^m_N)}^{3/2}_{L^2}\notag\\
		&\qquad+\bk\abs{n_m(s\wedge\tau^m_N)}_{L^2}\abs{\nabla c_m(s\wedge\tau^m_N)}_{L^4}\abs{\nabla n_m(s\wedge\tau^m_N)}_{L^2}\notag\\
		&\leq \frac{\delta}{2}\abs{\nabla n_m(s\wedge\tau^m_N)}^2_{L^2}+\bk\abs{n_m(s\wedge\tau^m_N)}^{2}_{L^2}(\abs{\nabla c_m(s\wedge\tau^m_N)}^4_{L^4}+\abs{\nabla c_m(s\wedge\tau^m_N)}^2_{L^4})\notag\\
		&\leq \frac{\delta}{2}\abs{\nabla n_m(s\wedge\tau^m_N)}^2_{L^2}+\bk\abs{n_m(s\wedge\tau^m_N)}^{2}_{L^2}\left(\abs{\nabla c_m(s\wedge\tau^m_N)}^4_{L^4}+1\right).\notag
	\end{align}
	This implies that for all $t\in [0,T]$,
	\begin{align}
		\sup_{0\leq s\leq t}\abs{n_m(s\wedge\tau^m_N)}^2_{L^2}+\delta\int_0^{t\wedge\tau^m_N}\abs{\nabla n_m(s)}_{L^2}^2ds\leq \abs{n^m_0}_{L^2}^2+\bk\int_0^{t\wedge\tau^m_N}\abs{n_m(s)}_{L^2}^{2}\left(\abs{\nabla c_m(s)}^4_{L^4}+1\right)ds.\notag
	\end{align}
Since $n_m^0\to n_0$ in $L^2(\bo)$, $\abs{n_0^m}_{L^2}^2$ is uniformly bounded.  Thus, applying the Gronwall lemma, we obtain that
	\begin{align}
		\sup_{0\leq s\leq t}	\abs{n_m(s\wedge\tau^m_N)}^2_{L^2}+\int_0^{t\wedge\tau^m_N}\abs{ n_m(s)}_{H^1}^2ds&\leq \bk_\delta\exp\left(\bk\int_0^{t\wedge\tau^m_N}\left(\abs{\nabla c_m(s)}^4_{L^4}+1\right)ds\right)\notag\\
		&\leq (\bk_\delta+1)e^{\bk T}\exp\left(\bk	\int_0^{t\wedge\tau^m_N}\abs{\nabla c_m(s)}^4_{L^4}ds\right),\notag
	\end{align}
and  complete the proof of  Lemma \ref{lemma3.5}.
\end{proof}
\begin{cor}\label{cor4.11*}
	Under the same assumptions as in Lemma \ref{proposition 3.2}, for any $p\geq1$, there exists a positive constant $\bk$  such that for all $m\in \mathbb{N}$ and $N\in \mathbb{N}$,
	\begin{align}
		&\be\sup_{0\leq s\leq T}	\abs{\bu_m(s\wedge\tau^m_N)}^{2p}_{L^2}+\be\left(\int_0^{T\wedge\tau^m_N}\abs{\nabla \bu_m(s)}_{L^2}^2ds\right)^p\leq \bk,\label{3.68*}\\
		& 
		\qquad\be\left(\int_0^{T\wedge\tau^m_N}\abs{n_m(s)}^2_{L^2}ds\right)^p\leq \bk,\label{3.69*}\\
		&\be\sup_{0\leq s\leq {T}}	\abs{c_m(s\wedge\tau^m_N)}^{2p}_{H^1}+\be\left(\int_0^{T\wedge\tau^m_N}\abs{c_m(s)}_{H^2}^2ds\right)^p\leq \bk\label{3.70*},\\
		&\qquad\be\int_0^{T\wedge\tau^m_N}\abs{\nabla c_m(s)}^4_{L^2}ds\leq \bk.\label{4.75*}
	\end{align}
\end{cor}
\begin{proof}
	The estimate \re{3.68*} is a consequence of the  estimates  \re{3.48} and \re{3.62}. From the inequalities \re{4.53}, \re{3.48} and \re{3.62}, we infer that
	\begin{equation*}
		\be\left(\int_0^{T\wedge\tau^m_N}\abs{n_m(s)}^2_{L^2}ds\right)^p\leq\be\left(\int_0^{T\wedge\tau^m_N}\left(\bk_{GN}\bk^{1/2}\abs{\nabla\sqrt{n_m(s)}}^2_{L^2}+ \bk\right)ds\right)^p\leq \bk,
	\end{equation*}
	which proves the second estimate of inequality \re{3.69*}.
	
	According to  \cite[Proposition 7.2, P. 404]{Tay}, we have
	\begin{equation*}
		\abs{c_m}^2_{H^2}\leq\bk(\abs{\Delta c_m}^2_{L^2}+\abs{c_m}_{H^1}^2),
	\end{equation*}
	from which along with  \re{3.48} and \re{3.62} we deduce \re{3.70*}.
	
By applying the  inequality \re{4.4.}, we obtain that $$\abs{\nabla c_m}_{L^4}^4\leq \bk(\abs{c_m}_{H^2}^2\abs{\nabla c_m}^2_{L^2}+\abs{\nabla c_m}^4_{L^2}).$$ 
Therefore, 
	\begin{align}
		\be\int_0^{T\wedge\tau^m_N}\abs{\nabla c_m(s)}_{L^4}^4ds&\leq \bk\be\int_0^{T\wedge\tau^m_N}\abs{c_m(s)}_{H^2}^2\abs{\nabla c_m(s)}^2_{L^2}ds+\bk\be\int_0^{T\wedge\tau^m_N}\abs{\nabla c_m(s)}^4_{L^2}ds\notag\\
		&\leq\bk\be\sup_{0\leq s\leq T}	\abs{c_m(s\wedge\tau^m_N)}^2_{H^1}\int_0^T\abs{c_m(s)}_{H^2}^2ds+\bk T\be\sup_{0\leq s\leq T}	\abs{c_m(s\wedge\tau^m_N)}^4_{H^1}\notag\\
		&\leq\bk \be\sup_{0\leq s\leq T}	\abs{c_m(s\wedge\tau^m_N)}^4_{H^1}+\bk\be\left(\int_0^{T\wedge\tau^m_N}\abs{c_m(s)}_{H^2}^2ds\right)^2,\notag
	\end{align}
from which along with \re{3.70*}  we deduce  \re{4.75*}. This completes  the proof of Corollary \ref{cor4.11*}.
\end{proof}
In the following lemma, we state and prove a result concerning the stopping time $\tau^m_N$. More precisely, we prove that $\Sup_{N\in\mathbb{N}}\tau_m^N\geq2T$ with probability $1$ such that the inequality \re{4.14***} holds.
\begin{lemma}\label{lem4.7*}
Let $\tau^m_N$, $m,N\in \mathbb{N}$ be the stopping times defined in \re{3.46}. Then, under the same assumptions as in Lemma \ref{proposition 3.2}, it holds that 
\begin{equation}
	\mathbb{P}\left\{ \omega\in \Omega: \sup_{N\in\mathbb{N}}\tau_m^N(\omega)\geq2T \right\}=1.\label{4.81*}
\end{equation}
Consequently, the solutions $(\bu_m,c_m,n_m)$ of system \re{3.40} exist almost surely for every $t\in [0,T]$.
\end{lemma}
\begin{proof}
We notice that the inequalities of Corollary \ref{cor4.11*} hold for every $T>0$. Hence, for a fixed $T>0$, we set $\tilde{T}=2T$ and note that for all $J\in \mathbb{N}$,  $$\left\{\omega\in \Omega: \sup_{N\in\mathbb{N}}\tau_m^N(\omega)<\tilde{T} \right\}\subset\left\{\omega\in \Omega: \tau_m^J(\omega)<\tilde{T} \right\},$$
which implies that
\begin{equation}
\mathbb{P}\left\{ \omega\in \Omega: \sup_{N\in\mathbb{N}}\tau_m^N(\omega)<2T \right\}\leq\lim_{N\longrightarrow\infty}\mathbb{P}\left\{\omega\in \Omega: \tau_m^N(\omega)<\tilde{T} \right\},\label{4.82*}
\end{equation}
and therefore, it is enough to show that the second term of the right hand side of this last equality converges to zero as $N\to \infty$. To this end, let $$A_N=\left\{ \omega\in \Omega: \tau_m^N<\tilde{T} \right\}$$ and $$B_N= \left\{ \omega\in \Omega:\abs{n_m(\tilde{T}\wedge\tau_m^N	)}^2_{L^2}+\abs{\bu_m(\tilde{T}\wedge\tau_m^N	)}_{L^2}^2+\abs{c_m(\tilde{T}\wedge\tau_m^N)}_{H^1}^2\geq N^2 
 \right\}. $$ Then, we have  $A_N\subset B_N$ for $N>\tilde{T}$. Indeed, let $\omega\in A_N$, then $\tilde{T}\wedge\tau_m^N(\omega)=\tau_m^N(\omega)$. Thus,  by the definition of the stopping time $\tau_m^N$, we see that for $N>\tilde{T}$,
\begin{align*}
	\abs{n_m(\tilde{T}\wedge\tau_m^N	)}^2_{L^2}+\abs{\bu_m(\tilde{T}\wedge\tau_m^N	)}_{L^2}^2+\abs{c_m(\tilde{T}\wedge\tau_m^N)}_{H^1}^2
	&=\abs{n_m(\tau_m^N)}^2_{L^2}+\abs{\bu_m(\tau_m^N	)}_{L^2}^2+\abs{c_m(\tau_m^N	)}_{H^1}^2\\
	&\geq N^2.
\end{align*}
We then conclude that $\omega\in B_N$.  

Now, for $N>\tilde{T}$, using the inclusion $A_N\subset B_N$  we derive that  
\begin{align}
\mathbb{P}\left\{ \omega\in \Omega: \tau_m^N	<\tilde{T} \right\}&\leq \mathbb{P}\left\{ \omega\in \Omega:\abs{n_m(\tilde{T}\wedge\tau_m^N	)}^2_{L^2}+\abs{\bu_m(\tilde{T}\wedge\tau_m^N	)}_{L^2}^2+\abs{c_m(\tilde{T}\wedge\tau_m^N)}_{H^1}^2\geq N^2 
\right\}\notag\\
&\leq \mathbb{P}\left\{ \omega\in \Omega:\abs{n_m(\tilde{T}\wedge\tau_m^N	)}^2_{L^2}\geq \frac{N^2}{3} 
\right\}+\mathbb{P}\left\{ \omega\in \Omega:\abs{\bu_m(\tilde{T}\wedge\tau_m^N	)}_{L^2}^2\geq \frac{N^2}{3} 
\right\}\label{4.83*}\\
&\qquad{}+\mathbb{P}\left\{ \omega\in \Omega:\abs{c_m(\tilde{T}\wedge\tau_m^N	)}_{H^1}^2\geq \frac{N^2 }{3}
\right\}.\notag
\end{align}
According to the estimates \re{3.68*} and \re{3.70*} of Corollary \ref{cor4.11*} as well as the Markov inequality, we derive that for $N>\tilde{T}$
\begin{align}
\mathbb{P}\left\{ \omega\in \Omega:\abs{c_m(\tilde{T}\wedge\tau_m^N	)}_{H^1}^2\geq \frac{N^2 }{3}\right\}&\leq \mathbb{P}\left\{ \omega\in \Omega:\sup_{0\leq s\leq {\tilde{T}}}\abs{c_m(s\wedge\tau^m_N)}_{H^1}^2\geq \frac{N^2 }{3}\right\}\notag\\
&\leq \frac{3}{N^2 }\be\sup_{0\leq s\leq {\tilde{T}}}\abs{c_m(s\wedge\tau^m_N)}_{H^1}^2\notag\\
&\leq \frac{\bk}{N^2 },\notag 
\end{align}
and 
\begin{align}
	\mathbb{P}\left\{ \omega\in \Omega:\abs{\bu_m(\tilde{T}\wedge\tau_m^N	)}_{L^2}^2\geq \frac{N^2 }{3}\right\}&\leq \mathbb{P}\left\{ \omega\in \Omega:\sup_{0\leq s\leq {\tilde{T}}}\abs{\bu_m(s\wedge\tau^m_N)}_{L^2}^2\geq \frac{N^2 }{3}\right\}\notag\\
	&\leq \frac{3}{N^2 }\be\sup_{0\leq s\leq {\tilde{T}}}\abs{\bu_m(s\wedge\tau^m_N)}_{L^2}^2\notag\\
	&\leq \frac{\bk}{N^2 }.\notag 
\end{align}
Also for $N>\max(\sqrt{3\eta_0},\tilde{T})$ (where $\eta_0$ is a constant obtained in Lemma \ref{lemma3.5}), we use the inequality \re{3.73*} of Lemma \ref{lemma3.5} to infer that
\begin{align}
\mathbb{P}\left\{ \omega\in \Omega:\abs{n_m(\tilde{T}\wedge\tau_m^N	)}^2_{L^2}\geq \frac{N^2}{3} 
\right\}&\leq \mathbb{P}\left\{ \omega\in \Omega:\sup_{0\leq s\leq {\tilde{T}}}\abs{n_m(s\wedge\tau^m_N)}^2_{L^2}\geq \frac{N^2}{3} 
\right\}\notag\\
&\leq \mathbb{P}\left\{ \omega\in \Omega:\eta_0\exp\left(\bk	\int_0^{\tilde{T}\wedge\tau^m_N}\abs{ \nabla c_m(s)}_{L^4}^{4}ds\right)\geq \frac{N^2}{3} 
\right\}\notag\\
&\leq \mathbb{P}\left\{ \omega\in \Omega:	\int_0^{\tilde{T}\wedge\tau^m_N}\abs{ \nabla c_m(s)}_{L^4}^{4}ds\geq \frac{\ln(\frac{N^2}{3\eta_0}) }{\bk}
\right\}\notag.
\end{align}
Invoking the Markov inequality and using the estimate \re{4.75*} of Corollary \ref{cor4.11*}, we see that 
\begin{align}
\mathbb{P}\left\{ \omega\in \Omega:\abs{n_m(\tilde{T}\wedge\tau_m^N	)}^2_{L^2}\geq \frac{N^2}{3} 
\right\}&\leq  \frac{\bk}{\ln(\frac{N^2}{3\eta_0}) }\be \int_0^{\tilde{T}\wedge\tau^m_N}\abs{ \nabla c_m(s)}_{L^4}^{4}ds\notag\\
&\leq \frac{\bk}{2\ln(N)-\ln(3\eta_0) }.\notag
\end{align}
Plugging these inequalities  into the inequality \re{4.83*}, we arrive at $$\mathbb{P}\left\{ \omega\in \Omega: \tau_m^N	<\tilde{T} \right\}\leq \frac{\bk}{N^2}+\frac{\bk}{2\ln(N)-\ln(\eta_0)-\ln(3)},$$ for all for $N>\max(\sqrt{3\eta_0},\tilde{T})$. Letting $N$ to infinity in this last  inequality we get $$\lim_{N\longrightarrow\infty}\mathbb{P}\left\{ \omega\in \Omega: \tau_m^N	<\tilde{T} \right\}=0,$$
which along with \re{4.82*} imply \re{4.81*}.

By the equality \re{4.81*} we infer the inequality \re{4.14***} and therefore, the relation \re{4.10} hold
and the lemma is then proved. 
\end{proof}
Since  $(T\wedge \tau_m^{N} )_{N\in\mathbb{N}}$ is increasing, we have $T\wedge \tau_m^{N}\to T$ a.s., as $N\to \infty$. 
With this almost surely convergence in hand, we are going  to give some consequences of Lemma \ref{lemma3.5} and Corollary \ref{cor4.11*}. 
\begin{cor}\label{cor4.11}
	Under the same assumptions as in Lemma \ref{proposition 3.2}, for any $p\geq1$, there exists a positive constant $\bk$  such that for all $m\in \mathbb{N}$,
	\begin{align}
		&	\sup_{0\leq s\leq T}	\abs{n_m(s)}^2_{L^2}+\int_0^{T}\abs{ n_m(s)}_{H^1}^2ds\leq \eta_0\exp\left(\bk	\int_0^{T}\abs{ \nabla c_m(s)}_{L^4}^{4}ds\right), \  \mathbb{P}\text{-a.s.}\label{3.73}\\
		&\be\sup_{0\leq s\leq T}	\abs{\bu_m(s)}^{2p}_{L^2}+\be\left(\int_0^T\abs{\nabla \bu_m(s)}_{L^2}^2ds\right)^p\leq \bk,\label{3.68}\\
		& 
		\qquad\be\left(\int_0^T\abs{n_m(s)}^2_{L^2}ds\right)^p\leq \bk,\label{3.69}\\
		&\be\sup_{0\leq s\leq T}	\abs{c_m(s)}^{2p}_{H^1}+\be\left(\int_0^T\abs{c_m(s)}_{H^2}^2ds\right)^p\leq \bk\label{3.70},\\
		&\qquad\be\int_0^T\abs{\nabla c_m(s)}^4_{L^2}ds\leq \bk,\label{4.75}
	\end{align}
where $\eta_0>1$ is a constant obtained in Lemma \ref{lemma3.5}.
\end{cor}
\begin{proof}
Since $T\wedge \tau_m^{N}\to T$ a.s., as $N\to \infty$, by the 
path continuity of the process $t\mapsto(\bu_m(t), c_m(t),n_m(t))$,  we can let $N\to \infty$ in the inequality \re{3.73*} of Lemma \ref{lemma3.5} and derive the inequality \re{3.73}. In addition to the almost surely convergence of $T\wedge \tau_m^{N}$ to $T$ and  the 
path continuity of the process $t\mapsto(\bu_m(t), c_m(t),n_m(t))$, we invoke the Fatou lemma and pass to the limit as $N\to \infty$ in the inequalities \re{3.68*}, \re{3.69*}, \re{3.70*} and \re{4.75*} and derive the estimate \re{3.68}, \re{3.69}, \re{3.70} and \re{4.75}.
\end{proof}
\begin{cor}
	Under the same assumptions as in Lemma \ref{proposition 3.2}, there exists a positive constant $\bk$ such that for all $m\in \mathbb{N}$,
	\begin{equation}\label{3.77}
		\be \abs{n_m}^2_{C^{1/2}([0,T]; H^{-3})}\leq \bk.
	\end{equation}
\end{cor}
\begin{proof}
	Let $v\in H^{3}(\bo)$. We recall that $\abs{\nabla v}_{L^\infty}\leq \bk\abs{v}_{H^3}$.  So, using an integration by part and the  H\"older inequality, we derive that
	\begin{equation*}
		\abs{(A_1n_m,v)}= \abs{(n_m,\Delta v)}\leq \abs{n_m}_{L^2}\abs{\Delta v}_{L^2}\leq \abs{n_m}_{L^2}\abs{ v}_{H^3},
	\end{equation*}	
	\begin{align*}
		\abs{(\bp_m^1B_1(\bu_m,n_m),v)}&= \abs{(B_1(\bu_m,n_m),\bp_m^1v)}
		\\
		&= \abs{(n_m\bu_m,\nabla\bp_m^1v)}\\
		&\leq \bk\abs{n_m}_{L^2}\abs{\bu_m}_{L^2}\abs{\nabla\bp_m^1v}_{L^\infty}\\
		&\leq \bk\abs{n_m}_{L^2}\abs{\bu_m}_{L^2}\abs{v}_{H^3},
	\end{align*}
	and
	\begin{align*}
		\abs{(\bp_m^1R_2(n_m,c_m),v)}	&= \xi\abs{(n_m\nabla c_m,\nabla\bp_m^1v)}\\
		&\leq\bk \abs{n_m}_{L^2}\abs{\nabla c_m}_{L^2}\abs{\nabla\bp_m^1v}_{L^\infty}\\
		&\leq\bk \abs{n_m}_{L^2}\abs{\nabla c_m}_{L^2}\abs{v}_{H^3}.
	\end{align*}	
	Due to the continuous Sobolev embeddings $W^{1,2}(0,T; H^{-3 }(\bo))\hookrightarrow C^{1/2}(0,T; H^{-3 }(\bo))$, and $L^2(\bo)\hookrightarrow H^{-3 }(\bo)$, we have
	\begin{align}
		\be \abs{n_m}^2_{C^{1/2}(0,T; H^{-3})}&\leq\be \abs{n_m}^2_{W^{1,2}(0,T; H^{-3})}\notag\\
		&=\be\int_0^T\abs{n_m(s)}^2_{H^{-3}}ds+\be\int_0^T\abs{\frac{d}{dt}n_m(s)}^2_{H^{-3}}ds\notag\\
		&\leq\bk\be\int_0^T\abs{n_m(s)}^2_{L^2}ds+\be\int_0^T\abs{\frac{d}{dt}n_m(s)}^2_{H^{-3}}ds.\notag
	\end{align}
	Using the estimates  \re{3.68}, \re{3.69} and \re{3.70}, we arrive at
	\begin{align}
		\be \abs{n_m}^2_{C^{1/2}(0,T; H^{-3}))}
		&\leq \bk +\bk\be\int_0^T\abs{A_1n_m(s)}^2_{H^{-3}}ds\notag\\
		&\qquad{}+\bk\be\int_0^T\left[\abs{\bp_m^1B_1(\bu_m(s),n_m(s))}^2_{H^{-3}}+\abs{\bp_m^1R_2(n_m(s),c_m(s))}^2_{H^{-3}}\right]ds\notag\\
		&\leq \bk +\bk\be\int_0^T\left[\abs{n_m(s)}^2_{L^{2}}+\abs{u_m(s)}^2_{L^2}\abs{n_m(s)}^2_{L^2}+\abs{n_m(s)}^2_{L^2}\abs{\nabla c_m(s)}^2_{L^2}\right]ds\notag\\
		&\leq \bk +\bk\be\sup_{0\leq s\leq T}\abs{\bu_m(s)}^2_{L^2}\int_0^T\abs{n_m(s)}^2_{L^2}ds+\bk\be\sup_{0\leq s\leq T}\abs{\nabla c_m(s)}^2_{L^2}\int_0^T\abs{n_m(s)}^2_{L^2}ds\notag\\
		&\leq \bk+\bk\be\sup_{0\leq s\leq T}\abs{\bu_m(s)}^4_{L^2}+\bk\be\sup_{0\leq s\leq T}\abs{\nabla c_m(s)}^4_{L^2}+\bk\be\left(\int_0^T\abs{n_m(s)}^2_{L^2}ds\right)^2\leq \bk.\notag
	\end{align}
\end{proof}
\begin{lemma}\label{proposition4.4}
	Under the same assumptions as in Lemma \ref{proposition 3.2}, there exists a positive constant $\bk$ such that for all $m\in \mathbb{N}$,
	\begin{equation}\label{3.78}
		\begin{split}
			&\be\int_0^T\left[\abs{A_1c_m(s)}_{L^2}^2+\abs{\bp_m^2B_1(\bu_m(s), c_m(s))}^2_{L^2}+\abs{\bp_m^2R_1(n_m(s),c_m(s))}_{L^2}^2\right]ds\leq \bk,\\
			&\be\int_0^T\left[\abs{A_0\bu_m(s)}_{V^*}^2+\abs{\bp_m^2B_0(\bu_m(s), \bu_m(s))}^2_{V^*}+\abs{\bp_m^2R_0(n_m(s),\varPhi)}_{V^*}^2\right]ds\leq \bk.
		\end{split}
	\end{equation}
\end{lemma}
\begin{proof}
	Thanks to the inequalities \re{3.68}, \re{3.69} and \re{3.70} once more, we note that
	\begin{equation*}
		\be\int_0^T \abs{A_1c_m(s)}_{L^2}^2ds= \be\int_0^T\abs{\Delta c_m(s)}^2_{L^2}ds\leq \bk\be\int_0^T\abs{ c_m(s)}^2_{H^2}ds\leq \bk,
	\end{equation*}
	and
	\begin{align*}
		\be\int_0^T\abs{\bp_m^2B_1(\bu_m(s), c_m(s))}_{L^2}^2ds&\leq \bk\be\int_0^T\abs{\bu_m(s)\cdot\nabla c_m(s)}^2_{L^2}ds\\
		&\leq\bk\be\sup_{0\leq s\leq T}\abs{\bu_m(s)}^2_{L^2}\int_0^T\abs{\nabla c_m(s)}^2_{L^2}\\
		&\leq \bk\be\sup_{0\leq s\leq T}\abs{u_m(s)}^4_{L^2}+\bk\be\left(\int_0^T\abs{\nabla c_m(s)}^2_{L^2}ds\right)^2\leq \bk,
	\end{align*}
	as well as
	\begin{align*}
		\be\int_0^T\abs{ \bp_m^2R_1(n_m(s),c_m(s))}^2_{L^2}ds&\leq \bk\be\int_0^T\abs{n_m(s)f(c_m(s))}^2_{L^2}ds\\
		&\leq\bk\sup_{0\leq s\leq \abs{c_0}_{L^\infty}}f^2(s)\be\int_0^T\abs{n_m(s)}^2_{L^2}\leq \bk,
	\end{align*}
	and
	\begin{equation*}
		\be\int_0^T \abs{A_0\bu_m(s)}_{V^*}^2ds\leq  \be\int_0^T\abs{\nabla \bu_m(s)}^2_{L^2}ds\leq \bk.
	\end{equation*}
	In the same way,
	\begin{align*}
		\be\int_0^T\abs{\bp_m^2B_0(\bu_m(s), \bu_m(s))}_{V^*}^2ds&\leq \bk\be\int_0^T\abs{\bu_m}_{L^2}^2\abs{\nabla\bu_m(s)}_{L^2}^2ds\\
		&\leq\bk\be\sup_{0\leq s\leq T}\abs{\bu_m(s)}^2_{L^2}\int_0^T\abs{\nabla \bu_m(s)}^2_{L^2}ds\\
		&\leq \bk\be\sup_{0\leq s\leq T}\abs{u_m(s)}^4_{L^2}+\bk\be\left(\int_0^T\abs{\nabla \bu_m(s)}^2_{L^2}ds\right)^2\leq \bk,
	\end{align*}
	and
	\begin{equation*}
		\be\int_0^T \abs{R_0(n_m(s),\varPhi)}_{V^*}^2ds\leq  \abs{\varPhi}^2_{W^{1,\infty}}\be\int_0^T\abs{ n_m(s)}^2_{L^2}ds\leq \bk.
	\end{equation*}
	Combining all these inequalities, we obtain the relation \re{3.78}.
\end{proof}

\subsection{Tightness result and passage to the limit}
This subsection is devoted to the study of the tightness of the approximations solutions and the proof of several convergences which will enable us to pass to the limit and construct a weak probabilistic solution to our problem via the martingale representation theorem given in \cite[Theorem 8.2]{Da}. For this purpose,  we consider the following  spaces:
\begin{equation}\label{3.79}
	\begin{split}
		&\mathcal{Z}_{n}=L_{w}^{2}(0,T;H^1(\bo))\cap L^{2}(0,T;L^2(\bo))\cap\mathcal{C}([0,T];H^{-3}(\bo))\cap\mathcal{C}([0,T];L^2_{w}(\bo)),\\
		&\mathcal{Z}_{\bu}=L_{w}^{2}(0,T;V)\cap L^{2}(0,T;H)\cap\mathcal{C}([0,T];V^*)\cap\mathcal{C}([0,T];H_{w}),\\
		&\mathcal{Z}_{c}=L_{w}^{2}(0,T;H^2(\bo))\cap L^{2}(0,T;H^1(\bo))\cap\mathcal{C}([0,T];L^2(\bo))\cap\mathcal{C}([0,T];H_w^1(\bo)),\\
		&  
		\qquad\qquad\qquad\qquad\qquad\qquad\qquad \bz=\bz_n\times\bz_\bu\times\bz_c.
	\end{split}
\end{equation}
By making appropriate use of Lemma \ref{lemma 3.3}, Corollary \ref{cor3.1}, and Corollary \ref{cor3.2}, we will now show that the sequence of  probability law $\bl_m=\bl(n_m)\times\bl(\bu_m)\times\bl(c_m)$ is tight in $\bz$.
\begin{lemma}
	We suppose that the hypotheses of Proposition \ref{proposition 3.2} hold. Then the family of probability laws $(\bl_m)_{m\in \mathbb{N}}$ is tight on the space $\bz$.
\end{lemma}
\begin{proof}
	We firstly prove that $(\bl(n_m))_m$ is tight on $\bz_n$. For any $\eps>0$ we set $\bk_\eps=\eta_0e^{\bk/\eps}>\eta_0$ where $\eta_0>1$ is given by Lemma \ref{lemma3.5}. From  the inequality \re{3.73}, we deduce that
	\begin{align*}
		\sup_m\mathbb{P}\left\{\omega\in \Omega:\  \abs{n_m}^2_{L^\infty(0,T;L^2)}>\bk_\eps\right\}&\leq \sup_m\mathbb{P}\left\{\omega\in \Omega:\  \eta_0\exp\left(\bk	\int_0^T\abs{ \nabla c_m(s)}_{L^4}^{4}ds\right)>\bk_\eps\right\}\\
		&\leq \sup_m\mathbb{P}\left\{\omega\in \Omega:\  \bk	\int_0^T\abs{ \nabla c_m(s)}_{L^4}^{4}ds>\ln\left(\frac{\bk_\eps}{\eta_0}\right)\right\}.
	\end{align*}
	Using the Markov inequality and inequality \re{4.75}, we infer that
	\begin{align}
		\sup_m\mathbb{P}\left\{\omega\in \Omega:\  \abs{n_m}^2_{L^\infty(0,T;L^2)}>\bk_\eps\right\}&\leq\frac{1}{\ln\left(\frac{\bk_\eps}{\eta_0}\right)}\be \left(\bk	\int_0^T\abs{ \nabla c_m(s)}_{L^4}^{4}ds\right)\notag\\
		&\leq \frac{\eps}{\bk}\be \left(\bk	\int_0^T\abs{ \nabla c_m(s)}_{L^4}^{4}ds\right)\notag\\
		&\le \eps.\notag
	\end{align}
	Similarly, we can also prove that
	\begin{align}
		\sup_m\mathbb{P}\left\{\omega\in \Omega:\  \abs{n_m}^2_{L^2(0,T;H^1)}>\bk_\eps\right\}&\leq \sup_m\mathbb{P}\left\{\omega\in \Omega:\  \eta_0\exp\left(\bk	\int_0^T\abs{ \nabla c_m(s)}_{L^4}^{4}ds\right)>\bk_\eps\right\}\notag\\
		&\leq \eps.\notag
	\end{align}
	Thanks to inequality \re{3.77} we derive that
	\begin{align}
		\sup_m\mathbb{P}\left\{\omega\in \Omega:\  \abs{n_m}^2_{\mathcal{C}^{1/2}([0,T];H^{-3})}>\frac{\bk}{\eps}\right\}&\leq \frac{\eps}{\bk}\be\abs{n_m}^2_{\mathcal{C}^{1/2}([0,T];H^{-3})}\leq \eps.\notag
	\end{align}
	Since these three last inequalities  hold, we can apply Lemma \ref{lemma 3.3} and  conclude that the law of $n_m$ form a family of probability measures which is tight on $\bz_n$.
	
	Secondly, we will prove  that  the laws of $\bu_m$ and $c_m$ are tight on $\bz_\bu\times \bz_c$. From inequalities \re{3.68} and \re{3.70},  we obtain the first two conditions of Corollaries \ref{cor3.1} and \ref{cor3.2} for $\bu_m$ and $c_m$ respectively. Hence, it is sufficient to prove that the sequences $(\bu_m)_m$ and $(c_m)$ satisfy the Aldous condition in the spaces $V^*$ and $L^2(\bo)$ respectively. Let $\theta > 0$ $(\tau_\ell)_{\ell\geq1}$ be a sequence of stopping times such that $0\leq \tau_\ell\leq  T$. From the second equation of system \re{3.40} we have
	\begin{align}
		c_m(\tau_\ell+\theta)-c_m(\tau_\ell) &=\xi \int_{\tau_\ell}^{\tau_\ell+\theta}A_1c_m(s)ds- \int_{\tau_\ell}^{\tau_\ell+\theta}\bp_m^2B_1(\bu_m(s), c_m(s))ds\notag\\
		&+\int_{\tau_\ell}^{\tau_\ell+\theta} \bp_m^2R_1(n_m(s),c_m(s))ds+\gamma\int_{\tau_\ell}^{\tau_\ell+\theta}\bp_m^2\phi(c_m(s))d\beta_s.\label{4.80}
	\end{align}
	By the Fubini theorem, the H\"older inequality and inequality \re{3.78}, we have the following estimates
	\begin{align}
		\be\abs{\xi \int_{\tau_\ell}^{\tau_\ell+\theta}A_1c_m(s)ds}_{L^2}^2&\leq \xi^2\theta^{1/2}\be\int_{\tau_\ell}^{\tau_\ell+\theta}\abs{A_1c_m(s)}^2_{L^2}ds\notag\\
		&\leq \xi^2\theta^{1/2}\be\int_{0}^{T}\abs{A_1c_m(s)}^2_{L^2}ds\leq\bk\theta^{1/2},\notag
	\end{align}
	\begin{align}
		\be\abs{\int_{\tau_\ell}^{\tau_\ell+\theta}\bp_m^2B_1(\bu_m(s), c_m(s))ds}_{L^2}^2ds&\leq \theta^{1/2}\be\int_{\tau_\ell}^{\tau_\ell+\theta}\abs{\bp_m^2B_1(\bu_m(s), c_m(s))}^2_{L^2}ds\notag\\
		&\leq \theta^{1/2}\be\int_{0}^{T}\abs{\bp_m^2B_1(\bu_m(s), c_m(s))}^2_{L^2}ds\leq\bk\theta^{1/2},\notag
	\end{align}
	and
	\begin{align}
		\be\abs{\int_{\tau_\ell}^{\tau_\ell+\theta}\bp_m^2R_1(n_m(s), c_m(s))ds}_{L^2}^2ds&\leq \theta^{1/2}\be\int_{\tau_\ell}^{\tau_\ell+\theta}\abs{\bp_m^2R_1(n_m(s), c_m(s))}^2_{L^2}ds\notag\\
		&\leq \theta^{1/2}\be\int_{0}^{T}\abs{\bp_m^2R_1(n_m(s), c_m(s))}^2_{L^2}ds\leq\bk\theta^{1/2}.\notag
	\end{align}
	By the It\^o isometry, we note that
	\begin{align}
		\be\abs{\gamma\int_{\tau_\ell}^{\tau_\ell+\theta}\bp_m^2\phi(c_m(s))d\beta_s}^2_{L^2}&\leq \gamma^2\be\int_{\tau_\ell}^{\tau_\ell+\theta}\abs{\phi(c_m(s))}^2_{\mathcal{L}^2(\mathbb{R}^2,L^2)}\notag\\
		&\leq \gamma^2 \sum_{k=1}^2\abs{\sigma_k}^2_{L^2}\be\int_{\tau_\ell}^{\tau_\ell+\theta}\abs{\nabla c_m(s)}_{L^2}^2ds\notag\\
		&\leq \bk \theta\be\sup_{0\leq s\leq T}\abs{\nabla c_m(s)}^2_{L^2}\leq \bk\theta.\notag
	\end{align}
	Combining these inequalities, we infer from equality \re{4.80} that the condition \re{equa3.10} is satisfies for $(c_m)_{m\geq 1}$ in $L^{2}(\bo)$. Hence by Lemma \ref{lem3.2} the sequence $(c_m)_{m\geq 1}$ satisfies the Aldous condition in the space $L^{2}(\bo)$.
	
	Now we will consider the sequence $(\bu_m)_{m\geq 1}$. We first observe that from the first equation of system \re{3.40} we infer that
	\begin{align}
		\bu_m(\tau_\ell+\theta)-\bu_m(\tau_\ell)&=-\eta\int_{\tau_\ell}^{\tau_\ell+\theta}A_0\bu_m(s)ds-\int_{\tau_\ell}^{\tau_\ell+\theta}\bp_m^1B_0(\bu_m(s),\bu_m(s))ds\notag\\
		&+\int_{\tau_\ell}^{\tau_\ell+\theta}\bp_m^1R_0(n_m(s),\varPhi)ds +\int_{\tau_\ell}^{\tau_\ell+\theta}\bp_m^1g(\bu_m(s),c_m(s)) dW_s.\label{4.85*}
	\end{align}
	Thanks to the H\"older inequality and \re{3.78}, we have the following estimates
	\begin{align}
		\be\abs{\eta \int_{\tau_\ell}^{\tau_\ell+\theta}A_0\bu_m(s)ds}_{V^*}^2&\leq \eta^2\theta^{1/2}\be\int_{\tau_\ell}^{\tau_\ell+\theta}\abs{A_0\bu_m(s)}^2_{V^*}ds\notag\\
		&\leq \eta^2\theta^{1/2}\be\int_{0}^{T}\abs{A_0\bu_m(s)}^2_{V^*}ds\leq\bk\theta^{1/2},\notag
	\end{align}
and
	\begin{align}
		\be\abs{\int_{\tau_\ell}^{\tau_\ell+\theta}\bp_m^2B_0(\bu_m(s), \bu_m(s))ds}_{V^*}^2ds&\leq \theta^{1/2}\be\int_{\tau_\ell}^{\tau_\ell+\theta}\abs{\bp_m^2B_1(\bu_m(s), \bu_m(s))}^2_{V^*}ds\notag\\
		&\leq \theta^{1/2}\be\int_{0}^{T}\abs{\bp_m^2B_1(\bu_m(s), \bu_m(s))}^2_{V^*}ds\leq\bk\theta^{1/2},\notag
	\end{align}
	as well as
	\begin{align}
		\be\abs{\int_{\tau_\ell}^{\tau_\ell+\theta}\bp_m^2R_0(n_m(s),\varPhi)ds}_{V^*}^2ds&\leq \theta^{1/2}\be\int_{\tau_\ell}^{\tau_\ell+\theta}\abs{\bp_m^2R_0(n_m(s),\varPhi)}^2_{V^*}ds\notag\\
		&\leq \theta^{1/2}\be\int_{0}^{T}\abs{\bp_m^2R_0(n_m(s),\varPhi)}^2_{V^*}ds\leq\bk\theta^{1/2}.\notag
	\end{align}
	Thanks to the  It\^o isometry and the assumption on $g$  we obtain
	\begin{align}
		\be\abs{\int_{\tau_\ell}^{\tau_\ell+\theta}\bp_m^1g(\bu_m(s),c_m(s)) dW_s}^2_{V^*}&\leq\bk\be\abs{\int_{\tau_\ell}^{\tau_\ell+\theta}\bp_m^1g(\bu_m(s),c_m(s)) dW_s}^2_{L^2}\notag\\
		&\leq\bk\be\int_{\tau_\ell}^{\tau_\ell+\theta}\abs{\bp_m^1g(\bu_m(s),c_m(s))}^2_{\mathcal{L}^2(\buc,H)}ds\notag\\
		&\leq \bk\be\int_{\tau_\ell}^{\tau_\ell+\theta}(1+\abs{(\bu_m(s),c_m(s))}_{\bhc}^2)ds\notag\\
		&\leq \bk\left(1+ \be\sup_{0\leq s\leq T}\abs{(\bu_m(s),c_m(s))}_{\bhc}^2\right)\theta\leq \bk\theta.\notag
	\end{align}
	From  these inequalities and  equality \re{4.85*}, we can conclude by Lemma \ref{lem3.2} that  the sequence $(\bu_m)_{m\geq 1}$ satisfies the Aldous condition in the space $V^*$.
	Hence, by  applying  Corollary \ref{cor3.1} and  Corollary \ref{cor3.2}, we see that  the laws of $c_m$ and $\bu_m$ are tight on $\bz_c$ and $\bz_\bu$, respectively.
\end{proof}
Since $(\bl_m)_m$ is tight on  $\bz$, invoking \cite[Corollary 2, Appendix B]{mot} (see also \cite[Theorem 4.13]{Brz1}) there exists a probability space $$(\Omega',\mathcal{F}', \mathbb{P}'),$$ and a subsequence  of random
vectors $(\bar{\bu}_{m_k},\bar{c}_{m_k},\bar{n}_{m_k})$ with values in $\bz$ such that
\begin{description}
	\item[i)] $(\bar{\bu}_{m_k},\bar{c}_{m_k},\bar{n}_{m_k})$ have the same probability distributions as $(\bu_{m_k},c_{m_k},n_{m_k})$,
	\item[ii)] $(\bar{\bu}_{m_k},\bar{c}_{m_k},\bar{n}_{m_k})$ converges  in the topology of $\bz$ to a random element $(\bu,c,n)\in \bz$ with probability $1$ on $(\Omega',\mathcal{F}', \mathbb{P}')$ as $k\to\infty$.
\end{description}
To simplify the notation, we will simply denote these sequences  by $(\bu_{m},c_{m},n_{m})_{m\geq1}$ and $(\bar{\bu}_{m},\bar{c}_{m},\bar{n}_{m})_{m\geq1}$, respectively.

Next, from the definition of the space $\bz$, we deduce that $\mathbb{P}'$-a.s.,
\begin{equation}\label{3.80}
	\begin{split}
		&\bar{\bu}_m\to \bu \text{ in } L_{w}^{2}(0,T;V)\cap L^{2}(0,T;H)\cap\mathcal{C}([0,T];V^*)\cap\mathcal{C}([0,T];H_{w}), \\
		&\bar{c}_m\to c \text{ in } L_{w}^{2}(0,T;H^2(\bo))\cap L^{2}(0,T;H^1(\bo))\cap\mathcal{C}([0,T];L^2(\bo))\cap\mathcal{C}([0,T];H_w^1(\bo)),\\
		&\bar{n}_m\to n \text{ in } L_{w}^{2}(0,T;H^1(\bo))\cap L^{2}(0,T;L^2(\bo))\cap\mathcal{C}([0,T];H^{-3}(\bo))\cap\mathcal{C}([0,T];L^2_{w}(\bo)).
	\end{split}
\end{equation}
According to   \cite[Theorem 1.10.4 and Addendum 1.10.5]{Van},  a family of measurable map $\Psi_m:\Omega'\to \Omega$ can be constructed such that   together with  the new probability space $(\Omega', \mathcal{F}', \mathbb{P}')$  satisfy the property 
\begin{equation}\label{5.61}
	\begin{split}
		&	\bar{\bu}_m(\omega')=\bu_m\circ\Psi_m(\omega'),\ \bar{n}_m(\omega')=n_m\circ\Psi_m(\omega'),\\ &\bar{c}_m(\omega')=c_m\circ\Psi_m(\omega'),\text{ and }\mathbb{P}=\mathbb{P}'\circ\Psi^{-1}_m,
	\end{split}
\end{equation}
for all $\omega'\in \Omega'$.
Taking into account the fact that inequality \re{4.46} holds, we can derive that for almost every $(t,\omega')\in[0,T]\times \Omega'$,
\begin{equation}
	\abs{\bar{c}_m(t,\omega')}_{L^\infty}=\abs{c_m(t,\Psi_m(\omega'))}_{L^\infty}\leq \abs{c_0}_{L^\infty},\qquad \text{for all } m\geq 1.\label{4.93*}
\end{equation}
Since the laws of $(\bu_{m},c_{m},n_{m})$ and $(\bar{\bu}_{m},\bar{c}_{m},\bar{n}_{m})$ are equal in the space $\bz_\bu\times\bz_c\times \bz_n$, we have the estimates (\ref{3.68}), (\ref{3.70}) and
\begin{equation}
	\be'\int_0^T\abs{\bar{c}_m(s)}_{H^2}^2ds\leq \bk,\quad\be'\int_0^T\abs{\nabla \bar{\bu}_m(s)}_{L^2}^2ds\leq \bk,\label{4.92}\\
\end{equation}
as well as
\begin{equation}
	\qquad\be'\int_0^T\abs{\bar{n}_m(s)}^2_{L^2}ds\leq \bk.\label{4.94}
\end{equation}
From (\ref{4.92}) and (\ref{4.94}) and the  Banach-Alaoglu Theorem, we conclude that, there exists a subsequence of $(\bar{\bu}_m)_{m\geq 1}$, $(\bar{c}_m)_{m\geq 1}$, and $(\bar{n}_m)_{m\geq 1}$ weakly convergent  in
$L^2(\Omega', \mathcal{F}', \mathbb{P}';L^2(0,T;V))$, $L^2(\Omega', \mathcal{F}', \mathbb{P}';L^2(0,T;H^2(\bo)))$, and $L^2(\Omega', \mathcal{F}', \mathbb{P}';L^2(0,T;L^2(\bo)))$ respectively.
i.e. 
\begin{equation}\label{4.100}
	\begin{split}
		&\bu\in L^2(\Omega', \mathcal{F}', \mathbb{P}';L^2(0,T;V)),\quad c\in L^2(\Omega', \mathcal{F}', \mathbb{P}';L^2(0,T;H^2(\bo))),\\
		&\qquad\qquad \qquad n\in L^2(\Omega', \mathcal{F}', \mathbb{P}';L^2(0,T;L^2(\bo))).
	\end{split}
\end{equation}
On the other hand, from estimates (\ref{3.68}), (\ref{3.69}) and (\ref{3.70}) of Corollary \ref{cor4.11}, and the equalities given by \re{5.61}, we get for any $p\geq 1$,
\begin{align}
	&\be'\sup_{0\leq s\leq T}	\abs{\bar{\bu}_m(s)}^{2p}_{L^2}+\be'\left(\int_0^T\abs{\nabla \bar{\bu}_m(s)}_{L^2}^2ds\right)^p\leq \bk,\label{4.95}\\
	& 
	\qquad\be'\left(\int_0^T\abs{\bar{n}_m(s)}^2_{L^2}ds\right)^p\leq \bk,\label{4.96}\\
	&\be'\sup_{0\leq s\leq T}	\abs{\bar{c}_m(s)}^p_{H^1}+\be'\left(\int_0^T\abs{\bar{c}_m(s)}_{H^2}^2ds\right)^p\leq \bk\label{4.97}.
\end{align}

Then, invoking the Fatou lemma,  we infer that for $p\geq 2$, we have
\begin{equation}
	\be'\sup_{0\leq s\leq T}	\abs{\bu(s)}^p_{L^2}<\infty,\qquad\be'\sup_{0\leq s\leq T}	\abs{c(s)}^p_{H^1}<\infty.\label{4.104*}
\end{equation}
and 
\begin{align}
	&\be'\left(\int_0^T\abs{\nabla \bu(s)}_{L^2}^2ds\right)^p< \infty,\ \be'\left(\int_0^T\abs{n(s)}^2_{L^2}ds\right)^p<\infty,\ \be'\left(\int_0^T\abs{c(s)}_{H^2}^2ds\right)^p<\infty\label{4.108}.
\end{align}
	Now, we prove  three lemmata which show how  convergence in $\bz$ given by \re{3.80} will be used for the convergence of the deterministic terms appearing in the Galerkin approximation.  We start by noting that since $n_0^m$, $c_0^m$ and $\bu_0^m$ have been chosen such that \re{4.42} holds, we can derive that for all $\psi\in H^3(\bo)$ and $(\psi,\bv)\in L^2(\bo)\times V$,
	\begin{equation}
		\lim_{m\longrightarrow\infty}(n^m_0,\psi)= (n_0,\psi),\
		\lim_{m\longrightarrow\infty}(c^m_0,\psi)=(c_0,\psi),\text{  and  }
		\lim_{m\longrightarrow\infty}(\bu^m_0,\bv)= (\bu_0,\bv).\label{4.115*}
	\end{equation}
	\begin{lemma}\label{lem4.7}
		For any $r,t\in[0,T]$ with $r\leq t$ and  $\psi\in H^3(\bo)$, the following  convergences hold $\mathbb{P}'$-a.s.
		\begin{equation}
			\begin{split}
				&\lim_{m\longrightarrow\infty}(\bar{n}_m(t),\psi)=(n(t),\psi),\\
				&\lim_{m\longrightarrow\infty}\int_r^t(A_1\bar{n}_m(s),\psi)ds= \int_r^t(A_1n(s),\psi)ds\label{4.106}\\
				&\lim_{m\longrightarrow\infty}\int_r^t(\bp_m^2B_1(\bar{\bu}_m(s),\bar{n}_m(s)),\psi)ds=\int_r^t(B_1(\bu(s),n(s)),\psi)ds,\\
				&\lim_{m\longrightarrow\infty}\int_r^t(\bp_m^2R_2(\bar{n}_m(s),\bar{c}_m(s)),\psi)ds=\int_r^t(R_2(n(s),c(s)),\psi)ds.
			\end{split}
		\end{equation}
	\end{lemma}
	\begin{proof}
	Let $\psi\in H^3(\bo)$ and $t\in [0,T]$ be arbitrary but fixed. By  the H\"older inequality we have
		\begin{align}
			\abs{(\bar{n}_m(t),\psi)-(n(t),\psi)}&\leq \abs{\bar{n}_m(t)-n(t)}_{H^{-3}}\abs{\psi}_{H^3}\notag\\
			&\leq \abs{\bar{n}_m-n}_{\mathcal{C}([0,T]; H^{-3})}\abs{\psi}_{H^3}\label{4.109},
		\end{align}
		which along with \re{3.80} implies the first convergence in  \re{4.106}.
		
		Now, we also fix $r\in [0,T]$ such that $r\leq t$.	By an integration-by-parts and  the H\"older inequality we note that
		\begin{align}
			\abs{\int_r^t(A_1\bar{n}_m(s),\psi)ds-\int_r^t(A_1n(s),\psi)ds}dt&\leq \int_0^T\abs{(A_1\bar{n}_m(s)-A_1n(s),\psi)}ds\notag\\
			&\leq \int_0^T\abs{(\bar{n}_m(s)-n(s),A_1\psi)}ds\label{4.110}\\
			&\leq T\sup_{0\leq s\leq T}\abs{(\bar{n}_m(s)-n(s),A_1\psi)}. \notag
		\end{align}
		From the convergence \re{3.80} we infer that  $\bar{n}_m\to n$ in $\mathcal{C}([0,T]; L_w^{2}(\bo)$, $\mathbb{P}'$-a.s. This means that $\sup_{0\leq s\leq T}\abs{(\bar{n}_m(s)-n(s),\varphi)}$ tends to zero for all $\varphi\in L^2(\bo)$ as $m$ goes to infinity with probability one. We plug $\varphi=A_1\psi$ and pass to the limit in \re{4.110} and derive the second convergence of \re{4.106}.
		We have for all $\omega\in \Omega$,
		\begin{align}
			&\abs{\int_r^t(\bp_m^2B_1(\bar{\bu}_m(s),\bar{n}_m(s)),\psi)ds-\int_r^t(B_1(\bu(s),n(s)),\psi)ds}\notag\\
			&\leq\int_0^T\abs{(B_1(\bar{\bu}_m(s),\bar{n}_m(s)),\bp_m^2\psi-\psi)}+\int_0^T\abs{(B_1(\bar{\bu}_m(s),\bar{n}_m(s))-B_1(\bu(s),n(s)),\psi)}ds\notag
		\end{align}
		Since $\bar{\bu}_m\to \bu$ in $L^{2}(0,T;H)$, and $\bar{n}_m\to n$ in $L^{2}(0,T;L^2(\bo))$ $\mathbb{P}'$-a.s.,   by integration-by-parts, we derive that
		\begin{align}
			&\int_0^T\abs{(B_1(\bar{\bu}_m(s),\bar{n}_m(s)),\bp_m^2\psi-\psi)}ds\notag\\
			&\leq \int_0^T\abs{(\bar{n}_m(s)\bar{\bu}_m(s),\nabla(\bp_m^2\psi-\psi))}\notag\\
			&\leq \abs{\nabla(\bp_m^2\psi-\psi)}_{L^\infty}\int_0^T\abs{\bar{n}_m(s)}_{L^2}\abs{\bar{\bu}_m(s)}_{L^2}ds\notag\\
			&\leq \abs{\bp_m^2\psi-\psi}_{H^3}\left(\int_0^T\abs{\bar{\bu}_m(s)}^2_{L^2}ds\right)^{1/2}\left(\int_0^T\abs{\bar{n}_m(s)}^2_{L^2}ds\right)^{1/2}\notag\\
			&\leq\bk\abs{\bp_m^2\psi-\psi}_{H^3}.\notag
		\end{align}
By using an integration-by-parts and the fact that  $\nabla\cdot \bu=0$ we get
		\begin{align}
			& \int_0^T\abs{(B_1(\bar{\bu}_m(s),\bar{n}_m(s))-B_1(\bu(s),n(s)),\psi),\psi)}ds\notag\\
			&\leq \int_0^T\abs{((\bar{\bu}_m(s)-\bu(s))\nabla\bar{n}_m(s),\psi)}ds+\int_0^T\abs{(\bu(s)\nabla(\bar{n}_m(s)-n(s)),\psi)}ds\notag  \\
			&\leq \int_0^T\abs{(\bar{n}_m(s),(\bar{\bu}_m(s)-\bu(s))\cdot\nabla\psi)}ds+\int_0^T\abs{((\bar{n}_m(s)-n(s)),\bu(s)\cdot\nabla\psi)}ds\notag \\
			&\leq \abs{\psi}_{L^\infty}\int_0^T\abs{\bar{\bu}_m(s)-\bu(s)}_{L^2}\abs{\bar{n}_m(s)}_{L^2}ds+ \abs{\nabla\psi}_{L^\infty}\int_0^T\abs{\bar{n}_m(s)-n(s)}_{L^2}\abs{\bu(s)}_{L^2}ds.\notag
		\end{align}
		Using the fact that $\abs{\nabla\psi}_{L^\infty}\leq \abs{\psi}_{H^3}$, we infer from the two last inequalities that
		\begin{align}
			&\abs{\int_0^t(\bp_m^2B_1(\bar{\bu}_m(s),\bar{n}_m(s)),\psi)ds-\int_0^t(B_1(\bu(s),n(s)),\psi)ds}\notag\\
			&\leq T\abs{\psi}_{H^3}\left(\int_0^T\abs{\bar{\bu}_m(s)-\bu(s)}^2_{L^2}ds\right)^{1/2}\left(\int_0^T\abs{\bar{n}_m(s)}^2_{L^2}ds\right)^{1/2}\notag\\
			&\qquad{}+ T\abs{\psi}_{H^3}\left(\int_0^T\abs{\bar{n}_m(s)-n(s)}^2_{L^2}ds\right)^{1/2}\left(\int_0^T\abs{\bu(s)}^2_{L^2}ds\right)^{1/2}+\bk\abs{\bp_m^2\psi-\psi}_{H^3}\\
			&\leq \bk \left(\int_0^T\abs{\bar{\bu}_m(s)-\bu(s)}^2_{L^2}ds\right)^{1/2}\notag\\
			&\qquad{}+\bk\left(\int_0^T\abs{\bar{n}_m(s)-n(s)}^2_{L^2}ds\right)^{1/2}\left(\int_0^T\abs{\bu(s)}^2_{L^2}ds\right)^{1/2}+\bk\abs{\bp_m^2\psi-\psi}_{H^3},\notag
		\end{align}
which upon letting $n\to \infty$, implies the third convergence in \re{4.106}.	

		Similarly, we have
		\begin{align}
			&\abs{\int_r^t(\bp_m^2R_2(\bar{n}_m(s),\bar{c}_m(s)),\psi)ds-\int_r^t(R_2(n(s),c(s)),\psi)ds}\notag\\
			&\leq\int_0^T\abs{(R_2(\bar{n}_m(s),\bar{c}_m(s))-R_2(n(s),c(s)),\psi)}ds\\
			&\qquad{}+\int_0^T\abs{(R_2(\bar{n}_m(s),\bar{c}_m(s)),\bp_m^2\psi-\psi)}ds.\notag
		\end{align}
Since $(\bar{c}_m,\bar{n}_m)\to (c,n)$ in $\bz_c\times\bz_n$, we see that $\mathbb{P}'$-a.s, 
		\begin{align}
			&\int_0^T\abs{(R_2(\bar{n}_m(s),\bar{c}_m(s)),\bp_m^2\psi-\psi)}ds\notag\\
			&\leq \abs{\nabla(\bp_m^2\psi-\psi)}_{L^\infty}\int_0^T\abs{\bar{n}_m(s)}_{L^2}\abs{\nabla\bar{c}_m(s)}_{L^2}ds\notag\\
			&\leq\bk\abs{\bp_m^2\psi-\psi}_{H^3}.\notag
		\end{align}
		On the other hand, we obtain
		\begin{align}
			&\int_0^T\abs{(R_2(\bar{n}_m(s),\bar{c}_m(s))-R_2(n(s),c(s)),\psi),\psi)}ds\notag\\
			&\leq \int_0^T\abs{((\bar{n}_m(s)-n(s))\nabla\bar{c}_m(s),\nabla\psi)}ds+\int_0^T\abs{(n(s)\nabla(\bar{c}_m(s)-c(s)),\nabla\psi)}ds\notag  \\
			&\leq \abs{\nabla\psi}_{L^\infty}\int_0^T\abs{\bar{n}_m(s)-n(s)}_{L^2}\abs{\nabla\bar{c}_m(s)}_{L^2}ds\notag\\
			&\qquad{}+ \abs{\nabla\psi}_{L^\infty}\int_0^T\abs{\nabla(\bar{c}_m(s)-c(s))}_{L^2}\abs{n(s)}_{L^2}ds\notag\\
			&\leq \bk \left(\int_0^T\abs{\bar{n}_m(s)-n(s)}^2_{L^2}ds\right)^{1/2}\notag\\
			&\qquad{}+\bk\left(\int_0^T\abs{\nabla(\bar{c}_m(s)-c(s))}^2_{L^2}ds\right)^{1/2}\left(\int_0^T\abs{n(s)}^2_{L^2}ds\right)^{1/2},\notag
		\end{align}
which along with \re{3.80} implies the fourth convergence in \re{4.106}.
	\end{proof}
	
	\begin{lemma}\label{lem4.8}
		For any $r,t\in[0,T]$ with $r\leq t$ and  $\psi\in H^2(\bo)$, the following  convergences hold $\mathbb{P}'$-a.s.
		\begin{align}
			\begin{split}
				&\lim_{m\longrightarrow\infty}(\bar{c}_m(t),\psi)= (c(t),\psi),\label{4.114}\\
				&\lim_{m\longrightarrow\infty}\int_r^t(A_1\bar{c}_m(s),\psi)ds= \int_r^t(A_1c(s),\psi)ds,\\
				&\lim_{m\longrightarrow\infty}\int_r^t(\bp_m^2B_1(\bar{\bu}_m(s),\bar{c}_m(s)),\psi)ds= \int_r^t(B_1(\bu(s),c(s)),\psi)ds,\\
				&\lim_{m\longrightarrow\infty}\int_r^t(\bp_m^2R_1(\bar{n}_m(s),\bar{c}_m(s)),\psi)ds= \int_r^t(R_1(n(s),c(s)),\psi)ds.
			\end{split}
		\end{align}
	\end{lemma}
	\begin{proof}
		Since $\bar{c}_m\to c$ in $\mathcal{C}([0,T]; L^{2}(\bo))$, $\mathbb{P}'$-a.s.,	the first convergence is done exactly using a similarly inequality as \re{4.109}.
		By an integration by part and  the  H\"older inequality we note that
		\begin{align}
			\abs{\int_r^t(A_1\bar{c}_m(s),\psi)ds-\int_r^t(A_1c(s),\psi)ds}&\leq \int_0^T\abs{(A_1\bar{c}_m(s)-A_1c(s),\psi)}ds\notag\\
			&\leq \int_0^T\abs{(\nabla(\bar{c}_m(s)-c(s)),\nabla\psi)}ds\notag\\
			&\leq T^{1/2}\abs{\psi}_{H^1}\left(\int_0^T\abs{\bar{c}_m(s)-c(s)}^2_{H^1}ds\right)^{1/2}, \notag
		\end{align}
which altogether with \re{3.80} implies  the second convergence in \re{4.114}.

Next, using  the Sobolev embedding $H^1(\bo)\hookrightarrow L^4(\bo)$, we get
		\begin{align}
			&\abs{\int_r^t(\bp_m^2B_1(\bar{\bu}_m(s),\bar{c}_m(s)),\psi)ds-\int_r^t(B_1(\bu(s),c(s)),\psi)ds}\notag\\
			&\leq \int_0^T\abs{(B_1(\bar{\bu}_m(s),\bar{c}_m(s))-B_1(\bu(s),c(s)),\psi),\psi)}ds+\int_0^T\abs{(B_1(\bar{\bu}_m(s),\bar{c}_m(s)),\bp_m^2\psi-\psi)}ds\notag\\
			&\leq \int_0^T\abs{((\bar{\bu}_m(s)-\bu(s))\nabla\bar{c}_m(s),\psi)}ds+\int_0^T\abs{(\bu(s)\nabla(\bar{c}_m(s)-c(s)),\psi)}ds\notag  \\
			&\qquad{}+T^{1/2}\abs{\bp_m^2\psi-\psi}_{L^2}\left(\int_0^T\abs{B_1(\bar{\bu}_m(s),\bar{c}_m(s))}^2_{L^2}ds\right)^{1/2}\notag\\
			&\leq \abs{\psi}_{L^{4}}\int_0^T\abs{\bar{\bu}_m(s)-\bu(s)}_{L^2}\abs{\nabla\bar{c}_m(s)}_{L^4}ds+ \abs{\psi}_{L^4}\int_0^T\abs{\nabla(\bar{c}_m(s)-c(s))}_{L^2}\abs{\bu(s)}_{L^4}ds\notag\\
			&\qquad{}+\bk\abs{\bp_m^2\psi-\psi}_{L^2}\notag\\
			&\leq T\abs{\psi}_{H^{1}}\int_0^T\abs{\bar{\bu}_m(s)-\bu(s)}_{L^2}\abs{\bar{c}_m(s)}_{H^2}ds\notag\\
			&\qquad{}+ T\abs{\psi}_{H^1}\int_0^T\abs{\nabla(\bar{c}_m(s)-c(s))}_{L^2}\abs{\nabla\bu(s)}_{L^2}ds+\bk\abs{\bp_m^2\psi-\psi}_{L^2}.\notag
		\end{align}
		Since the convergence \re{3.80} holds, we arrive at
		\begin{align}
			&\abs{\int_r^t(\bp_m^2B_1(\bar{\bu}_m(s),\bar{c}_m(s)),\psi)ds-\int_r^t(B_1(\bu(s),c(s)),\psi)ds}\notag\\
			&\leq T\abs{\psi}_{H^1}\left(\int_0^T\abs{\bar{\bu}_m(s)-\bu(s)}^2_{L^2}ds\right)^{1/2}\left(\int_0^T\abs{\bar{c}_m(s)}^2_{H^2}ds\right)^{\frac{1}{2}}\notag\\
			&\quad{}+ T\abs{\psi}_{H^1}\left(\int_0^T\abs{\bar{c}_m(s)-c(s)}^2_{H^1}ds\right)^{\frac{1}{2}}\left(\int_0^T\abs{\nabla\bu(s)}^2_{L^2}ds\right)^{\frac{1}{2}}+\bk\abs{\bp_m^2\psi-\psi}_{L^2},\notag
		\end{align}
	which along with \re{3.80} implies  the third convergence in \re{4.114}.
	
	Now we prove the last convergence. To this purpose, we note that
		\begin{align}
			&\abs{\int_r^t(\bp_m^2R_1(\bar{n}_m(s),\bar{c}_m(s)),\psi)ds-\int_r^t(R_1(n(s),c(s)),\psi)ds}\notag\\
			&\leq \int_0^T\abs{(R_1(\bar{n}_m(s),\bar{c}_m(s))-R_1(n(s),c(s)),\psi)}ds\notag\\
			&\qquad{}+\int_0^T\abs{(R_1(\bar{n}_m(s),\bar{c}_m(s)),\bp_m^2\psi-\psi)}ds\label{4.116}\\
			&\leq \int_0^T\abs{((\bar{n}_m(s)-n(s))f(\bar{c}_m(s)),\psi)}ds\notag\\
			&\qquad{}+\int_0^T\abs{n(s)(f(\bar{c}_m(s))-f(c(s))),\psi)}ds+\bk\abs{\bp_m^2\psi-\psi}_{L^2}\notag.
		\end{align}
		Using \re{4.93*}, we derive that
		\begin{align}
			&\int_0^T\abs{((\bar{n}_m(s)-n(s))f(\bar{c}_m(s)),\psi)}ds\notag\\
			&\leq\abs{\psi}_{L^\infty}\int_0^T\int_\bo\abs{\bar{n}_m(s)-n(s)}\abs{f(\bar{c}_m(s))}dxds\notag\\
			&\leq T^{1/2}\abs{\bo}^{1/2}\abs{\psi}_{H^2}\sup_{0\leq s\leq \abs{c_0}_{L^\infty}}f(s)\left(\int_0^T\abs{\bar{n}_m(s)-n(s)}^2_{L^2}ds\right)^{1/2}.\notag
		\end{align}
In a similar way, we see that
		\begin{equation}
			\begin{split}
				&\int_0^T\abs{n(s,x)(f(\bar{c}_m(s,x))-f(c(s,x))),\psi)}dsdx\\
				&\leq \abs{\psi}_{H^2}\int_0^T\int_\bo\abs{n(s,x)f(\bar{c}_m(s,x))-n(s,x)f(c(s,x))}dxds.\label{4.118}
			\end{split}
		\end{equation}
		Since the strong convergence $\bar{c}_m\to c$ in $L^2(0,T; H^{1}(\bo))$, $\mathbb{P}'$-a.s., holds, we derive that  up to a subsequence
		$$\bar{c}_m\to c \qquad dt\otimes dx\text{-a.e}$$
		Owing to the fact that $f$  is continuous, we infer that $\mathbb{P}'$-a.s., 
		$$nf(\bar{c}_m)\to nf(c) \qquad a.e \text{ in } \times (0,T)\times \bo.$$
		We also note that $\mathbb{P}$-a.s.,  $\{nf(\bar{c}_m)\}_{m\geq 1}$ is  uniformly integrable over $(0,T)\times \bo$. Indeed, we have
		\begin{align*}
			\int_{ (0,T)\times \bo}\abs{n(s,x)f(\bar{c}_m(s,x))}^2dxdsd\mathbb{P}&\leq \sup_{0\leq s\leq \abs{c_0}_{L^\infty}}f^2(s)\int_0^T\int_\bo\abs{n(s,x)}^2dxds\\
			&\leq \bk\int_0^T\abs{n(s)}^2_{L^2}ds.
		\end{align*}
		Therefore, by the Vitali Convergence Theorem, we derive that $\mathbb{P}'$-a.s.,  the right answer of the inequality \re{4.118} tends to zero as $m$ tends to $\infty$. Owing to this result, we can pass to the limit in the inequality \re{4.116} and obtain the last convergence of \re{4.114}.
	\end{proof}
	Next we prove the following convergences.
	\begin{lemma}\label{lem4.9}
		For any $r,t\in[0,T]$ with $r\leq t$ and  $\bv\in V$, the following  convergences hold $\mathbb{P}'$-a.s.
		\begin{equation}
			\begin{split}
				&\lim_{m\longrightarrow\infty}(\bar{\bu}_m(t),\bv)= (\bu(t),\bv),\\
				&\lim_{m\longrightarrow\infty}\int_r^t(A_0\bar{\bu}_m(s),\bv)ds= \int_r^t(A_0\bu(s),\bv)ds,\label{4.119}\\
				&\lim_{m\longrightarrow\infty}\int_r^t(\bp_m^1B_0(\bar{\bu}_m(s),\bar{\bu}_m(s)),\bv)ds= \int_r^t(B_0(\bu(s),\bu(s)),\bv)ds,\\
				&\lim_{m\longrightarrow\infty}\int_r^t(\bp_m^1R_0(\bar{n}_m(s),\varPhi),\bv)ds= \int_r^t(R_0(n(s),\varPhi),\bv)ds.
			\end{split}
		\end{equation}
	\end{lemma}
	\begin{proof}
		The proof is similar to the proof of Lemma \ref{lem4.7} and Lemma \ref{lem4.8}.
	\end{proof}
	In what follows, we will combine the convergence result from Lemma \ref{lem4.7}, Lemma \ref{lem4.8} and Lemma \ref{lem4.9} as well as martingale representation theorem  to construct a probabilistic weak solution to the problem \re{1.1}. In order to simplify the notation, we define on the probability space $(\Omega',\mathcal{F}',\mathbb{P}')$ the processes $N_m^1$, $N_m^2$, and $N_m^3$, respectively, by for $t\in[0,T]$,
	\begin{equation*}
		N_m^1(t):=-\bar{\bu}_m(t) -\Int_0^t[\eta A_0\bar{\bu}_m(s)+\bp_m^1B_0(\bar{\bu}_m(s),\bar{\bu}_m(s))]ds
		+\bu^m_0+\Int_0^t\bp_m^1R_0(\bar{n}_m(s),\varPhi)ds ,
	\end{equation*}
	\begin{equation*}
		N_m^2(t):=	-\bar{c}_m(t) - \Int_0^t[\xi A_1\bar{c}_m(s)+\bp_m^2B_1(\bar{\bu}_m(s), \bar{c}_m(s))]ds+c^m_0-\Int_0^t \bp_m^2R_1(\bar{n}_m(s),\bar{c}_m(s))ds,
	\end{equation*}
	and 
	\begin{equation*}
		N_m^3(t):=	-	\bar{n}_m(t) -\Int_0^t[\delta A_1\bar{n}_m(s)+\bp_m^2B_1(\bar{\bu}_m(s),\bar{n}_m(s))]ds +n^m_0-\Int_0^t \bp_m^2R_2(\bar{n}_m(s),\bar{c}_m(s))ds.
	\end{equation*}
	\begin{lemma}\label{lem5.13}
For all $m\in \mathbb{N}$ and for any $t\in[0,T]$,  we have
		\begin{equation}\label{5.88}
			N_m^3(t)=0, \qquad\mathbb{P}'\text{-a.s.}
		\end{equation}
	\end{lemma}
	\begin{proof}
Let  $m\in \mathbb{N}$ and $t\in[0,T]$ be arbitrary but fixed. On the probability space $(\Omega,\mathcal{F},\mathbb{P})$, we define  the processes $M_m^3(t)$ by 
		\begin{equation*}
			M_m^3(t):=-n_m(t) -\int_0^t[\delta A_1n_m(s)+\bp_m^2B_1(\bu_m(s),n_m(s))]ds +n_0^m-\int_0^t \bp_m^2R_2(n_m(s),c_m(s))ds.
		\end{equation*}
		We also define the following subsets of $\Omega$ and $\Omega'$
		$$ \mathcal{A}_m^N(t):=\left\{\omega'\in \Omega':  N_m^3(t)=0\right\}\text{ and }\mathcal{A}_m^M(t):=\left\{\omega\in \Omega:  M_m^3(t)=0\right\}.$$
		We note that, since the last equation of \re{3.40} holds,  $\mathbb{P}(\mathcal{A}_m^M(t))=1$. Furthermore,  by \re{5.61}, we derive that for all $\omega'\in \Omega$,  $N_m^3(t, \omega')=M_m^3(t,\Psi_m(\omega'))$ and therefore we observe that $ \mathcal{A}_m^N(t)=\Psi_m^{-1}( \mathcal{A}_m^M(t))$. Invoking  \re{5.61} once more,  we deduce that $$ \mathbb{P}'(\mathcal{A}_m^N(t))= \mathbb{P}'(\Psi^{-1}_m(\mathcal{A}_m^M(t)))=\mathbb{P}(\mathcal{A}_m^M(t))=1,$$
which completes the proof of Lemma \ref{lem5.13}.
	\end{proof}
	Using the convergences \re{4.115*} and \re{4.106} as well as 
	Lemma \ref{lem5.13} we see that  for all $t\in [0,T]$, $\mathbb{P}'$-a.s. 
	\begin{equation}\label{5.89}
		n(t) +\Int_0^t[\delta A_1n(s)+B_1(\bu(s),n(s))]ds =n_0-\Int_0^t R_2(n(s),c(s))ds, \qquad \text{in } H^{-3}(\bo).
	\end{equation}
	Now, on the probability space $(\Omega',\mathcal{F}',\mathbb{P}')$ we define a  the $\bh_m\times H_m$-valued processes $N_m$ by  $N_m(t)=(N^1_m(t),N^2_m(t) )$ for all $m\geq 1$ and $t\in [0,T]$. Since 
	\begin{equation}
		\bh_m\times H_m\subset H\times L^2(\bo)\hookrightarrow V^*\times H^{-2}(\bo),\label{5.90}
	\end{equation}
	the process $N_m$ can be seen as a $V^*\times H^{-2}(\bo)$-valued process.  
	
	Next,  we  collect the necessary ingredients for the application of the martingale representation theorem from \cite[Theorem 8.2]{Da}. To this aim, we consider the following Gelfand triple $V
	\hookrightarrow H\hookrightarrow V^*$ and $ H^2(\bo)
	\hookrightarrow  L^2(\bo)\hookrightarrow H^{-2}(\bo)$.  Let  $i^1:V
	\hookrightarrow H$ be the usual embedding and $i^{1*}$ its Hilbert-space-adjoint such that $(ix,y)=(x,i^{1*}y)_V$ for all $x\in V$ and $y\in H$.  In a very similar way,  we denote the usual embedding $H^2(\bo)
	\hookrightarrow  L^2(\bo)$ by $i^2$ and by $i^{*2}$ its Hilbert-space-adjoint. 
	We define the   embedding $i:V\times H^2(\bo)
	\hookrightarrow H\times L^2(\bo)$  and its adjoint  $i^*: H\times L^2(\bo)\longrightarrow V\times H^2(\bo)$ respectively by 
	\begin{equation*}
		i=\begin{pmatrix}
			i^1&0\\
			0&i^2
		\end{pmatrix}, \qquad  	i^*=\begin{pmatrix}
			i^{1*}&0\\
			0&i^{2*}
		\end{pmatrix}. 
	\end{equation*}
	
	Further, we set $L^1=(i^{1*})':V^*\longrightarrow H$  as the dual operator of $i^{1*}$ such that  for all $x\in H$ and $y\in V^*$, $(Ly,x)=\langle y,x\rangle$.  Similarly, the dual operator of  $i^{2*}$  will be denoted by $L^2:H^{-2}(\bo)\longrightarrow L^2(\bo)$.  We then define the following dual operator $L:=(i^*)':V^*\times H^{-2}(\bo)\longrightarrow H\times L^2(\bo)$ by 
	\begin{equation*}
		L=\begin{pmatrix}
			L^{1}&0\\
			0&L^{2}
		\end{pmatrix}. 
	\end{equation*}
	On the space  $\bh_m\times H_m$, we define  a mapping $G_m$ by 
	\begin{equation*}
		G_m	(\bv,\psi)=\begin{pmatrix}
			L^1\bp_m^1g(\bv,\psi)&0\\
			0&L^2\bp_m^2\phi(\psi)
		\end{pmatrix}, \qquad  (\bv,\psi)\in \bh_m\times H_m.
	\end{equation*}
	Here $(\bp_m^1g(\bv,\psi), \bp_m^2\phi(\psi))=(\bp_m^1g(\bv,\psi), \bp_m^2\phi(\psi))$ is seen as an element of $V^*\times H^{-2}(\bo)$ owing to the inclusion \re{5.90}.
	
	In the following lemma, we prove the martingale property of the process $LN_m$.
	\begin{lemma}
		For each $m\geq 1$, the process $LN_m$ is an $H\times L^2(\bo)$-valued continuous square integrable martingale with respect to the filtration $$\mathbb{F}^{'m}=\left\{\sigma\left(\sigma\left(  (\bar{\bu}_{m}(s),\bar{c}_{m}(s),\bar{n}_{m}(s)); s\leq t\right)\cup\mathcal{N}'\right)\right\}_{t\in [0,T]},$$ where $\mathcal{N}'$ is the set of null sets of $\mathcal{F}'$. The quadratic variation of $LN_m$ is given by 
		\begin{equation}\label{5.90*}
			\langle\langle LN_m\rangle\rangle_t=\int_0^tG_m	(\bar{\bu}_m(s),\bar{c}_m(s))G_m	(\bar{\bu}_m(s),\bar{c}_m(s))^*ds,
		\end{equation}
		where $G_m(\bar{\bu}_m,\bar{c}_m)^*:H\times L^2(\bo)\to \buc\times\mathbb{R}^2$ is the adjoint of the operator $G_m(\bar{\bu}_m,\bar{c}_m)$ and is given by
		$$G_m(\bar{\bu}_m,\bar{c}_m)^*\bv=\left( \sum_{k=1}^\infty(\bp_m^1g(\bar{\bu}_m,\bar{c}_m)e_k,i^{1*}\bw)e_k, \sum_{k=1}^2(\bp_m^2\phi(\bar{c}_m)g_k,i^{2*}\psi)g_k \right),$$
		for all $\bv=(\bw,\psi)\in H\times L^2(\bo)$.
	\end{lemma}
	\begin{proof}
		For any $m\geq 1$ we define the  $V^*\times H^{-2}(\bo)$-valued processes $M_m$ by  $$M_m(t)=(M^1_m(t),M^2_m(t) ), \quad t\in [0,T],$$ where
		\begin{equation*}
			M_m^1(t):=-\bu_m(t) -\Int_0^t[\eta A_0\bu_m(s)+\bp_m^1B_0(\bu_m(s),\bu_m(s))]ds
			+\bu^m_0+\Int_0^t\bp_m^1R_0(n_m(s),\varPhi)ds ,
		\end{equation*}
		\begin{equation*}
			M_m^2(t):=	-c_m(t) - \Int_0^t[\xi A_1c_m(s)+\bp_m^2B_1(\bu_m(s), c_m(s))]ds+c^m_0-\Int_0^t \bp_m^2R_1(n_m(s),c_m(s))ds.
		\end{equation*}
Let us set $\mathbf{W}_s:=(W_s,\beta_s)$.  Then, since $(\bu_m, c_m,n_m )$ is a solution of the finite dimensional system \re{3.40}, we deduce that $LM_m$ can be represented as
		\begin{equation*}
			LM_m(t)=\int_0^t	G_m	(\bu_m(s),c_m(s))d\mathbf{W}_s,\quad \mathbb{P}\text{-a.s.} \quad\text{for all } t\in[0.T].
		\end{equation*}
Using the continuity property of the operators $L^1$ and $L^2$  as well as Corollary \ref{cor4.11}, the estimate 
		\begin{align}
			&\be\int_{0}^{T}\abs{G_m(\bu_m(s),c_m(s))}^2_{\mathcal{L}^2(\buc\times\mathbb{R}^2,H\times L^2)}ds\notag\\
			&\leq\bk\be\int_{0}^{T}\abs{\bp_m^1g(\bu_m(s),c_m(s))}^2_{\mathcal{L}^2(\buc,H)}ds+\bk\be\int_{0}^{T}\abs{\bp_m^2\phi(c_m(s))}^2_{\mathcal{L}^2(\mathbb{R}^2,L^2)}ds\notag\\
			&\leq \bk\be\int_{0}^{T}(1+\abs{(\bu_m(s),c_m(s))}_{\bhc}^2)ds+\gamma^2 \sum_{k=1}^2\abs{\sigma_k}^2_{L^2}\be\int_{0}^{T}\abs{\nabla c_m(s)}_{L^2}^2ds\notag\\
			&\leq\bk\left(1+ \be\sup_{0\leq s\leq T}\abs{(\bu_m(s),c_m(s))}_{\bhc}^2\right)+ \bk \be\sup_{0\leq s\leq T}\abs{\nabla c_m(s)}^2_{L^2}<\infty,\notag
		\end{align}
		yields that $M_m$ is a square integrable continuous martingale over the probability space $(\Omega,\mathcal{F},(\mathcal{F}_t)_{t\in[0,T]},\mathbb{P})$.  Moreover, from the definition of $M_m$ we derive that for each $t\in[0.T]$,  $M_m(t)$ is measurable with respect to the $\sigma$-field $$\mathbb{F}^{m}=\left\{\sigma\left(\sigma\left(  (\bu_{m}(s),c_{m}(s),n_{m}(s)); s\leq t\right)\cup \mathcal{N}\right)\right\}_{t\in [0,T]},$$ where $\mathcal{N}$ is the set of null sets of $\mathcal{F}$. Hence, invoking \cite[Theorem 4.27]{Da} we infer that $M_m$ is a $\mathbb{F}^{m}$-martingale with quadratic variation 
		\begin{equation*}
			\langle\langle M_m\rangle\rangle_t=\int_0^tG_m	(\bu_m(s),c_m(s))G_m	(\bu_m(s),c_m(s))^*ds.
		\end{equation*}
		This means that  for all $s,t\in [0,T]$, $s\leq t$, all  $\bv_i=(\bw_i,\psi_i)\in H\times L^2(\bo)$, $i=1,2$, and  all bounded and continuous  real-valued functions $h=(h_1,h_2,h_3)$ on $\mathcal{C}([0,T]; H\times  L^2(\bo)\times  L^2(\bo))$, we have
		\begin{equation*}
			\be\left[\left(LM_m(t)-LM_m(s),\bv_1\right)_{H\times L^2(\bo)}h_1(\bu_m|_{[0,s]})h_2(c_m|_{[0,s]})h_3(n_m|_{[0,s]})\right]=0,
		\end{equation*}
		and 
		\begin{align*}
			&\be\left[\left(\left(LM_m(t),\bv_1\right)_{H\times L^2(\bo)}\left(LM_m(t),\bv_2\right)_{H\times L^2(\bo)}-\left(LM_m(s),\bv_1\right)_{H\times L^2(\bo)}\left(LM_m(s),\bv_2\right)_{H\times L^2(\bo)}\right.\right.\\
			&\qquad	\left.\left. -\int_0^t\left(G_m	(\bu_m(s),c_m(s))^*\bv_1,G_m	(\bu_m(s),c_m(s))^*\bv_2\right)_{\buc\times \mathbb{R}^2}ds \right)\times\right.\\
			&\qquad\left.\times h_1(\bu_m|_{[0,s]})h_2(c_m|_{[0,s]})h_3(n_m|_{[0,s]})\right]=0.
		\end{align*}
		Since $(\bu_m,c_m,n_m)$ and $(\bar{\bu}_m,\bar{c}_m,\bar{n}_m)$ have the same laws on $\mathcal{C}([0,T]; \bhc_m)$, we deduce from these two last equalities that 
		\begin{equation}
			\be'\left[\left(LN_m(t)-LN_m(s),\bv_1\right)_{H\times L^2(\bo)}h_1(\bar{\bu}_m|_{[0,s]})h_2(\bar{c}_m|_{[0,s]})h_3(\bar{n}_m|_{[0,s]})\right]=0,\label{5.91*}
		\end{equation}
		and 
		\begin{align}
			&\be'\left[\left(\left(LN_m(t),\bv_1\right)_{H\times L^2(\bo)}\left(LN_m(t),\bv_2\right)_{H\times L^2(\bo)}\right.\right.\notag\\
			&	\left.\left.\qquad-\left(LN_m(s),\bv_1\right)_{H\times L^2(\bo)}\left(LN_m(s),\bv_2\right)_{H\times L^2(\bo)}\right.\right.\notag\\
			&	\left.\left. \qquad-\int_0^t\left(G_m	(\bar{\bu}_m(s),\bar{c}_m(s))^*\bv_1,G_m	(\bar{\bu}_m(s),\bar{c}_m(s))^*\bv_2\right)_{\buc\times \mathbb{R}^2}ds \right)\times\right.\label{5.92*}\\
			&\qquad\left.\times h_1(\bar{\bu}_m|_{[0,s]})h_2(\bar{c}_m|_{[0,s]})h_3(\bar{n}_m|_{[0,s]})\right]=0,\notag
		\end{align}
		for all $s,t\in [0,T]$, $s\leq t$, all  $\bv_i=(\bw_i,\psi_i)\in H\times L^2(\bo)$, $i=1,2$, and  all (real-valued) function $h_i$, $i=1,2,3$ bounded and continuous on $\mathcal{C}([0,T]; \bhc_m)$, $\mathcal{C}([0,T];  H_m)$, and $ \mathcal{C}([0,T];  H_m)$ respectively,
		This implies that $LN_m$ is a continuous square integrable martingale with respect to $\mathbb{F}^{'m}$ and the quadratic variation is given as claimed by equality \re{5.90*}.
	\end{proof}
	On the new probability space $(\Omega',\mathcal{F}',\mathbb{P}')$, we consider the $V^*\times H^{-2}(\bo)$-valued continuous process $N$ defined by $N(t)=(N^1(t),N^2(t))$ for all $t\in [0,T]$, where
	\begin{equation*}
		N^1(t):=-\bu(t) -\int_0^t[\eta A_0\bu(s)+B_0(\bu(s),\bu(s))]ds
		+\bu_0+\int_0^tR_0(n(s),\varPhi)ds ,
	\end{equation*}
	\begin{equation*}
		N^2(t):=	-c(t) - \int_0^t[\xi A_1c(s)+B_1(\bu(s), c(s))]ds+c_0-\int_0^t R_1(n(s),c(s))ds.
	\end{equation*}
	In the next lemma, we state that $LN=(L^1N^1,L^2N^2)$ is also an $H\times L^2(\bo)$-valued martingale.
	\begin{lemma}\label{lem5.15}
		The process $LN$ is an $H\times L^2(\bo)$-valued continuous square integrable martingale with respect to the filtration $\mathbb{F}'=\left\{\sigma\left(  (\bu(s),c(s),n(s)); s\leq t\right)\right\}_{t\in[0,T]}$. The quadratic variation is given by 
		\begin{equation*}
			\langle\langle LN\rangle\rangle_t=\int_0^tG	(\bu(s),c(s))G	(\bu(s),c(s))^*ds, 
		\end{equation*} 
		where \begin{equation*}
			G	(\bu,c)=\begin{pmatrix}
				L^1g(\bu,c)&0\\
				0&L^2\phi(c)
			\end{pmatrix},
		\end{equation*}
		and  $G	(\bu,c)^*:H\times L^2(\bo)\to \buc\times\mathbb{R}^2$ is the adjoint of the operator $G	(\bu,c)$ given by
		$$G	(\bu,c)^*\bv=\left( \sum_{k=1}^\infty(L^1g(\bu(s),c(s))e_k,\bw)e_k, \sum_{k=1}^2(L^2\phi(c(s))g_k,\psi)g_k \right),$$
		for all $\bv=(\bw,\psi)\in H\times L^2(\bo)$.
	\end{lemma}
	\begin{proof}
		Let $t\in [0,T]$. We first prove that 	$LN$ is an $H\times L^2(\bo)$-valued  square integrable random variable.  Thanks to the continuity of $L$, it will be sufficient to prove that $\be \abs{N}_{V^*\times H^{-2}}^2<\infty$. Using Lemma \ref{lem4.8} and Lemma \ref{lem4.9}, we conclude that 
		$$\lim_{m\longrightarrow\infty} N_m(t)=N(t)\qquad \mathbb{P}'\text{-a.s. in } \ V^*\times H^{-2}(\bo).$$
		By the continuity of the injection $H\times L^2(\bo)\hookrightarrow V^*\times H^{-2}(\bo)$, the  Burkholder-Gundy-Davis inequality for continuous martingales and equality \re{5.90*} as well as inequalities \re{4.95} and \re{4.97}, we have
		\begin{align}
			\be'\sup_{0\leq s\leq T}\abs{N_m(s)}_{V^*\times H^{-2}}^4&\leq \bk\be'\sup_{0\leq s\leq T}\abs{N_m(s)}_{L^2\times L^{2}}^4\notag\\
			&\leq\bk\be'\left(\int_{0}^{T}\abs{G_m(\bar{\bu}_m(s),\bar{c}_m(s))}^2_{\mathcal{L}^2(\buc\times\mathbb{R}^2,H\times L^2)}ds\right)^2\notag\\
			&=2\bk\be'\left(\int_{0}^{T}\abs{\bp_m^1g(\bar{\bu}_m(s),\bar{c}_m(s))}^2_{\mathcal{L}^2(\buc,H)}ds\right)^2\notag\\
			&\qquad{}+2\bk\be'\left(\int_{0}^{T}\abs{\bp_m^2\phi(\bar{c}_m(s))}^2_{\mathcal{L}^2(\mathbb{R}^2,L^2)}ds\right)^2\label{5.91}\\
			&\leq\bk\left(1+ \be'\sup_{0\leq s\leq T}\abs{(\bar{\bu}_m(s),\bar{c}_m(s))}_{\bhc}^4\right)+ \bk \be'\sup_{0\leq s\leq T}\abs{\nabla \bar{c}_m(s)}^4_{L^2}<\bk.\notag
		\end{align}
		Hence, by the Vitali Theorem, we infer that $N(t)\in L^2(\Omega'; V^*\times H^{-2}(\bo))$ and 
		$$\lim_{m\longrightarrow\infty} N_m(t)=N(t)\ \text{ in } \ L^2(\Omega'; V^*\times H^{-2}(\bo)).$$
		Next, let $\bv=(\bw,\psi)\in V^*\times H^{-2}(\bo)$, and  $h_i$, $i=1,2,3$ be a bounded and continuous function on $\mathcal{C}([0,T]; V^*)$, $\mathcal{C}([0,T];  H^{-2}(\bo))$, and $ \mathcal{C}([0,T];  H^{-3}(\bo))$ respectively. Let $s,t\in [0,T]$ such that  $ s\leq t$. Let
		\begin{align*}
			&F_m(t,s):=\left(LN_m(t)-LN_m(s),\bv\right)_{H\times L^2(\bo)}h_1(\bar{\bu}_m|_{[0,s]})h_2(\bar{c}_m|_{[0,s]})h_3(\bar{n}_m|_{[0,s]}),\\
			&F(t,s):=\left(LN(t)-LN(s),\bv\right)_{H\times L^2(\bo)}h_1(\bu|_{[0,s]})h_2(c|_{[0,s]})h_3(n|_{[0,s]}).
		\end{align*} 
		We will prove that \begin{equation}
			\lim_{m\longrightarrow\infty}\be'F_m(t,s)=\be'F(t,s).\label{5.94}
		\end{equation}
		To this aim, we start by noting that  by the $\mathbb{P}'$-a.s.-convergence $(\bar{\bu}_{m},\bar{c}_{m},\bar{n}_{m})\to (\bu,c,n)$ in $\bz$ and Lemma \ref{lem4.8} as well as Lemma \ref{lem4.9},  we infer that  $$\lim_{m\longrightarrow\infty} F_m(t,s)=F(t,s),\qquad \mathbb{P}'\text{-a.s. } $$
		We will now show that the function $\{F_m(t,s)\}_{m\geq 1}$ are uniformly integrable.
		We use the estimate \re{5.91} to derive that
		\begin{align*}
			\be'\abs{F_m(t,s)}^4&\leq \bk \abs{h_1}^4_{L^\infty}\abs{h_2}^4_{L^\infty}\abs{h_3}^4_{L^\infty}\abs{\bv}^4_{H\times L^2}\be'\left[\abs{N_m(t)}_{L^2\times L^{2}}^4+\abs{N_m(s)}_{L^2\times L^{2}}^4\right]\\
			&\leq \bk \abs{h_1}^4_{L^\infty}\abs{h_2}^4_{L^\infty}\abs{h_3}^4_{L^\infty}\abs{\bv}^4_{H\times L^2}.
		\end{align*}
		Invoking the Vitali Theorem, we get  the convergence \re{5.94}.
		
Let  $0\leq s\leq t\leq T$ and $\bv_i=(\bw_i,\psi_i)\in H\times L^2(\bo)$, $i=1,2$. Let
		\begin{align*}
			Q_m(t,s):&=\left(\left(LN_m(t),\bv_1\right)_{H\times L^2(\bo)}\left(LN_m(t),\bv_2\right)_{H\times L^2(\bo)}\right.\\
			&\left.-\left(LN_m(s),\bv_1\right)_{H\times L^2(\bo)}\left(LN_m(s),\bv_2\right)_{H\times L^2(\bo)}\right)h_1(\bar{\bu}_m|_{[0,s]})h_2(\bar{c}_m|_{[0,s]})h_3(\bar{n}_m|_{[0,s]}),
		\end{align*} 
		\begin{align*}
			Q(t,s):&=\left(\left(LN(t),\bv_1\right)_{H\times L^2(\bo)}\left(LN(t),\bv_2\right)_{H\times L^2(\bo)}\right.\\
			&\left.-\left(LN(s),\bv_1\right)_{H\times L^2(\bo)}\left(LN(s),\bv_2\right)_{H\times L^2(\bo)}\right)h_1(\bu|_{[0,s]})h_2(c|_{[0,s]})h_3(n|_{[0,s]}).
		\end{align*} 
		Our purpose now is to prove that  
		\begin{equation}
			\be'Q(t,s)=\lim_{m\longrightarrow\infty}\be'Q_m(t,s),\label{5.92}
		\end{equation} 
imitating the proof before. Indeed, by $\mathbb{P}'$-a.s.-convergence $(\bar{\bu}_{m},\bar{c}_{m},\bar{n}_{m})\to (\bu,c,n)$ in $\bz$ and Lemma \ref{lem4.8} as well as Lemma \ref{lem4.9} once more, we obtain $$\lim_{m\longrightarrow\infty} Q_m(t,s)=Q(t,s),\qquad \mathbb{P}'\text{-a.s. } $$
We now prove the  uniform integrability of $Q_m(t,s)$.  For this purpose, by  \re{5.91}  we find that
		\begin{align*}
			\be'\abs{Q_m(t,s)}^2&\leq \bk \abs{h_1}^2_{L^\infty}\abs{h_2}^2_{L^\infty}\abs{h_3}^2_{L^\infty}\be'\left[\abs{\left(N_m(t),\bv_1\right)_{H\times L^2(\bo)}\left(N_m(t),\bv_2\right)_{H\times L^2(\bo)}}^2\right.\\
			&\qquad{}\left.+\abs{\left(N_m(s),\bv_1\right)_{H\times L^2(\bo)}\left(N_m(s),\bv_2\right)_{H\times L^2(\bo)}}^2\right]\\
			&\leq \bk \abs{h_1}^2_{L^\infty}\abs{h_2}^2_{L^\infty}\abs{h_3}^2_{L^\infty}\abs{\bv_1}^2_{H\times L^2}\abs{\bv_2}^2_{H\times L^2}\be'\left[\abs{N_m(t)}_{L^2\times L^{2}}^4+\abs{N_m(s)}_{L^2\times L^{2}}^4\right]\\
			&\leq \bk \abs{h_1}^2_{L^\infty}\abs{h_2}^2_{L^\infty}\abs{h_3}^2_{L^\infty}\abs{\bv}^2_{H\times L^2}.
		\end{align*}
		As before, the Vitali Theorem yields equality \re{5.92}.
		
		Finally, we also  define 
		\begin{align*}
			R_m(t,s):=&\left(\int_s^t\left(G_m(\bar{\bu}_m(r),\bar{c}_m(r))^*\bv_1,G_m	(\bar{\bu}_m(r),\bar{c}_m(r))^*\bv_2\right)_{\buc\times\mathbb{R}^2}dr\right)\times\\
			&\times h_1(\bar{\bu}_m|_{[0,s]})h_2(\bar{c}_m|_{[0,s]})h_3(\bar{n}_m|_{[0,s]}),
		\end{align*}
		and 
		\begin{equation*}
			R(t,s):=\left(\int_s^t\left(G(\bu(r),c(r))^*\bv_1,G	(\bu(r),c(r))^*\bv_2\right)_{\buc\times\mathbb{R}^2}dr\right)h_1(\bu|_{[0,s]})h_2(c|_{[0,s]})h_3(n|_{[0,s]}),
		\end{equation*}
		We claim that 
		\begin{equation}
			\lim_{m\longrightarrow\infty}\be'R_m(t,s)=\be'R(t,s).\label{4.93}
		\end{equation}
In order to establish this claim we first show that
		\begin{equation}
			\lim_{m\longrightarrow\infty}R_m(t,s)=R(t,s),\qquad\mathbb{P}'\text{-a.s}.\label{4.94*}
		\end{equation}
Since $h_1(\bar{\bu}_m|_{[0,s]})h_2(\bar{c}_m|_{[0,s]})h_3(\bar{n}_m|_{[0,s]})\to h_1(\bu|_{[0,s]})h_2(c|_{[0,s]})h_3(n|_{[0,s]})$ $\mathbb{P}$-a.s., in order to prove \re{4.94*},  it is  sufficient to prove that 
		\begin{align}
			&	\lim_{m\longrightarrow\infty}\int_s^t\left(G_m(\bar{\bu}_m(r),\bar{c}_m(r))^*\bv_1,G_m	(\bar{\bu}_m(r),\bar{c}_m(r))^*\bv_2\right)_{\buc\times\mathbb{R}^2}dr\notag\\
			&=\int_s^t\left(G(\bu(r),c(r))^*\bv_1,G	(\bu(r),c(r))^*\bv_2\right)_{\buc\times\mathbb{R}^2}dr,\qquad\mathbb{P}'\text{-a.s}.\label{5.95}
		\end{align}
		For all $r\in [s,t]$, we set
		\begin{align*}
			J(r):=&\left(G_m(\bar{\bu}_m(r),\bar{c}_m(r))^*\bv_1,G_m	(\bar{\bu}_m(r),\bar{c}_m(r))^*\bv_2\right)_{\buc\times\mathbb{R}^2}\notag\\
			&\qquad{}-\left(LG(\bu(r),c(r))^*\bv_1,LG	(\bu(r),c(r))^*\bv_2\right)_{\buc\times\mathbb{R}^2}.
		\end{align*}
		Then, we note that
		\begin{align}
			\int_s^t\abs{J(r)}dz
			&\leq \int_0^T\abs{\left(G_m(\bar{\bu}_m(r),\bar{c}_m(r))^*\bv_1-LG(\bu(r),c(r))^*\bv_1,G_m	(\bar{\bu}_m(r),\bar{c}_m(r))^*\bv_2\right)_{\buc\times\mathbb{R}^2}}dr\notag\\
			&\qquad{}+ \int_0^T\abs{\left(G	(\bu(r),c(r))^*\bv_1,G_m(\bar{\bu}_m(r),\bar{c}_m(r))^*\bv_2-G(\bu(r),c(r))^*\bv_2\right)_{\buc\times\mathbb{R}^2}}dr\label{5.96*}\\
			&=I_1(m)+I_2(m).\notag
		\end{align}
		Using the Cauchy-Schwarz inequality and the H\"older inequality, we derive that
		\begin{align*}
			I_1(m)&\leq \left(\int_0^T\abs{G_m(\bar{\bu}_m(r),\bar{c}_m(r))^*\bv_1-G(\bu(r),c(r))^*\bv_1)}_{\buc\times\mathbb{R}^2}^2dr\right)^{\frac{1}{2}}\times\\
			&\qquad\times\left(\int_0^T\abs{G_m	(\bar{\bu}_m(r),\bar{c}_m(r))^*\bv_2}_{\buc\times\mathbb{R}^2}^2dr\right)^{\frac{1}{2}}.
		\end{align*}
		Owing to the fact that $\bp_m^1g(\bar{\bu}_m,\bar{c}_m)e_k\in H$ and $\bp_m^2\phi(\bar{c}_m)g_k\in L^2(\bo)$,  we infer that $$(L^1\bp_m^1g(\bar{\bu}_m,\bar{c}_m)e_k,\bw_1):=\langle \bp_m^1g(\bar{\bu}_m,\bar{c}_m)e_k,i^{1*}\bw_1\rangle=(g(\bu,c)e_k,i^{1*}\bw_1).$$  and 
		$$(L^2\bp_m^2\phi(\bar{c}_m)g_k,\psi_2):=\langle \bp_m^2\phi(\bar{c}_m)g_k,i^{2*}\psi_2\rangle=(\bp_m^2\phi(\bar{c}_m)g_k,i^{2*}\psi_2).$$  
		Thus, using the inequality \re{2.5} and the fact that $\{e_k\}_{k\geq 1}$ and $\{g_k\}_{k= 1,2}$ are orthonormal basis of $\buc$ and $\mathbb{R}^2$ respectively, we derive that
		\begin{align}
			&\int_0^T\abs{G_m	(\bar{\bu}_m(r),\bar{c}_m(r))^*\bv_2}_{\buc\times\mathbb{R}^2}^2dr\notag\\
			&=\int_0^T\left(\abs{\sum_{k=1}^\infty(L^1\bp_m^1g(\bar{\bu}_m(r),\bar{c}_m(r))e_k,\bw_2)e_k}_\buc^2+\abs{\sum_{k=1}^2(L^2\bp_m^2\phi(\bar{c}_m(r))g_k,\psi_2)g_k}_{\mathbb{R}^2}^2\right)dr\notag\\
			&\leq \int_0^T\sum_{k=1}^\infty\abs{(\bp_m^1g(\bar{\bu}_m(r),\bar{c}_m(r))e_k,i^{1*}\bw_2)}^2dr+\int_0^T\sum_{k=1}^2\abs{(\bp_m^2\phi(\bar{c}_m(r))g_k,i^{2*}\psi_2)}^2dr\label{5.96}\\
			&\leq \abs{i^{1*}\bw_2}^2_{L^2}\int_0^T\abs{g(\bar{\bu}_m(r),\bar{c}_m(r))}^2_{\mathcal{L}^2(\buc,H)}dr+\abs{i^{2*}\psi_2}^2_{L^2}\int_0^T\abs{\phi(\bar{c}_m(r))}^2_{\mathcal{L}^2(\mathbb{R}^2,L^2)}dr\notag\\
			&\leq \bk\int_0^T(1+\abs{(\bar{\bu}_m(r),\bar{c}_m(r))}^2_{\bhc})dr+\bk\int_0^T\abs{\nabla\bar{c}_m(r))}^2_{L^2}dr\notag\\
			&\leq \bk, \qquad \mathbb{P}'\text{-a.s.}\notag
		\end{align}
		In the last line we used the fact that $\bar{c}_m\to c$ in $L^{2}(0,T;H^1(\bo))$  and $\bar{\bu}_m\to \bu$ in $L^{2}(0,T;H)$  $\mathbb{P}'$-a.s.
		
		On the other hand, we note that
		\begin{align}
			&\int_0^T\abs{G_m(\bar{\bu}_m(r),\bar{c}_m(r))^*\bv_1-G(\bu(r),c(r))^*\bv_1)}_{\buc\times\mathbb{R}^2}^2dr\notag\\
			&\leq \int_0^T\abs{\left[\sum_{k=1}^\infty(L^1\bp_m^1g(\bar{\bu}_m(r),\bar{c}_m(r))e_k,\bw_1)-\sum_{k=1}^\infty(L^1g(\bu(r),c(r))e_k,\bw_1)\right]e_k}_\buc^2dr\notag\\
			&\qquad{}+\int_0^T\abs{\left[\sum_{k=1}^2(L^2\bp_m^2\phi(\bar{c}_m(r))g_k,\psi_1)-\sum_{k=1}^2(L^2\phi(c(r))g_k,\psi_1)\right]g_k}_{\mathbb{R}^2}^2dr\notag.
		\end{align}
		Then  by this last inequality and the inequality \re{5.96}, we  infer that 
		\begin{align}
			I_1^2(m)&\leq \bk\int_0^T\abs{\sum_{k=1}^\infty(g(\bar{\bu}_m(r),\bar{c}_m(r))e_k,\bp_m^1i^{1*}\bw_1)-\sum_{k=1}^\infty(g(\bu(r),c(r))e_k,i^{1*}\bw_1)}^2dr\notag\\
			&\qquad{}+\bk\int_0^T\abs{\sum_{k=1}^2(\phi(\bar{c}_m(r))g_k,\bp_m^2i^{2*}\psi_1)-\sum_{k=1}^2(\phi(c(r))g_k,i^{2*}\psi_1)}^2dr\notag\\
			&\leq \bk\abs{i^{1*}\bw_1}^2_{L^2}\int_0^T\abs{g(\bar{\bu}_m(r),\bar{c}_m(r))-g(\bu(r),c(r))}^2_{\mathcal{L}^2(\buc,H)}dr\label{5.97}\\
			&\qquad{}+\bk\abs{\bp_m^1i^{1*}\bw_1-i^{1*}\bw_1}^2_{L^2}\int_0^T\abs{g(\bar{\bu}_m(r),\bar{c}_m(r))}^2_{\mathcal{L}^2(\buc,H)}dr\notag\\
			&\qquad{}+\abs{i^{2*}\psi_1}^2_{L^2}\int_0^T\abs{\phi(\bar{c}_m(r))-\phi(c(r))}^2_{\mathcal{L}^2(\mathbb{R}^2,L^2)}dr\notag\\
			&\qquad{}+\abs{\bp_m^2i^{2*}\psi_1-i^{2*}\psi_1}^2_{L^2}\int_0^T\abs{\phi(\bar{c}_m(r))}^2_{\mathcal{L}^2(\mathbb{R}^2,L^2)}dr\notag\\
			&:=II_1(m)+II_2(m)+II_3(m)+II_4(m).\notag
		\end{align}
		By means of the continuity of $g$, the $\mathbb{P}'$-a.s.-convergence $(\bar{\bu}_{m},\bar{c}_{m},\bar{n}_{m})\to (\bu,c,n)$ in $\bz$, the inequality \re{2.5} and the Vitali Theorem, we can derive that $\lim_{m\longrightarrow\infty} II_1(m)=0$. Furthermore, since 
		\begin{align}
			& \int_0^T\abs{g(\bar{\bu}_m(r),\bar{c}_m(r))}^2_{\mathcal{L}^2(\buc,H)}dr+\int_0^T\abs{\phi(\bar{c}_m(r))}^2_{\mathcal{L}^2(\mathbb{R}^2,L^2)}dr\notag\\
			&\leq \bk\int_0^T(1+\abs{(\bar{\bu}_m(r),\bar{c}_m(r))}^2_{\bhc})dr+\bk\int_0^T\abs{\nabla\bar{c}_m(r)}^2_{L^2}dr\notag\\
			&\leq \bk \qquad \mathbb{P}'\text{-a.s.},\notag
		\end{align}
		we deduce that  $$\lim_{m\longrightarrow\infty} II_2(m)=\lim_{m\longrightarrow\infty} II_4(m)=0.$$ 
		Now, we study the $II_3(m)$. We see that
		\begin{align}
			II_3(m)&\leq \abs{\psi_1}^2_{L^2}\gamma^2\abs{\sigma}^2_{L^\infty}\int_0^T\abs{\nabla\bar{c}_m(r)-\nabla c (r)}^2_{L^2}dr\notag\\
			&\leq\abs{\psi_1}^2_{L^2}\gamma^2\abs{\sigma}^2_{L^\infty}\int_0^T\abs{\bar{c}_m(r)-c (r)}^2_{H^1}dr.\notag
		\end{align}
By using the fact that $\bar{c}_m\to c$ in $L^{2}(0,T;H^1(\bo))$,  $\mathbb{P}'$-a.s.,  we can pass to the limit in this last inequality and infer that $\lim_{m\longrightarrow\infty} II_3(m)=0$. Hence passing to the limit in \re{5.97} we get  $\lim_{m\longrightarrow\infty} I_1(m)=0$. In a similar fashion, we can also prove that $\lim_{m\longrightarrow\infty} I_2(m)=0$. Therefore, passing to the limit in \re{5.96*}, we obtain the convergence \re{5.95} and completes the proof of  the almost surely convergence \re{4.94*}.  
		
		To finish the proof of equality \re{4.93}, it remains to prove the uniform integrability of $R_m(t,s)$. For this purpose, using the Young inequality, a similar calculations as in inequality \re{5.96} and the estimates \re{4.95} and \re{4.97},   we arrive at
		\begin{align}
			\be'\abs{R_m(t,s)}^2&\leq\prod_{i=1}^{3}\abs{h_i}^2_{L^\infty}\be'\left(\int_s^t\left(G_m(\bar{\bu}_m(r),\bar{c}_m(r))^*\bv_1,G_m	(\bar{\bu}_m(r),\bar{c}_m(r))^*\bv_2\right)_{\buc\times\mathbb{R}^2}dr\right)^2\notag\\
			&\leq \bk (t-s)\be'\int_s^t\abs{G_m	(\bar{\bu}_m(r),\bar{c}_m(r))^*\bv_1}_{\buc\times\mathbb{R}^2}^2\abs{G_m	(\bar{\bu}_m(r),\bar{c}_m(r))^*\bv_2}_{\buc\times\mathbb{R}^2}^2dr\notag\\
			&\leq \bk \be'\int_0^T\abs{G_m	(\bar{\bu}_m(r),\bar{c}_m(r))^*\bv_1}_{\buc\times\mathbb{R}^2}^4dr+\bk\be'\int_0^T\abs{G_m	(\bar{\bu}_m(r),\bar{c}_m(r))^*\bv_2}_{\buc\times\mathbb{R}^2}^4dr\notag\\
			&\leq \bk \be'\int_0^T\abs{g(\bar{\bu}_m(r),\bar{c}_m(r))}^4_{\mathcal{L}^2(\buc,H)}dr+\bk \be'\int_0^T\abs{\phi(\bar{c}_m(r))}^4_{\mathcal{L}^2(\mathbb{R}^2,L^2)}dr\notag\\
			&\leq \bk\be'\sup_{0\leq r\leq T}(1+\abs{(\bar{\bu}_m(r),\bar{c}_m(r))}^4_{\bhc})+\bk\be'\sup_{0\leq r\leq T}\abs{\nabla\bar{c}_m(r))}^4_{L^2}\notag\\
			&\leq \bk,\notag
		\end{align}
which prove the uniform integrability of $R_m(t,s)$.	Thus, invoking the Vitali Theorem, we obtain the convergence \re{4.93}. 

Taking into account the convergences \re{5.94}, \re{5.92} and \re{4.93}, we can pass to the limit in the equalities \re{5.91*} and \re{5.92*} to get
		\begin{equation}
			\be\left[\left(LN(t)-LN(s),\bv_1\right)_{H\times L^2(\bo)}h_1(\bu|_{[0,s]})h_2(c|_{[0,s]})h_3(n|_{[0,s]})\right]=0,\notag
		\end{equation}
		and 
		\begin{align}
			&\be\left[\left(\left(LN(t),\bv_1\right)_{H\times L^2(\bo)}\left(LN(t),\bv_2\right)_{H\times L^2(\bo)}-\left(LN(s),\bv_1\right)_{H\times L^2(\bo)}\left(LN(s),\bv_2\right)_{H\times L^2(\bo)}\right.\right.\notag\\
			&	\left.\left. -\int_0^t\left(G	(\bu(s),c(s))^*\bv_1,G	(\bu(s),c(s))^*\bv_2\right)_{\buc\times \mathbb{R}^2}ds \right)h_1(\bu|_{[0,s]})h_2(c|_{[0,s]})h_3(n|_{[0,s]})\right]=0,\notag
		\end{align}
		which complete the proof of  Lemma \ref{lem5.15}.
	\end{proof}
Thanks to  Lemma \ref{lem5.15}, we apply the usual martingale representation theorem proved in \cite[Theorem 8.2]{Da} to the process $LN$ and conclude that there exists a probability space $(\tilde{\Omega}, \tilde{\mathcal{F}}, \tilde{\mathbb{P}}) $, a filtration  $\tilde{\mathbb{F}}$ and a   $\buc\times\mathbb{R}^2$-cylindrical Wiener process $\bar{\mathbf{W}}_s:=(\bar{W}_s,\bar{\beta}_s)$  defined on the probability space $(\bar{\Omega},\bar{\mathcal{F}},\bar{\mathbb{P}})=(\Omega'\times\tilde{\Omega},\mathcal{F}'\otimes\tilde{\mathcal{F}},\mathbb{P}'\otimes\tilde{\mathbb{P}})$  adapted to the filtration  $\bar{\mathbb{F}}=\mathbb{F}'\otimes\tilde{\mathbb{F}}$ such that 
	\begin{equation*}
		LN(t,\omega',\tilde{\omega})=\int_0^tG	(\bu(s,\omega',\tilde{\omega}),c(s,\omega',\tilde{\omega}))d\bar{\mathbf{W}}_s(\omega',\tilde{\omega}), \qquad t\in [0,T],\qquad (\omega',\tilde{\omega})\in\bar{\Omega},
	\end{equation*}
	where 
	\begin{equation*}
		LN(t,\omega',\tilde{\omega})=LN(t,\omega'), \ (\bu(s,\omega',\tilde{\omega}),c(s,\omega',\tilde{\omega}))=(\bu(s,\omega'),c(s,\omega')),\ t\in [0,T],\ (\omega',\tilde{\omega})\in\bar{\Omega}.
	\end{equation*}
	This implies that in the probability space $(\bar{\Omega},\bar{\mathcal{F}},\bar{\mathbb{P}})$, for $t\in [0,T]$ and $\bar{\mathbb{P}}$-a.s.
	\begin{equation}\label{5.103}
		\begin{cases}
			L^1N^1(t)=\Int_0^tL^1g	(\bu(s),c(s))d\bar{W}_s, \text{  in  } H,\vspace{0.2cm}\\
			L^2N^2(t)=\Int_0^tL^2\phi	(c(s))d\bar{\beta}_s, \text{  in  } L^2(\bo).
		\end{cases}
	\end{equation}
	Thanks to \re{2.5} and  \re{4.104*}  the estimate 
	\begin{align}
		\bar{\be}\int_0^T\abs{g(\bu(s),c(s))}^2_{\mathcal{L}^2(\buc,V^*)}ds&\leq\bk \bar{\be}\int_0^T\abs{g(\bu(s),c(s))}^2_{\mathcal{L}^2(\buc,H)}ds\notag\\
		&\leq \bk\left(1+ \be'\sup_{0\leq s\leq T}\abs{(\bu(s),c(s))}_{\bhc}^2\right)<\infty,\notag
	\end{align}
	and 
	\begin{align}
		\bar{\be}\int_0^T\abs{\phi(c(s))}^2_{\mathcal{L}^2(\mathbb{R}^2,H^{-2})}ds&\leq\bk \bar{\be}\int_0^T\abs{\phi(c(s))}^2_{\mathcal{L}^2(\mathbb{R}^2,L^{2})}ds\notag\\
		&\leq \bk\left(1+\be'\sup_{0\leq s\leq T}\abs{c(s)}_{H^1}^2\right)<\infty,\notag
	\end{align}
	yield that $L^1N^1$ and $L^2N^2$ in \re{5.103} are continuous martingale in $H$ and $L^2(\bo)$ respectively. In a similar fashion as in \cite[Proof of Theorem 1.1]{Brze1}, using the continuity of the operators $L^1$ and $L^2$, we get
	\begin{equation*}
		\int_0^tL^1g	(\bu(s),c(s))d\bar{W}_s=L^1\left(\int_0^tg	(\bu(s),c(s))d\bar{W}_s\right) \text{ and } \int_0^tL^2\phi	(c(s))d\bar{\beta}_s=L^2\left(\int_0^t\phi	(c(s))d\bar{\beta}_s\right),
	\end{equation*}
	for all $t\in[0,T]$. Combining these two last inequalities with  the injectivity of the operators $L^1$ and $L^2$, we infer from the system \re{5.103} that for $t\in [0,T]$,
	\begin{equation}\label{5.104}
		\begin{cases}
			N^1(t)=\Int_0^tg	(\bu(s),c(s))d\bar{W}_s,\text{  in  } V^*,\vspace{0.2cm}\\
			N^2(t)=\Int_0^t\phi	(c(s))d\bar{\beta}_s,  \text{  in  } H^{-2}(\bo).
		\end{cases}
	\end{equation}  
	On the new probability space $(\bar{\Omega}, \bar{\mathcal{F}}, \bar{\mathbb{P}}) $, we also extend the random variable $n(t)$ by $$n(t,\omega',\tilde{\omega})=n(t,\omega'),\ t\in [0,T],\ (\omega',\tilde{\omega})\in\bar{\Omega},$$
	and infer that the equality \re{5.89} also hods in $(\bar{\Omega}, \bar{\mathcal{F}}, \bar{\mathbb{P}})$. Using this, the definition of $N^1$ and $N^2$, and the system \re{5.104}, we derive that $(\bar{\Omega}, \bar{\mathcal{F}},\bar{\mathbb{F}}, \bar{\mathbb{P}},(\bu,c,n),(\bar{W},\bar{\beta}))$ satisfies the system \re{4.1}. In particular, we have for all $t\in[0,T]$ and $\bar{\mathbb{P}}$-a.s.
	\begin{equation*}
		\begin{cases}
			\bu(t) =\bu_0-\Int_0^t[\eta A_0\bu(s)+B_0(\bu(s),\bu(s))+R_0(n(s),\varPhi)]ds+ \Int_0^tg(\bu(s),c(s)) d\bar{W}_s,\text{  in  } V^*,\vspace{0.2cm}\\
			c(t)=c_0- \Int_0^t[\xi A_1c(s)+B_1(\bu(s), c(s))- R_1(n(s),c(s))]ds+\gamma\int_0^t\phi(c(s))d\bar{\beta}_s,\text{  in  } H^{-2}(\bo),
		\end{cases}
	\end{equation*}
	which can be written as
	\begin{equation*}
		\begin{cases}
			\bu(t) =\bu_0-\Int_0^tG_0(s)ds + \Int_0^tS_0(s)d\bar{W}_s,\text{  in  } V^*,\vspace{0.2cm}\\
			c(t)=c_0- \Int_0^tG_1(s)ds + \Int_0^tS_1(s)d\bar{\beta}_s,\text{  in  } H^{-2}(\bo),
		\end{cases}
	\end{equation*}
where  for all $t\in [0,T]$, 
\begin{equation*}
\begin{split}
	&G_0(t):=\eta A_0\bu(t)+B_0(\bu(t),\bu(t))+R_0(n(t),\varPhi),\\
&G_1(t):=\xi A_1c(t)+B_1(\bu(t), c(t))- R_1(n(t),c(t)),\\
&S_0(t):= g(\bu(t),c(t)),\quad\text{and}\quad S_1(t):=\gamma\phi(c(t)).
\end{split}
\end{equation*}
	Since the identities \re{4.100},  \re{4.104*}  and \re{4.108} hold, following the idea of the proof of estimate \re{3.78}, we can see that $ G_0 \in L^2( [0, T]\times\bar{ \Omega}; V^*)$, $ G_1 \in L^2( [0, T]\times \bar{\Omega}; L^2(\bo))$, $S_0 \in L^2([0, T]\times\bar{\Omega}; H)$ and $S_1 \in L^2([0, T]\times\bar{\Omega}; H^1(\bo))$. Therefore, it follows from \cite[Theorem 3.2]{Kry} that there exists $\bar{\Omega}_0\in\bar{\mathcal{F}}$ such that $\bar{\mathbb{P}}(\bar{\Omega}_0) = 1$ and for all $\omega\in\bar{\Omega}_0$, the function $\bu$ and $c$ take values in $H$ and in $H^1(\bo)$ respectively and  are continuous in $H$ and $H^1(\bo)$ with respect to $t$.
	Owing to the fact that $(\bu,c,n)$ is $\bz_\bu\times\bz_c\times\bz_n$-valued random variable and progressively measurable over the filtration $\bar{\mathbb{F}}$, we derive that $(\bar{\Omega}, \bar{\mathcal{F}},\bar{\mathbb{F}}, \bar{\mathbb{P}},(\bu,c,n),(\bar{W},\bar{\beta}))$  is a  probabilistic weak solution of system \re{1.1}. We recall that the inequalities \re{4.41}  directly follows from  relations  \re{4.100}, \re{4.104*}, and \re{4.108}.

\section{Properties of solution and energy inequality}
In  this section we prove the  mass conservation property, the non-negativity property and the $L^\infty$-norm stability for the prrobabilistic strong solution of system \re{1.1}. By these properties, we also prove an energy inequality which may be useful for the study of the invariant measure of system (\ref{1.1}) which is still an opened problem according to our knowledge.

\subsection{Non-negativity and mass conservation}
The following theorem gives the conservation of the total mass property  and the non-negativity of the strong solutions of system (\ref{1.1}).
\begin{theorem} \label{lem2.2}
	Let  $\mathfrak{A}=(\Omega,\mathcal{F},\{\mathcal{F}_t\}_{t\in[0,T]},\mathbb{P})$ be a filtered probability space,  $\mathcal{U}$ be a separable Hilbert space,  $W$ be
	cylindrical Wiener process on $\mathcal{U}$ over $\mathfrak{A}$, 	  and $\beta=(\beta^1,\beta^2)$ be a two dimensional standard Brownian motion
	over $\mathfrak{A}$ independent of $W$. If $(\bu,c,n)$ is a  probabilistic strong solution of system (\ref{1.1}), then the following equality holds for all $t\in [0,T]$
	\begin{equation}
		\int_\mathcal{O}n(t,x)dx=\int_\mathcal{O}n_0(x)dx, \ \mathbb{P}\text{-a.s}.\label{2.14}
	\end{equation}
	Furthermore, if  $c_0>0$ and $n_0> 0$, then the following inequality hold $\mathbb{P}$-a.s
	\begin{equation}
		n(t)>0,\text{ and } c(t)> 0, \text{   for all  } t\in [0,T].\label{2.15}
	\end{equation}
\end{theorem}
\begin{proof} 
We	Note that, the conservation of the total mass (\ref{2.14}) follows straightforwardly
	from  the fact that  $\nabla\cdot\bu=0$ and the proof of \re{2.15} is very similar to the proof of Lemma \ref{lem3.6}.
\end{proof}
The following theorem gives the $L^\infty$-stability of the probabilistic  strong solution of system \re{1.1}.
\begin{theorem}\label{lem2.3}
	Let $\mathfrak{A}=(\Omega,\mathcal{F},\{\mathcal{F}_t\}_{t\in[0,T]},\mathbb{P})$ be a filtered probability space,  $\mathcal{U}$ be a separable Hilbert space,  $W$ be
	cylindrical Wiener process on $\mathcal{U}$ over $\mathfrak{A}$, 	  and $\beta=(\beta^1,\beta^2)$ be a two dimensional standard Brownian motion
	over $\mathfrak{A}$ independent of $W$. If $(\bu,c,n)$ is a  probabilistic strong solution of system (\ref{1.1}) in  the filtered probability space $\mathfrak{A}$,   then   for all $t\in [0,T]$
	\begin{equation}
		\abs{c(t)}_{L^\infty}\leq\abs{c_0}_{L^\infty},\quad\mathbb{P}\text{-a.s.}\label{2.17}
	\end{equation}
\end{theorem}
\begin{proof}
	The  proof is similar to the proof of Corollary \ref{lem3.7}. 
\end{proof}
 \subsection{Energy inequality}
 In this subsection, we will derive  an energy inequality. The probabilistic strong solution $(\bu,n,c )$ involving the following Lyapunov functional 
 \begin{equation*}
 	\bec (n,c,\bu)(t)=\int_\bo n(t)\ln n(t)dx +\bk_f\abs{\nabla c(t)}^2_{L^2}	+\frac{8\bk_f\bk_{GN}\abs{c_0}^2_{L^\infty}}{3\xi\eta}\abs{\bu(t)}^2_{L^2}+e^{-1}\abs{\bo},\quad t\in [0,T],
 \end{equation*}
where $\bk_{GN}$ is a  constant given by the Gagliardo-Niremberg inequality \re{4.4.} and $\bk_f$ is defined in \eqref{Eq:K-f}.
\begin{proposition}\label{proposition3.1}
Suppose that Assumption 1, Assumption 2 and the following inequality
\begin{equation}
	\frac{4\bk_f\Max_{0\leq c\leq \abs{c_0}_{L^\infty}}f^2}{\Min_{0\leq c\leq \abs{c_0}_{L^\infty}}f'} \leq \delta,\label{3.35}
\end{equation}
are satisfied.  Let $\mathfrak{A}=(\Omega,\mathcal{F},\{\mathcal{F}_t\}_{t\in[0,T]},\mathbb{P})$ be a filtered probability space,  $\mathcal{U}$ be a separable Hilbert space,  $W$ be
cylindrical Wiener process on $\mathcal{U}$ over $\mathfrak{A}$, 	  and $\beta=(\beta^1,\beta^2)$ be a two dimensional standard Brownian motion
over $\mathfrak{A}$ independent of $W$. Then, any probabilistic strong solution $(\bu,c,n)$  of system (\ref{1.1}) in the filtered probability space $\mathfrak{A}$ satisfies the following entropy functional relations for almost all $t \in[0,T]$,
\begin{align}
	\abs{c(t)}^2_{L^2}+2\eta\int_0^t\abs{\nabla c(s)}^2_{L^2}ds+2\int_0^t( n(s)f(c(s)),c(s))ds=\abs{c_0}^2_{L^2},\label{3.30}
\end{align}
\begin{align}
	\bec (n,c,\bu)(t)&+\int_0^t\left[\delta\abs{\nabla\sqrt{n(s)}}^2_{L^2}
	+\frac{3\xi\bk_f}{2}\abs{\Delta c(s)}^2_{L^2}+\frac{8\bk_f\bk_{GN}\abs{c_0}^2_{L^\infty}}{3\xi}\abs{\nabla \bu(s)}^2_{L^2}+\abs{\sqrt{n(s)}\nabla c(s)}_{L^2}^2\right]ds\notag\\
	&\leq \bec (n_0,c_0,\bu_0)+\bk_5t+\bk_6\int_0^t\abs{\bu(s)}^2_{L^2}ds+\gamma^2\bk_f\int_0^t\abs{\nabla\phi(c(s))}^2_{\mathcal{L}^2(\mathbb{R}^2;L^2)}ds\notag\\
	&\qquad{}+\frac{8\bk_f\bk_{GN}\abs{c_0}^2_{L^\infty}}{3\xi\eta}\int_0^t\abs{g(\bu(s),c(s))}^2_{\mathcal{L}^2(\buc;H)}ds+2\gamma\bk_f\int_0^t(\nabla\phi(c(s)),\nabla c(s))d\beta_s\label{3.21}\\
		&\qquad{}+\frac{16\bk_f\bk_{GN}\abs{c_0}^2_{L^\infty}}{3\xi\eta}\int_0^t(g(\bu(s),c(s)),\bu(s))dWs,\notag
\end{align}
$\mathbb{P}$-a.s.,
where  $\bk_5$ and $\bk_6$ are some positive constant to be given later.
\end{proposition}
\begin{proof}
The equality \re{3.30} follows directly from the  application of  the It\^o formula to $t\mapsto\abs{c(t)}^2_{L^2}$ and the fact that 
\begin{equation*}
	(B_1(\bu,c),c)=\frac{1}{2}\int_\bo\bu(x)\cdot\nabla c^2(x)dx=-\frac{1}{2}\int_\bo c^2(x)\nabla\cdot\bu(x) dx=0,
\end{equation*}
as well as
\begin{equation*}
(\phi(c),c)=\sum_{k=1}^2\int_\bo\sigma_k(x)\cdot\nabla c(x)c(x)dx
=\frac{1}{2}\sum_{k=1}^2\int_\bo\sigma_k(x)\cdot\nabla c^2(x)dx=0
\end{equation*}
and 
\begin{equation*}
	\abs{\phi(c)}^2_{\mathcal{L}^2(\mathbb{R}^2;L^2)}=\abs{\nabla c}^2_{L^2}.
\end{equation*}
Next, we multiply  equation \re{3.1}$_3$ by $1+\ln n(s)$ for $s\in [0,t]$ and integrate the resulting equation in $\bo$ to obtain
\begin{equation}\label{3.27**}
\frac{d}{dt}\int_\bo n(s,x)\ln n(s,x)dx+\delta\int_\bo\dfrac{\abs{\nabla n(s,x)}^2}{n(s,x)}dx=\chi\int_\bo \nabla n(s,x)\cdot\nabla c(s,x)dx.
\end{equation}
Thanks to  the Young  inequality and the Cauchy-Schwartz inequality we note that
\begin{equation*}
	\chi\int_\bo \nabla n(x)\cdot \nabla c(x)dx\leq2\delta\int_\bo\abs{\nabla\sqrt{n(x)}}^2dx+\frac{\chi^2}{2\delta}\int_\bo n(x)\abs{\nabla c(x)}^2dx.
\end{equation*}
Combining the last inequality with equality \re{3.27**} we arrive at
\begin{align}\label{3.27}
	\int_\bo n(t,x)\ln n(t,x)dx+2\delta\int_0^t\abs{\nabla\sqrt{n(s)}}^2_{L^2}ds&\leq	\int_\bo n_0(x)\ln n_0(x)dx\notag\\
	&\qquad{}+\frac{\chi^2}{2\delta} \int_0^t\abs{\sqrt{n(s)}\nabla c(s)}_{L^2}^2ds.
\end{align}
By applying  the It\^o formula  to $t\mapsto\abs{\nabla c(t)}^2_{L^2}$, we find that
\begin{align}
	\abs{\nabla c(t)}^2_{L^2}+2\xi\int_0^t\abs{\Delta c(s)}^2_{L^2}ds&=\abs{\nabla c_0}^2_{L^2}-2\int_0^t(\nabla B_1(\bu(s),c(s)),\nabla c(s))ds\notag\\
	&\qquad{}-2\int_0^t (\nabla R_2(n(s),c(s)),\nabla c(s))ds\notag\\
	&\qquad{}+\gamma^2\int_0^t\abs{\nabla\phi(c(s))}^2_{\mathcal{L}^2(\mathbb{R}^2;L^2)}+2\gamma\int_0^t(\nabla\phi(c(s)),\nabla c(s))d\beta_s.\label{3.31}
\end{align}
Due to the Assumption $1$ and the $L^\infty$-norm stability obtained in Theorem \ref{lem2.3},  we obtain
\begin{align}
	(\nabla B_1(\bu,c),\nabla c)&\leq \abs{\nabla \bu}_{L^2}\abs{\nabla c}^2                                                                                                     _{L^4}\notag\\
	&\leq \frac{3\xi}{16\bk_{GN}\abs{c_0}^2_{L^\infty}}\abs{\nabla c}^4                                                                                                     _{L^4}+\frac{4\bk_{GN}\abs{c_0}^2_{L^\infty}}{3\xi} \abs{\nabla \bu}^2_{L^2}\notag\\
	&\leq \frac{\xi}{4}\abs{\Delta c}_{L^2}^2+\frac{4\bk_{GN}\abs{c_0}^2_{L^\infty}}{3\xi} \abs{\nabla \bu}^2_{L^2}+\frac{\xi(4\bk_2+3)}{16}\abs{c_0}_{L^\infty}^{2}.\notag
\end{align}
 and
\begin{align}
- (\nabla R_2(n,c),\nabla c(s))ds&\leq -\frac{\Min_{0\leq c\leq \abs{c_0}_{L^\infty}}f'(c)}{2}\int_\bo n(x)\abs{\nabla c(x)}^2dx\notag\\
&\qquad{}+\frac{1}{2\Min_{0\leq c\leq \abs{c_0}_{L^\infty}}f'}\int_\bo f^2(c(x))\frac{\abs{\nabla n(x)}^2}{n(x)}dx\notag\\
&\leq -\frac{\Min_{0\leq c\leq \abs{c_0}_{L^\infty}}f'(c)}{2}\abs{\sqrt{n}\nabla c}_{L^2}^2+\frac{2\Max_{0\leq c\leq \abs{c_0}_{L^\infty}}f^2}{\Min_{0\leq c\leq \abs{c_0}_{L^\infty}}f'(c)} \abs{\nabla \sqrt{n}}_{L^2}^2.\notag
\end{align}
Thus, we see from \re{3.31} that
\begin{align}
\abs{\nabla c(t)}^2_{L^2}	&+\frac{3\xi}{2}\int_0^t\abs{\Delta c(s)}^2_{L^2}ds+\Min_{0\leq c\leq \abs{c_0}_{L^\infty}}f'\int_0^t\abs{\sqrt{(s)}\nabla c(s)}_{L^2}^2ds\notag\\
	&\leq\abs{\nabla c_0}^2_{L^2}+\frac{\xi(4\bk_2+3)}{8}\abs{c_0}_{L^\infty}^{2}t+\frac{8\bk_{GN}\abs{c_0}^2_{L^\infty}}{3\xi} \int_0^t\abs{\nabla \bu(s)}^2_{L^2}ds\notag\\
	&\qquad{}+\frac{4\Max_{0\leq c\leq \abs{c_0}_{L^\infty}}f^2}{\Min_{0\leq c\leq \abs{c_0}_{L^\infty}}f'} \int_0^t\abs{\nabla \sqrt{n(s)}}_{L^2}^2ds\notag\\
	&\qquad{}+\gamma^2\int_0^t\abs{\nabla\phi(c(s))}^2_{\mathcal{L}^2(\mathbb{R}^2;L^2)}ds+2\gamma\int_0^t(\nabla\phi(c(s)),\nabla c(s))d\beta_s.\notag
\end{align}
Now, we multiply this last inequality by $\bk_f$, add the  result with  inequality \re{3.27}, and use the inequality  \re{3.35} to obtain
\begin{align}
	\int_\bo n(t,x)\ln n(t,x)dx& +\bk_f\abs{\nabla c(t)}^2_{L^2}	+\frac{3\xi\bk_f}{2}\int_0^t\abs{\Delta c(s)}^2_{L^2}ds\notag\\
	&\qquad{}+2\delta\int_0^t\abs{\nabla\sqrt{n(s)}}^2_{L^2}ds+\int_0^t\abs{\sqrt{n(s)}\nabla c(s)}_{L^2}^2ds\notag\\
	&\leq\bk_f\abs{\nabla c_0}^2_{L^2}+\int_\bo n_0(x)\ln n_0(x)dx+\frac{\bk_f\xi(4\bk_f\bk_2+3)}{8}\abs{c_0}_{L^\infty}^{2}t\notag\\
	&\qquad{}+\frac{8\bk_f\bk_{GN}\abs{c_0}^2_{L^\infty}}{3\xi} \int_0^t\abs{\nabla \bu(s)}^2_{L^2}ds+\gamma^2\bk_f\int_0^t\abs{\nabla\phi(c(s))}^2_{\mathcal{L}^2(\mathbb{R}^2;L^2)}ds\label{3.36}\\
	&\qquad{}+2\gamma\bk_f\int_0^t(\nabla\phi(c(s)),\nabla c(s))d\beta_s.\notag
\end{align}
Using the equality \re{2.14} and the inequality \re{4.4.} we note that
\begin{align}
\abs{n}_{L^2}&\leq \bk_{GN}\left( \abs{\sqrt{n}}_{L^2}\abs{\nabla\sqrt{n}}_{L^2}+ \abs{\sqrt{n}}^2_{L^2}\right)\notag\\
&\leq \bk_{GN}\left( \abs{n_0}^{\frac{1}{2}}_{L^1}\abs{\nabla\sqrt{n}}_{L^2}+ \abs{n_0}_{L^1}\right),\label{3.37}
\end{align}
which altogether with  the It\^o formula to $t\mapsto\abs{\bu(t)}^2_{L^2}$ implies  the existence of $\bk_3>0$ such that
\begin{align}
\abs{\bu(t)}^2_{L^2}+2\eta\int_0^t\abs{\nabla \bu(s)}^2_{L^2}ds&\leq 2\int_0^t\abs{\nabla\Phi}_{L^\infty}\abs{n(s)}_{L^2}\abs{\bu(s)}_{L^2}ds\notag\\
&\qquad{}+\int_0^t\abs{g(\bu(s),c(s))}^2_{\mathcal{L}^2(\buc;H)}ds+2\int_0^t(g(\bu(s),c(s)),\bu(s))dWs\notag\\
&\leq\abs{\bu_0}^2_{L^2}+ \frac{\delta\eta}{\bk_4}\int_0^t\abs{\nabla\sqrt{n(s)}}_{L^2}^2ds+\bk_3\abs{\nabla\Phi}_{L^\infty}^2\abs{n_0}_{L^1}\int_0^t\abs{\bu(s)}^2_{L^2}ds\label{3.38}\\
&\qquad{}+ \frac{1}{2}t+\frac{1}{2}\abs{\nabla\Phi}_{L^\infty}^2\abs{n_0}_{L^1}^2\int_0^t\abs{\bu(s)}^2_{L^2}ds\notag\\
&\qquad{}+\int_0^t\abs{g(\bu(s),c(s))}^2_{\mathcal{L}^2(\buc;H)}ds+2\int_0^t(g(\bu(s),c(s)),\bu(s))dWs,\notag
\end{align}
with $\bk_4=\frac{8\bk_f\bk_{GN}\abs{c_0}^2_{L^\infty}}{3\xi}$. Multiplying the inequality \re{3.38} by $\frac{\bk_4}{\eta}$, and adding the result with inequality \re{3.36}, we obtain some positive constants $\bk_5$ and $\bk_6$ such that the inequality \re{3.21} holds.
\end{proof}

\appendix
\section{Compactness and tightness criteria}
In this appendix we recall several compactness and tightness criteria that are frequently used in this paper.
%
%

We start with the following lemma based on the Dubinsky Theorem.
\begin{lemma}
	Let us consider the space
	\begin{equation}
	\tilde{\mathcal{Z}}_{0}=L_{w}^{2}(0,T;H^1(\bo))\cap L^{2}(0,T;L^2(\bo))\cap\mathcal{C}([0,T];H^{-3}(\bo))
	\end{equation}
	and $\tilde{\mathcal{T}}_{0}$ be the supremum of the corresponding topologies. Then a set $\bar{\bar{K}}_{0}\subset\tilde{\mathcal{Z}}_{0}$ is $\tilde{\mathcal{T}}_{0}$-relatively compact if the following three conditions hold
	\begin{itemize}
		\item[(a)]
		$\sup\limits_{\varphi\in\bar{\bar{K}}_{0}}\Int\limits_{0}^{T}|\varphi(s)|_{H^1}^{2}ds<\infty$, i.e., $\bar{\bar{K}}_{0}$ is bounded in $L^{2}(0,T;H^1(\bo))$,
		\item[(b)]$\exists \gamma>0$: $\sup\limits_{\varphi\in\bar{\bar{K}}_{0}}\abs{\varphi}_{C^\gamma([0,T]; H^{-3})}<\infty$.
	\end{itemize}
\end{lemma}
\begin{proof}
	We note that the following embedding is continuous $H^1(\bo)\hookrightarrow L^2(\bo)\hookrightarrow H^{-3}(\bo)$ with $H^1(\bo)\hookrightarrow L^2(\bo)$ compact. By the Banach-Alaoglu
	Theorem condition (a) yields that $\bar{\bar{K}}_{0}$ is compact in $L_{w}^{2}(0,T;H^1(\bo))$. Moreover (b) implies that the functions $\varphi \in \bar{\bar{K}}_{0}$ are equicontinuous, i.e. for all $\eps>0$, there exists $\delta>0$ such that if $ \abs{t-s}<\delta $ then $\abs{\varphi(t)-\varphi(s)}_{H^{-3}}<\varepsilon$ for all $\varphi\in\bar{\bar{K}}_{0}$. We can then apply Dubinsky's Theorem (see \cite[Theorem IV.4.1]{Vis}) since by condition (a), $\bar{\bar{K}}_{0}$ is bounded in $L^2(0,T; H^{1}(\bo))$.
\end{proof}
Following the same method as in \cite[Lemma 3.3 ]{Brz2}, we obtain the following compactness result.

\begin{lemma}
	Let us consider the space
	\begin{equation}
	\tilde{\mathcal{Z}}_{n}=L_{w}^{2}(0,T;H^1(\bo))\cap L^{2}(0,T;L^2(\bo))\cap\mathcal{C}([0,T];H^{-3}(\bo))\cap\mathcal{C}([0,T];L^2_{w}(\bo)),
	\end{equation}
	and $\tilde{\mathcal{T}}_{0}$ be the supremum of the corresponding topologies. Then a set $\bar{\bar{K}}_{0}\subset\tilde{\mathcal{Z}}_{n}$ is $\tilde{\mathcal{T}}_{0}$-relatively compact if the following three conditions hold
	\begin{itemize}
		\item[(a)]$\sup\limits_{\varphi\in\bar{\bar{K}}_{0}}\abs{\varphi}_{L^\infty(0,T;L^2)}<\infty$,
		\item[(b)]
		$\sup\limits_{\varphi\in\bar{\bar{K}}_{0}}\Int\limits_{0}^{T}|\varphi(s)|_{H^1}^{2}ds<\infty$, i.e., $\bar{\bar{K}}_{0}$ is bounded in $L^{2}(0,T;H^1(\bo))$,
		\item[(c)]$\exists \gamma>0$: $\sup\limits_{\varphi\in\bar{\bar{K}}_{0}}\abs{\varphi}_{C^\gamma([0,T]; H^{-3})}<\infty$.
	\end{itemize}
\end{lemma}
From this lemma we also get the following  tightness criteria for stochastic processes with paths in $\tilde{\mathcal{Z}}_{n}$ where the proof is the same as the proof of \cite[Lemma 5.5]{Brz3*}.
\begin{lemma}[Tightness criterion for $n$]\label{lemma 3.3}
	Let $\gamma>0$ be a given parameters and $(\varphi_n)_n$ be a sequence of continuous $\{\mathcal{F}_t\}_{t\in [0,T]}$-adapted $H^{-3}(\bo)$-valued processes. Let $\bl_m$ be the law of $\varphi_n$ on $\tilde{\mathcal{Z}}_{n}$. If for any $\eps>0$ there exists a constant $\bk_i$, $i=1,...,3$ such that
	\begin{equation*}
	\begin{split}
	&\sup_m\mathbb{P} \left( \abs{\varphi_m}_{L^\infty(0,T;L^2)}>K_1\right)\leq \eps, \\
	&\sup_m\mathbb{P} \left( \abs{\varphi_m}_{L^2(0,T;H^1)}>K_2\right)\leq \eps,\\
	&\sup_m\mathbb{P} \left( \abs{\varphi_m}_{C^\gamma(0,T;H^{-3})}>K_3\right)\leq \eps,
	\end{split}
	\end{equation*}
	then the sequence $(\bl_m)_m$ is tight on $\tilde{\mathcal{Z}}_{n}$.
\end{lemma}

The following compactness results are due to \cite[Theorem 4.4 and Theorem 4.5]{Brz1} (see also \cite{mot}), where we can see the details of the proof.

\begin{lemma}
	\label{lem3.3}
	Let us consider the space
	\begin{equation}
	\tilde{\mathcal{Z}}_{\bu}=L_{w}^{2}(0,T;V)\cap L^{2}(0,T;H)\cap\mathcal{C}([0,T];V^*)\cap\mathcal{C}([0,T];H_{w}),
	\end{equation}
	and $\tilde{\mathcal{T}}_{1}$ be the supremum of the corresponding topologies. Then a set $\bar{\bar{K}}_{1}\subset\tilde{\mathcal{Z}}_{\bu}$ is $\tilde{\mathcal{T}}_{1}$-relatively compact if the following three conditions hold
	\begin{itemize}
		\item[(a)]
		$\sup\limits_{\bv\in\bar{\bar{K}}_{1}}\sup\limits_{t\in[0,T]}|\bv(t)|_{L^2}<\infty$,
		\item[(b)]
		$\sup\limits_{\bv\in\bar{\bar{K}}_{1}}\Int\limits_{0}^{T}|\nabla\bv(s)|_{L^2}^{2}ds<\infty$, i.e., $\bar{\bar{K}}_{2}$ is bounded in $L^{2}(0,T;V)$,
		\item[(c)]$\lim\limits_{\delta\to 0}\sup\limits_{\bv\in\bar{\bar{K}}_{1}}\sup\limits_{s,t\in[0,T], |t-s|\leq\delta} \abs{\bv(t)-\bv(s)}_{V^*}=0$.
	\end{itemize}
\end{lemma}
\begin{lemma}
	\label{lem3.5}
	Let us consider the space
	\begin{equation}
	\tilde{\mathcal{Z}}_{c}=L_{w}^{2}(0,T;H^2(\bo))\cap L^{2}(0,T;H_w^1(\bo))\cap\mathcal{C}([0,T];L^2(\bo))\cap\mathcal{C}([0,T];H_w^1(\bo)),
	\end{equation}
	and $\tilde{\mathcal{T}}_{2}$ be the supremum of the corresponding topologies. Then a set $\bar{\bar{K}}_{2}\subset\tilde{\mathcal{Z}}_{c}$ is $\tilde{\mathcal{T}}_{2}$-relatively compact if the following three conditions hold
	\begin{itemize}
		\item[(a)]
		$\sup\limits_{\varphi\in\bar{\bar{K}}_{2}}\sup\limits_{t\in[0,T]}|\varphi(t)|_{H^1}<\infty$,
		\item[(b)]
		$\sup\limits_{\varphi\in\bar{\bar{K}}_{2}}\Int\limits_{0}^{T}|\varphi(s)|_{H^2}^{2}ds<\infty$, i.e., $\bar{\bar{K}}_{2}$ is bounded in $L^{2}(0,T;H^2(\bo))$,
		\item[(c)]$\lim\limits_{\delta\to 0}\sup\limits_{\varphi\in\bar{\bar{K}}_{2}}\sup\limits_{s,t\in[0,T], |t-s|\leq\delta}\abs{\varphi(t)-\varphi(s)}_{L^2}=0$.
	\end{itemize}
\end{lemma}
We now consider a filtered probability space $(\Omega,\mathcal{F},\mathbb{P})$ with filtration $\mathbb{F}:=\left\{\mathcal{F}_{t}\right\}_{t\geq 0}$ satisfying the usual hypotheses. Let $(\mathbb{M},d_{1})$ be a complete, separable metric space and
$(y_{n})_{n\in\mathbb{N}}$ be a sequence of $\mathbb{F}$-adapted and $\mathbb{M}$-valued processes. We recall from \cite{Joffe} the following definition.
\begin{definition}\label{definition3.1}
	A sequence $(y_{n})_{n\in\mathbb{N}}$ satisfies the \textbf{Aldous condition} in the space $\mathbb{M}$ if and only if
	\begin{align*}
	&\forall\epsilon>0~\forall\zeta>0~ \exists\delta>0~\text{such that for every sequence}~(\tau_{n})_{n\in\mathbb{N}}~\text{of}~ \mathbb{F}\text{-}\text{stopping times with}\\
	&\tau_{n}\leq T~\text{one has}~ \sup_{n\in\mathbb{N}}\sup_{0\leq\theta\leq \delta}\mathbb{P}\left\{|y_{n}(\tau_{n}+\theta)-y_{n}(\tau_{n})|_{\mathbb{M}}\geq\zeta\right\}\leq\epsilon.
	\end{align*}
\end{definition}
In Definition \ref{definition3.1}, and throughout we understand that $y_{n}$ is extended to zero outside the interval $[0,T]$.

The following lemma is proved in \cite[Appendix A, Lemma 6.3]{mot}.
\begin{lemma}
	\label{lem3.2}
	Let $(X,|.|_{X})$ be a separable Banach space and let $(y_{n})_{n\in\mathbb{N}}$ be a sequence of $X$-valued random variables. Assume that for every $(\tau_{n})_{n\in\mathbb{N}}$ of $\mathbb{F}$-stoppings times with $\tau_{n}\leq T$ and for every $n\in\mathbb{N}$ and $\theta\geq 0$ the following condition holds
	\begin{equation}
	\label{equa3.10}
	\mathbb{E}\left|y_{n}(\tau_{n}+\theta)-y_{n}(\tau_{n})\right|_{X}^{\alpha}\leq C\theta^{\beta},
	\end{equation}
	for some $\alpha,\beta>0$ and some constant $C>0$. Then the sequence $(y_{n})_{n\in\mathbb{N}}$ satisfies the \textbf{Aldous condition} in the space $X$.
\end{lemma}
In the view of Lemma \ref{lem3.3} and Lemma \ref{lem3.4}, in the next corollaries, we will state a tightness criteria for stochastic processes with part in $\tilde{\mathcal{Z}}_{\bu}$ or in $\tilde{\mathcal{Z}}_{c}$.
\begin{cor}\label{cor3.1}
	Let $(\bv_m)_m$ be a sequence of continuous $\{\mathcal{F}_t\}_{t\in [0,T]}$-adapted $V^*$-valued processes satisfying
	\begin{description}
		\item[(a)] there exists a constant $\bk_1>0$ such that $$ \sup_m\be\sup_{0\leq s\leq T}\abs{\bv_m(s)}_{L^2}^2\leq \bk_1,$$
		\item[(b)] there exists a constant $\bk_2>0$ such that $$ \sup_m\int_0^T\abs{\nabla\bv_m(s)}_{L^2}^2ds\leq \bk_2,$$
		\item[(c)]  $(\bv_m)_m$ satisfies the \textbf{Aldous condition} in $V^*$.
	\end{description}
	Let $\bl_m(\bv_m)$ be the law of $\bv_m$ on $\tilde{\mathcal{Z}}_{\bu}$.  Then, the sequence  $(\bl_m(\bv_m))_m$  is tight in $\tilde{\mathcal{Z}}_{\bu}$.
\end{cor}
\begin{cor}\label{cor3.2}
	$(v_m)_m$ be a sequence of continuous $\{\mathcal{F}_t\}_{t\in [0,T]}$-adapted $L^2(\bo)$-valued processes satisfying
	\begin{description}
		\item[(a)] there exists a constant $\bk_1>0$ such that $$ \sup_m\be\sup_{0\leq s\leq T}\abs{v_m(s)}_{H^1}^2\leq \bk_1,$$
		\item[(b)] there exists a constant $\bk_2>0$ such that $$ \sup_m\int_0^T\abs{v_m(s)}_{H^2}^2ds\leq \bk_2,$$
		\item[(c)]  $(v_m)_m$ satisfies the \textbf{Aldous condition} in $L^2(\bo)$.
	\end{description}
	Let $\bl_m(v_m)$ be the law of $v_m$ on $\tilde{\mathcal{Z}}_{c}$.  Then, the sequence  $(\bl_m(v_m))_m$  is tight in $\tilde{\mathcal{Z}}_{c}$.
\end{cor}

\section*{Acknowledgment}
We acknowledge financial support provided by the Austrian Science Fund (FWF). In particular, Boris Jidjou Moghomye  and partially Erika Hausenblas   were supported by the Austrian Science Fund, project 32295.

\begin{center}
  
\end{center}
\end{document}